\documentclass[11pt]{article}
\usepackage{latexsym,amsmath,amssymb,amsthm,amsfonts,graphicx, color}
\numberwithin{equation}{section}
\numberwithin{figure}{section}
\newtheorem{theorem}{Theorem}[section]
\newtheorem{lemma}{Lemma}[section]
\newtheorem{proposition}{Proposition}[section]

\newtheorem{remark}{Remark}[section]

\numberwithin{equation}{section}
\newcommand{\nc}{\newcommand}
\def\R{\mathbb{R}}

\nc{\diff}[2]{\frac{d #1}{d #2}}
\nc{\diffn}[3]{\frac{d^{#3} #1}{d {#2}^{#3}}}
\nc{\pdiff}[2]{\frac{\partial #1}{\partial #2}}
\nc{\pdiffn}[3]{\frac{\partial^{#3} #1}{\partial{#2}^{#3}}}
\nc{\abs}[1] {\lvert #1 \rvert}
\nc{\norm}[2] {{\lVert #1 \rVert}_{#2}}
\nc{\threeline}[1] {\lvert \lvert \lvert #1 \rvert\rvert\rvert}
\nc{\iamMc}{\frac{i\alpha-M}{c}}
\nc{\iapMc}{\frac{i\alpha+M}{c}}

\nc{\bXYZ}{{\bf XYZ\ }}
\nc{\GW}{{\bf GW1\ }}
\nc{\Rem}{{\rm Rem}}
\nc{\loc}{{\rm loc}}
\nc{\cF}{{\cal F}}
\nc{\cG}{{\cal G}}
\nc{\cO}{{\cal O}}
\nc{\cQ}{{\cal Q}}
\nc{\cR}{{\cal R}}
\nc{\cS}{{\cal S}}
\nc{\sqrtE}{\mu}
\nc{\cI}{{\cal I}}
\nc{\cK}{{\cal K}}
\nc{\cL}{{\cal L}}
\nc{\cM}{{\cal M}}
\nc{\cN}{{\cal N}}
\nc{\cE}{{\cal E}}
\nc{\cH}{{\cal H}}
\nc{\cZ}{{\cal Z}}
\nc{\cT}{{\cal T}}
\nc{\rhoo}{{(z\cdot\xi)}}
\nc{\omegaa}{{(z\cdot\eta)}}

\nc{\order}{{\cal O}}
\nc{\ores}{{\omega_{\rm res}}}
\nc{\nit}{\noindent}
\nc{\Eplus}{E_+}
\nc{\Eminus}{E_-}
\nc{\Epm}{E_\pm}

\nc{\w}{\omega}
\nc{\eps}{\epsilon}
\nc{\e}{\varepsilon}
\nc{\g}{\gamma}
\nc{\z}{\zeta}
\nc{\G}{\Gamma}

\nc{\nn}{\nonumber}
\nc{\D}{\partial}
\nc{\pZ}{\partial_Z}
\nc{\pT}{\partial_T}
\nc{\pz}{\partial_z}
\nc{\pt}{\partial_t}
\nc{\vu}{\Vec u}
\nc{\vE}{\Vec {\cal E}}
\nc{\vr}{\Vec r}
\nc{\vrho}{\Vec \rho}
\nc{\Reps}{R^{\e}}
\nc{\Vreps}{\Vec \Reps}
\nc{\half}{\frac{1}{2}}
\nc{\bphi}{\bar{\phi}}
\nc{\efour}{{\Hat e}_4}
\nc{\marginnote}[1] {\marginpar{\tiny #1}}

\numberwithin{equation}{section}
\huge
\begin{document}
\title{Equipartition of  Mass  in Nonlinear Schr\"odinger / Gross-Pitaevskii Equations \\ }\author{Zhou Gang$^{\ast}$, \
Michael I. Weinstein$^{\dag}$}
\maketitle
\centerline{\small{$^{\ast}$Institute for Theoretical Physics, ETH Z\"urich, Switzerland}} \centerline{\small{$^{\dag}$Department of
Applied Physics and Applied Mathematics, Columbia University, New York, U.S.A.}}
\setlength{\leftmargin}{.1in}
\setlength{\rightmargin}{.1in}
\normalsize \vskip.1in \setcounter{page}{1}
\setlength{\leftmargin}{-.2in} \setlength{\rightmargin}{-.2in}
\section*{Abstract}

We study the infinite time dynamics of a class of nonlinear Schr\"odinger / Gross-Pitaevskii equations. In a previous paper, \cite{GaWe}, we prove the asymptotic stability of the nonlinear ground state in a general situation which admits degenerate {\it neutral modes} of arbitrary finite multiplicity, a typical situation in systems with symmetry. Neutral modes correspond  to
 purely imaginary (neutrally stable) point spectrum of the
 linearization of the Hamiltonian PDE about a critical point. In particular,  a small perturbation of the nonlinear ground state, which typically excites such neutral modes and radiation, will evolve toward an asymptotic nonlinear ground state soliton plus decaying neutral modes plus decaying radiation. In the present article, we give a much more detailed, in fact quantitative, picture of the asymptotic evolution. Specificially we prove an {\it equipartition law}:\\
The asymptotic soliton which emerges, $\phi^{\lambda_\infty}$, has a mass which is equal to the initial soliton mass  plus
 one half the mass, $|z_0|^2$,  contained in initially perturbing neutral modes:
 \begin{equation}
 \|\phi^{\lambda_\infty}\|_{L^2}^2\ =\   \|\phi^{\lambda_0}\|_{L^2}^2\ +\ \frac{1}{2}|z_0|^2\ +\ o(|z_0|^2)\
 \nn\end{equation}
\section{Introduction}

In this paper we study the nonlinear Schr\"odinger / Gross-Pitaevskii (NLS/GP) equations in $\mathbb{R}^{3}$
\begin{align}\label{eq:NLS}
i\D_t\psi&=-\Delta \psi+V\psi-|\psi|^{2\sigma}\psi,
\end{align}
where $\sigma\geq 1$, $V:\mathbb{R}^{3}\rightarrow \mathbb{R}$ is a real, smooth function decaying rapidly at spatial infinity.
 We study the large time distribution of mass /  energy  of solutions  with initial data
 \begin{equation}
 \psi(x,0)\ =\ \psi_0,
 \label{data}
 \end{equation}
 which are sufficiently  small in the ${H}^2(\mathbb{R}^{3})$ norm \footnote{Since our results are in the low energy / small amplitude regime, our analysis goes through without change for the nonlinearities of the form $+g|\psi|^{2\sigma}\psi$ for any fixed real $g$. Here, we have taken $g=1$.}

NLS/GP arises in many physical contexts. In quantum physics, it describes a mean-field limit, $N\to\infty$, of the  linear quantum description of $N-$ weakly interacting bosons. Here, $\psi$ is a collective wave-function and $V$, a trapping potential, and the nonlinear potential arises due to the collective effect of many quantum particles on a  representative
 particle \cite{Pitaevskii-Stringari:03,ESY:07}.  In classical electromagetics and nonlinear optics,  NLS/GP arises via the {\it paraxial approximation} to Maxwell's equations,  and governs the slowly varying envelope, $\psi$, of a nearly monochromatic beam of light, propagating through a waveguide \cite{Boyd:08,Newell-Maloney:03}. The waveguide has linear refractive index profile, determining the potential, $V$,  and cubic ($\sigma=1$) nonlinear refractive index, due to the optical Kerr effect.

NLS/GP is a infinite-dimensional Hamiltonian system and a unitary evolution in $L^2(\mathbb{R})$. In the $N-$ body quantum setting the time-invariant $L^2$ norm corresponds to the conservation of mass. In the electromagnetic setting, it is the conservation of energy (optical power). In this paper, we prove an equipartition law
(Theorem \ref{SEC:MainTheorem}) for the  $L^2$ mass / energy
 small (weakly nonlinear) solutions. Hence, we may refer to this result equipartition of energy or equipartition of mass.
\bigskip

The mathematical set-up is as follows. We choose a spatially decaying potential $V $ for which the Schr\"odinger operator, $-\Delta+V$, has only two negative eigenvalues
\begin{equation}
e_0\ <\ e_1<0.
\nn\end{equation}
$e_1$ is chosen to be closer to the continuous spectrum than to $e_0$,
(permitting coupling via nonlinearity of discrete and continuum modes at quadratic order in the nonlinear coupling coefficient, $g$):
\begin{equation}
 2e_1\ -\ e_0\ >\ 0.
 \nn\end{equation}
  The excited state eigenvalue $e_1$ may be degenerate with multiplicity $N$. (In Section \ref{SEC:summary}, we allow for nearly degenerate excited state eigenvalues.)  Denote the corresponding eigenvectors by
\begin{equation}\label{eq:excitedMode}
  \phi_{lin},\ \ \xi_1^{lin}, \cdots,\ \xi_{N}^{lin}.
\end{equation}

For NLS/GP,  \eqref{eq:NLS}, there is a family of {\it nonlinear ground states} which bifurcates from the zero solution in the direction of $\phi_{lin}$. The excited state eigenvectors  are manifested as neutral modes (time periodic states with non-zero frequency) of the linearized NLS/GP equation about the ground state family; see Section \ref{SEC:GroundState}.

More specifically,  there exists an open interval $\mathcal{I}$, with $e_0$ as an endpoint, such that for any
   $\lambda\in \mathcal{I}$, NLS/GP~\eqref{eq:NLS} has solutions of the form
\begin{equation}\label{eq:groundstate}
\psi(x,t)=e^{i\lambda t}\phi^{\lambda}(x),
\end{equation}
where $\phi^{\lambda}$ is asymptotically collinear to  $\phi_{lin}$ for small  $H^2$ norm and $\lambda\to-e_0,\ \lambda\in\mathcal{I}$.

The excited state
eigenvalues give rise, in the linear approximation, to neutral modes, $(\xi,\eta)^T$, and therefore linearized time-dependent solutions, which are undamped (neutral) oscillations about $\phi^\lambda$ :
  \begin{equation}\label{eq:LinApp}
   e^{i\lambda t} \left(\ \phi^{\lambda} +\ (\Re z)\cdot\xi\ +\ i(\Im z)\cdot\eta\
\ \right)
\end{equation}
where $z\in\mathbb{C}^N$.

  %

In ~\cite{GaWe}, also referred to in this paper as \GW,  we proved the asymptotic stability of the ground states. Namely, if the initial condition is of the form
\begin{equation}\label{eq:asymStab}
\psi_0= e^{i\gamma_0}[\phi^{\lambda_0}+ R_0]
\end{equation} for some $\gamma_0\in\mathbb{R}$ and $R_0:\mathbb{R}^3\rightarrow \mathbb{C}$ satisfying $\|\langle x\rangle^{4}R_0\|_{H^2}\ll \|\phi^{\lambda_0}\|_{2},$ then generically there exists a $\lambda_{\infty}\in \mathcal{I}$ such that
\begin{equation}\label{eq:asympto2}
\min_{\gamma\in \mathbb{R}}\|\psi(t)-e^{i\gamma}\phi^{\lambda_{\infty}}\|_{\infty}\rightarrow 0\ \text{as}\ t\rightarrow \infty\ ;
\end{equation}
In particular, the neutral oscillatory modes eventually damp to zero as $t\to\infty$ via the coupling and transfer of their energy to the nonlinear ground state and to continuum radiation modes. When the neutral mode is simple, i.e. $N=1$ in ~\eqref{eq:excitedMode}, similar results have been obtained in \cite{SW:99,TsaiYau02,BuSu,SW:04,GS2,Cuccagna:03}.

In the present paper, we seek a more detailed, quantitative description of the large time dynamics.
  We consider a special class of  initial conditions to which the results of \GW, in particular, \eqref{eq:asympto2} apply:
$$\psi_0=e^{i\gamma_0}\phi^{\lambda_0}\ +\ {\rm neutral\ modes}\ +\ R_0$$  with
 $$\|\phi^{\lambda_0}\|_{2}\gg\  \|{\rm neutral\ modes}\|_2\ \gg \|\langle x\rangle^{4}R_0\|_{H^2}.$$
The main result of this paper, proved by a considerable refinement of the analysis in \cite{GaWe},  is that the emerging asymptotic ground state has, up to high order corrections, a mass equal to its initial mass plus one-half of the mass of the initial excited state mass:
\begin{equation}
\|\phi^{\lambda_{\infty}}\|_{2}^2\ =\ \|\phi^{\lambda_0}\|_{2}^{2}\ +\
\frac{1}{2}\ \|{\rm neutral\ modes}\|_2^2\ \left(\ 1+o(1)\ \right).
\nn\end{equation}
Thus, half of the excited state mass goes into forming a limiting, more massive,  ground state, $\phi^{\lambda_\infty}$ and the other half of the excited state mass is radiated away to infinity. We call this the  {\it mass-} or {\it energy- equipartition}. That this phenomenon is expected, was discussed in
~\cite{SW:99,SW:04,SW-PRL:05}. The main achievement of the present work is a  {\it rigorous quantification} of the asymptotic ($t\to\infty$) mass / energy distribution. \bigskip

The paper is organized as follows: In Section ~\ref{SEC:GroundState} we review  results on the existence and properties of the ground state manifold, and on the spectral properties of the linearized NLS/GP operator about the ground state. In Section ~\ref{SEC:MainTheorem} we state and discuss Theorem  \ref{THM:MassTransfer} on equipartition. In Section ~\ref{sec:ProveMainTHM} we present the proofs, using technical estimates established in the appendices, {\it e.g.} Sections ~\ref{sec:approxPos}-~\ref{sec:compare}. In Section ~\ref{SEC:summary}, we present a generalization of the Theorem \ref{THM:MassTransfer} to the case of nearly degenerate case, and an outline of its proof.  A more extensive list of references and a discussion of related work on NLS/GP appears in \GW~.

\section*{Acknowledgments} ZG was supported, in part, by a Natural Sciences and Engineering Research Council of Canada (NSERC)  Postdoctoral Fellowship and NSF Grant DMS-04-12305 .
 MIW was supported, in part, by  U.S.
NSF Grants DMS-04-12305, DMS-07-07850 and DMS-10-08855. This work was initiated while ZG was a visitor at the Department of Applied Physics and Applied Mathematics at Columbia University, and was continued while he was a visiting postdoctoral fellow at the Department of Mathematics of Princeton University.
\subsection{Notation}\label{notation}
\begin{itemize}
\item[(1)]\
$
\alpha_+ = \max\{\alpha,0\},\ \ [\tau]=\max_{\eta\in \mathbb{Z}}\ \{\eta\le\tau\}
$
\item[(2)]\ $\Re z$ = real part of $z$,\ \ $\Im z$ = imaginary part of $z$
\item[(3)] Multi-indices
\begin{align}
&z\ =\ (z_1,\dots, z_N) \in \mathbb{C}^N,\ \bar{z}=(\bar{z}_1,\dots, \bar{z}_N)\\
&a\in \mathbb{N}^N,\ z^a=z_1^{a_1}\cdot\cdot\cdot z_N^{a_N}\nonumber\\
&|a|\ =\ |a_1|\ +\ \dots\ +\ |a_N|
\nonumber
\end{align}
\item[(4)] $Q_{m,n}$ denotes an expression of the form
\begin{equation}
Q_{m,n}\ = \sum_{|a|=m,\ |b|=n}\ q_{a,b}\ z^a\ \bar{z}^b\
          = \sum_{|a|=m,\ |b|=n}\ q_{a,b}\prod_{k=1}^{N}\ z_k^{a_k}\bar{z_k}^{b_k} \nonumber
\end{equation}
\item[(5)]  \begin{equation}
J\ =\ \left(\begin{array}{cc} 0 & 1\\ -1 & 0\end{array}\right),
\ \ H\ =\ \left(\begin{array}{cc} L_+ & 0\\ 0 & L_-\end{array}\right),\ \
 L=JH=\left(\begin{array}{cc} 0 & L_-\\ -L_+ & 0\end{array}\right)
\nonumber\end{equation}
\item[(6)] $\sigma_{ess}(L)=\sigma_c(L)$ is the essential (continuous) spectrum of $L$,\\ $\sigma_{disc}(L)$ is the discrete spectrum of $L$.
\item[(7)]\ Riesz projections: $P_{disc}(L)$ and $P_c(L)=I-P_{disc}(L)$\\
 $P_{disc}(L)$ projects onto the discrete spectral part of $L$\\
$P_c(L)$ projects onto the continuous spectral part of $L$
\item[(8)]\ $\langle f,g\rangle = \int\ f(x)\ {\overline{g(x)}}\ dx $
\item[(9)]\ $\| f\|_p^p=\ \int_{\R^3}\ |f(x)|^p\ dx,\  \ 1\le p\le\infty$
\item[(10)]\ $\| f\|_{H^s(\mathbb{R}^3)}^2=\  \int \left|\left(I-\Delta_x\right)^{s\over2}f(x)\right|^2\ dx$
\item[(11)]\ $\| f\|_{\cH^{s,\nu}}^2\ =\  \int_{\R^3}\ \left|\langle x\rangle^\nu\ (I-\Delta)^{\frac{s}{2}}f(x)\right|^2\ dx$
\end{itemize}
%
\section{Review of the set up}\label{SEC:GroundState}
In this section we review the setting presented in detail in \cite{GaWe}.

\subsection{Assumptions on the potential, $V(x)$}\label{Vassumptions}
We assume that the Schr\"odinger operator $-\Delta+V$ has
the following properties:
%
\begin{enumerate}
\item[(V1)]
$V$ is real valued and decays sufficiently rapidly, {\it e.g.} exponentially,  as $|x|$ tends to infinity.\\
\item[(V2)]  $-\Delta+V$ has two eigenvalues $e_{0}<e_{1}<0$.\\
 $e_{0}$ is the lowest eigenvalue with
ground state $\phi_{lin}>0$, the eigenvalue $e_{1}$ is degenerate
with multiplicity $N$ and eigenvectors
$\xi_{1}^{lin},\xi_{2}^{lin},\cdot\cdot\cdot,\xi_{N}^{lin}.$
\end{enumerate}

\subsection{Bifurcation of ground states from $e_0$}
\begin{proposition}\label{bif-of-gs}
Suppose that the linear operator $-\Delta+V$ satisfies the
conditions above in subsection \ref{Vassumptions}. Then there exists a constant
$\delta_{0}>0$ and a nonempty interval $\mathcal{I}\subset [e_{0}-\delta_{0},
e_{0})$ such that for any $\lambda \in \mathcal{I}$, NLS/GP  (~\ref{eq:NLS}) has solutions of the form
$\psi(x,t)=e^{i\lambda t}\phi^{\lambda}\in {L}^{2}$\
with
\begin{equation}\label{eqn:perturb}
\phi^{\lambda}=\delta\left(\ \phi_{lin}+\cO(\delta^{2\sigma})\ \right)\ {\rm and}\
\delta=\delta(\lambda)=|e_{0}+\lambda|^{\frac{1}{2\sigma}}\left(\int \phi^{2\sigma+2}_{lin}\right)^{-\frac{1}{2\sigma}}.
\end{equation}
\end{proposition}

\subsection{Linearization of NLS/GP about the ground state}

If we write $\psi(x,t)=e^{i\lambda t}\left(\ \phi^\lambda\ + u +\ iv\ \right)$, then we find  the linearized perturbation equation to be:
\begin{equation}
\frac{\D}{\D t}\left(\begin{array}{ll} u\\ v \end{array} \right)\ =\
 L(\lambda)\ \left(\begin{array}{ll} u\\ v \end{array} \right)\ =\ JH(\lambda)\ \left(\begin{array}{ll} u\\ v \end{array} \right),
 \label{linearized}
 \end{equation}
 where
\begin{equation}\label{eq:opera}
L(\lambda):=
\left(\begin{array}{lll}0&L_{-}(\lambda)\\
 -L_{+}(\lambda)&0 \end{array} \right)
 = \left(\begin{array}{lll}0&1\\
 -1&0 \end{array} \right)\ \left(\begin{array}{lll}L_+(\lambda)&0\\
0& L_{-}(\lambda) \end{array} \right)\ \equiv\ JH(\lambda)\ .
 \end{equation}

Here, $L_+$ and $L_-$ are given by:
\begin{align}
 L_{-}(\lambda)&\ :=\ -\Delta+\lambda+V-(\phi^{\lambda})^{2\sigma}\nn\\L_{+}(\lambda)&\ :=-\Delta+\lambda+V-(2\sigma+1)(\phi^{\lambda})^{2\sigma}\label{Lpm}
 \end{align}

\nit The following results on the point spectrum of $L(\lambda)$ appear in  ~\cite{GaWe};
 see Proposition 4.1, p. 275 and Propositions 5.1-5.2, p. 277:
\begin{lemma}\label{LM:NearLinear}
Let $L(\lambda)$, or more explicitly,  $L(\lambda(\delta),\delta)$
denote the linearized operator about the the bifurcating state
$\phi^\lambda, \lambda=\lambda(\delta)$. Note that $\lambda(0)=
-e_0$. Corresponding to the degenerate e-value, $e_1$, of
$-\Delta+V$, the matrix operator $L(\lambda=-e_0,\delta=0)$ has
degenerate eigenvalues $\pm iE(-e_0)=\pm i(e_1-e_0)$, each of
multiplicity $N$. For $\delta>0$ and small these bifurcate to
(possibly degenerate) eigenvalues $\pm iE_1(\lambda),\dots,$ $\pm
iE_N(\lambda)$ with neutral modes
$$\left(
\begin{array}{lll}
\xi_{1}\\
\pm i\eta_{1}
\end{array}
\right),\ \left(
\begin{array}{lll}
\xi_{2}\\
\pm i\eta_{2}
\end{array}
\right),\ \cdot\cdot\cdot, \left(
\begin{array}{lll}
\xi_{N}\\
\pm i\eta_{N}
\end{array}
\right)$$ satisfying the estimates
\begin{equation}\label{eq:Orthogonality}
\langle \xi_{m},\eta_{n}\rangle =\delta_{m,n},\ \langle \xi_{m},\phi^{\lambda}\rangle=\langle \eta_{m},\partial_{\lambda}\phi^{\lambda}\rangle=0
\end{equation}
and
\begin{equation}\label{eq:GoToNear}
0\not=\displaystyle\lim_{\lambda\rightarrow
e_{0}}\xi_{n}=\lim_{\lambda\rightarrow e_{0}}\eta_{n}\in
span\{\xi^{lin}_{n},\ n=1,2,\cdot\cdot\cdot,N\}\ \text{in}\
H^{k}\ \text{spaces for any}\ k>0.
\end{equation}
\end{lemma}

\begin{remark}\label{remark:2ndorderres} Since $E(-e_0)=e_1-e_0$, it follows that if $2e_1-e_0>0$, then for sufficiently small $\delta$, $2E_n(\lambda)>\lambda,\ n=1,2,\cdot\cdot\cdot,N$. This ensures nonlinear coupling of discrete to continuous spectrum at second order (in the nonlinearity coefficient, $g$). Thus, to ensure such coupling, we assume:
\begin{itemize}
\item[(V3)]
 \begin{equation} 2e_1-e_0>0.\ \label{2e1me0}
\end{equation}
\end{itemize}
\end{remark}

\begin{lemma}\label{mainLem2}
Assume the potential $V=V(|x|)$ and the functions $\xi^{lin}_{n}$ admit the form $\xi^{lin}_{n}=\frac{x_{n}}{|x|}\xi^{lin}(|x|)$ for some function $\xi^{lin}$, then $\phi^{\lambda}$, hence $\partial_{\lambda}\phi^{\lambda}$, is spherically symmetric, $E_{n}=E_{1}$ for any $n=1,2,\cdot\cdot\cdot,N=d$ and we can choose $\xi_{n},\eta_{n}$ such that $\xi_{n}=\frac{x_{n}}{|x|}\xi(|x|)$ and $\eta_{n}=\frac{x_{n}}{|x|}\eta(|x|)$ for some real functions $\xi$ and $\eta.$
\end{lemma}
\bigskip

In this paper we make the following assumptions on the spectrum of the operator
$L(\lambda):$
\begin{enumerate}
 \item[{\bf (SA)}] The linearized operator $L(\lambda)$ has discrete spectrum given by:
 \subitem - an eigenvalue $0$
    with
generalized eigenspace spanned by
$\left\{\ \left(
\begin{array}{lll}
0\\
\phi^{\lambda}
\end{array}
\right)\ ,\ \left(
\begin{array}{lll}
\partial_{\lambda}\phi^{\lambda}\\
0
\end{array}
\right)\ \right\}$
\subitem - {\it neutral eigenvalues} $\pm iE(\lambda),\ E(\lambda)>0$,\\  satisfying the condition $2E(\lambda)>\lambda$ and with corresponding eigenvectors
$\left|
\begin{array}{lll}
\xi_{n}\\
\pm i\eta_{n}
\end{array}
\right\rangle,\ n=1,\dots,N $.
\end{enumerate}
For the non self-adjoint operator $L(\lambda)$ the (Riesz)
projection onto the discrete spectrum subspace of $L(\lambda)$,
$P_{d}=P_{d}(L(\lambda))=P_{d}^\lambda$, is given explicitly in ~\cite{GaWe}, Proposition 5.6, p. 280:
\begin{equation}\label{eq:PdProjection}
\begin{array}{lll}
P_{d}&\equiv&\frac{2}{\partial_{\lambda}\| \phi^{\lambda}\|^{2}}\left(\ \left|
\begin{array}{lll}
0\\
\phi^{\lambda}
\end{array}
\right\rangle \left\langle
\begin{array}{lll}
0\\
\partial_{\lambda}\phi^{\lambda}
\end{array}
\right|\ +\ \left|
\begin{array}{lll}
\partial_{\lambda}\phi^{\lambda}\\
0
\end{array}
\right\rangle \left\langle
\begin{array}{lll}
\phi^{\lambda}\\
0
\end{array}
\right|\ \right)\\
& &\\
& &+\frac{1}{2}i\displaystyle\sum_{n=1}^{N}\left(\ \left|
\begin{array}{lll}
\xi_{n}\\
i\eta_{n}
\end{array}
\right\rangle\left\langle
\begin{array}{lll}
-i\eta_{n}\\
\xi_{n}
\end{array}
\right| \ -\  \left|
\begin{array}{lll}
\xi_{n}\\
-i\eta_{n}
\end{array}
\right\rangle\left\langle
\begin{array}{lll}
i\eta_{n}\\
\xi_{n}
\end{array}
\right|\ \right)\ .
\end{array}
\end{equation}
and the projection onto the essential spectrum by $P_{c}\ \equiv\ 1-P_{d}.$
\medskip

The large time analysis of NLS/GP requires good decay estimates on the linearized evolution operator, $e^{L(\lambda)t}P^\lambda_c$. An obstruction to such estimates are, so-called, threshold resonances
 (see \cite{GaWe} and references therein), which we preclude with the following hypothesis.
\begin{enumerate}
 \item[({\bf Thresh$_\lambda$)}] Assume $L(\lambda)$ has no resonances at $\pm i\lambda$
\end{enumerate}
For small solitons, $\delta$ sufficiently small, ({\bf Thresh$_\lambda$}) follows from the absence of a zero energy resonance for  $-\Delta+V$.

\subsection{ Second order (nonlinear) Fermi Golden Rule}\label{sec:fgr}
In this subsection we review the definitions and constructions presented in detail in ~\cite{GaWe} pp. 281-282. The amplitudes and phases of the neutral modes are governed by the complex-valued vector parameter $z: \ \mathbb{R}^{+}\rightarrow \mathbb{C}^{N}$, first arising in the  linear approximation of solution $\psi$, see {\it e.g.} ~\eqref{eq:LinApp}. Its precise definition is seen in the decomposition of the solution $\psi$ in ~\eqref{Decom}, under the condition ~\eqref{s-Rorthogonal}, below, from which it follows that
\begin{equation}
\partial_{t}z =-iE(\lambda) z -\Gamma( z ,\bar{ z }) z +\Lambda( z ,\bar{ z })
z\ +\ \cO\left((1+t)^{-\frac{3}{2}-\delta}\right),\ \ \delta>0\label{new-nf1}
\end{equation}
where $\pm iE(\lambda)$ are complex conjugate $N-fold$ degenerate
neutral eigenfrequencies of $L(\lambda)$, $\Gamma$ is non-negative symmetric and $\Lambda$ is skew symmetric.

In what follows we define the non-negative, Fermi Golden Rule matrix,  $\Gamma$.
Define vector functions $G_{k},\ k=1,2,\cdot\cdot\cdot, N$, as
\begin{equation}\label{eq:Fk2}
G_{k}(z,x):=\left(
\begin{array}{lll}
B(k)\\
D(k)
\end{array}
\right)
\end{equation} with the functions $B(k)$ and $D(k)$ defined as $$\begin{array}{lll}
B(k)&:=&-i\sigma(\phi^{\lambda})^{2\sigma-1} \ \ \left[\ \rhoo\ \eta_{k}+\omegaa\ \xi_{k}\ \right]\ , \\
D(k)&:=&-\sigma(\phi^{\lambda})^{2\sigma-1}
\left[\ 3\rhoo\xi_{k}-\omegaa\eta_{k}\ \right]\
 -\ 2\sigma(\sigma-1)(\phi^{\lambda})^{2\sigma-1}\ \rhoo\xi_{k}\ ,
\end{array}$$
where
$$
z\cdot\xi\ :=\ \displaystyle\sum_{n=1}^{N}z_{n}\xi_{n},\ z\cdot\eta\  :=\ \displaystyle\sum_{n=1}^{N}z_{n}\eta_{n}.
  $$
In terms of the column 2-vector, $G_{k}$, we define
a $N \times N$ matrix $Z(z,\bar{z})$ as
\begin{equation}\label{eq:zMatrix}
Z(z,\bar{z})=(Z^{(k,l)}(z,\bar{z})),\ \ 1\le k,l\le N
\end{equation} and
\begin{equation}
Z^{(k,l)}\ =\ -\left\langle
(L(\lambda)+2iE(\lambda)-0)^{-1}P_{c}G_{l}, iJP_cG_{k}\right\rangle
\label{Zkl-sym}
\end{equation}
Finally, we define $\Gamma(z,\bar{z})$ as follows:
\begin{equation}
\Gamma(z,\bar{z})\ :=\ \frac{1}{2}[Z(z,\bar{z})+Z^{*}(z,\bar{z})].
\label{Gammadef}
\end{equation}
Thus,
 \begin{equation}
[\ \Gamma(z,\bar{z})\ ]_{kl}\ =\ -\ \Re\ \left\langle
(L(\lambda)+2iE(\lambda)-0)^{-1}P_{c}G_{l}, iJP_cG_{k}\right\rangle.
\label{Gamma-kl}
\end{equation}
By ~\eqref{new-nf1} and ~\eqref{Gamma-kl} we find
\begin{equation}
\partial_{t}\ |z(t)|^2\ =\ -2\ z^*\ \Gamma(z,\bar{z})\ z\ +\ \dots.
\label{energy-id}\end{equation}

In \GW $\Gamma$ was shown to be non-negative and we require it to be positive definite.
 In particular, we shall require the following Fermi Golden Rule hypothesis.\\ \\ Let $P_{c}^{lin}$ denote the spectral projection onto the essential spectrum of $-\Delta+V$. Then
\begin{enumerate}
\item[{\bf (FGR)}] We assume there exists a constant $C>0$ such that
$$-\Re\langle i[-\Delta+V+\lambda-2E(\lambda)-i0]^{-1}P_{c}^{lin}(\phi_{lin})^{2\sigma-1}(z\cdot \xi^{lin})^{2},(\phi_{lin})^{2\sigma-1}(z\cdot \xi^{lin})^{2}\rangle\ge C|z|^4$$ for any $z\in \mathbb{C}^{N}.$
\end{enumerate} The assumption FGR implies that there exists a constant $C_1>0$ such that for any $z\in \mathbb{C}^{N}$
\begin{equation}\label{eq:FGR}
z^*\ \Gamma(z,\bar{z})\ z\geq C_1\|\phi^{\lambda}\|_{\infty}^{4\sigma-2} |z|^{4}.
\end{equation}

Note that for each fixed $z$, smallness of $|\lambda+e_0|$ together with ~\eqref{eqn:perturb} and ~\eqref{eq:GoToNear} imply that the leading term in $z^*\ \Gamma(z,\bar{z})\ z$ is
\begin{equation}\label{eq:Gamma0}
\begin{array}{lll}
& &z^* \Gamma_0(z,\bar{z})\ z\\
&\equiv &-
2\sigma^{2}(\sigma+1)^{2}\delta^{4\sigma-2}(\lambda)\ \times \ \\
& &\Re\langle i[-\Delta+V+\lambda-2E(\lambda)-i0]^{-1}P_{c}^{lin}(\phi_{lin})^{2\sigma-1}(z\cdot\xi^{lin})^{2},
(\phi_{lin})^{2\sigma-1}(z\cdot\xi^{lin})^{2}\rangle.
\end{array}
\end{equation}

\section{Main Theorem}\label{SEC:MainTheorem}
In this section we state our main results, Theorems \ref{THM:MassTransfer} and \ref{THM:MainTheorem2}.
\begin{theorem}\label{THM:MassTransfer}
Assume a cubic nonlinearity, $\sigma=1$, in  ~\eqref{eq:NLS}.  If the spectral conditions {\bf (SA) (Thres$_\lambda$) and (FGR)}  are satisfied,
then there exists a constant $\delta$ such that if the initial condition $\psi_{0}$ satisfies the condition
$$\psi_{0}(x)=e^{i\gamma_{0}}[\phi^{\lambda_{0}}+\alpha_{0}\cdot\xi+i\beta_{0}\cdot \eta +R_{0}]$$
for some real constants $\gamma_{0},\ \lambda_{0}$, real $N$ vectors $\alpha_{0}$ and $\beta_{0}$, function $R_0:\mathbb{R}^{3}\rightarrow \mathbb{C}$, such that for $\epsilon\le\delta$:\\
 $|\lambda_{0}-|e_{0}||\leq \epsilon,\ \ |\alpha_{0}|+|\beta_{0}|\ \lesssim \epsilon \|\phi^{\lambda_{0}}\|_{2}$,
  $\|\langle x\rangle^{4}R_{0}\|_{H^2}\lesssim |\alpha_{0}|^{2}+|\beta_{0}|^{2}=\mathcal{O}(\epsilon^2)$, for  $\epsilon\leq \delta$,  then there exist smooth functions
  \begin{align}
  &\lambda(t):\mathbb{R}^{+}\rightarrow
\mathcal{I},\ \ \ \gamma(t): \mathbb{R}^{+}\rightarrow
\mathbb{R},\  z(t):\mathbb{R}^{+}\rightarrow \mathbb{C}^{N},\nn\\
 &\ \  R(x,t):\mathbb{R}^{3}\times\mathbb{R}^{+}\rightarrow
\mathbb{C}
\nn\end{align}
 such that the solution of NLS evolves in the form:
\begin{align}\label{Decom}
\psi(x,t)&\ =\ e^{i\int_{0}^{t}\lambda(s)ds}e^{i\gamma(t)}\nn\\
&\ \ \ \ \ \ \ \times[\ \ \phi^{\lambda}+a_{1}(z,\bar{z})
\D_\lambda\phi^{\lambda}+ia_{2}(z,\bar{z})\phi^{\lambda}\ +\ (Re\ \tilde{z})\cdot\xi\ +\ i(Im \tilde{z})\cdot\eta\ +\ R\ \ ],\end{align}
where  $\lim_{t\rightarrow \infty}\lambda(t)=\lambda_{\infty},$
for some $\lambda_{\infty}\in \mathcal{I}$.\\
Here, $a_{1}(z,\bar{z}),\ a_{2}(z,\bar{z}): \mathbb{C}^{N}\times\mathbb{C}^{N}\rightarrow \mathbb{R}$ and $\tilde{z}-z: \mathbb{C}^{N}\times\mathbb{C}^{N}\rightarrow \mathbb{C}^{N}$
 are some polynomials of $z$ and $\bar{z}$, beginning with terms of order $|z|^{2}$.

\begin{enumerate}
\item[(A)]  The dynamics of mass/energy transfer is captured by the following reduced dynamical system for the  key modulating parameters, $\lambda(t)$ and $z(t)$:
\begin{equation}\label{eq:IncreaseLambda}
\frac{d}{dt}\ \|\phi^{\lambda(t)}\|_2^2 =z^* \Gamma_0(z,\bar{z})\ z +\ \cS_\lambda(t),
\end{equation}
\begin{equation}\label{eq:DecayZ}
\frac{d}{dt}\ |z(t)|^2  = -2z^* \Gamma_0(z,\bar{z})\ z+\cS_{z}(t),
\end{equation}
where $z^*\Gamma_0(z,\bar{z})z$ is given in ~\eqref{eq:Gamma0}, and
\begin{equation}
S_\lambda(t)\lesssim\ (1+t)^{-\frac{19}{10}},\ \ {\rm and}\ \ S_z(t)\ \lesssim\ (1+t)^{-\frac{12}{5}}.\label{new-lam-z}
\end{equation}
Furthermore,
\begin{equation}
\int_0^\infty\ S_\lambda(\tau)d\tau,\  \ \ \ \int_{0}^{\infty} \cS_z(\tau)\ d\tau\ =o(|z_0|)^2.
\label{SL1}\end{equation}
 \item[(B)] $\vec{R}(t)=( Re\ R(t)\ ,\ Im\ R(t))^T$ lies in the essential spectral part of $L(\lambda(t))$. Equivalently, $R(\cdot,t)$ satisfies the symplectic orthogonality conditions:
 \begin{align}\label{s-Rorthogonal}
&\omega\langle R,i\phi^{\lambda}\rangle\ =\ \omega\langle
R,\partial_{\lambda}\phi^{\lambda}\rangle\ =\ 0, \nn\\
&\omega\langle
R,i\eta_{n}\rangle=\omega\langle R,\xi_{n}\rangle=0,\
n=1,2,\cdot\cdot\cdot, N,
\end{align}
 where $\omega\langle X,Y\rangle:=Im\int X\overline{Y}$.
 \item[(C)] {\bf Decay estimates:} For any time $t\geq 0$
 \begin{align}
&\|(1+x^{2})^{-\nu}\vec{R}(t)\|_{2}\leq C(\|\langle x\rangle^{4}\psi_0\|_{H^2})(1+t)^{-1},
\label{Rdecay}\\
&\|\vec{R}(t)\|_{{H}^2}\leq \epsilon_{\infty},
\label{eq:stability}\\
&|z(t)|\leq C(\|\langle x\rangle^{4}\psi_0\|_{H^2})(1+t)^{  -\frac{1}{2}  }.
\label{z-decay}
\end{align} 
\item[(D)] {\bf Mass / Energy equipartition:}\ Half of the mass of the neutral modes contributes to forming a more massive asymptotic ground state and half is radiated away
\begin{equation}\label{eq:Mass}
\|\phi^{\lambda_{\infty}}\|_{2}^{2}=
\|\phi^{\lambda_{0}}\|_{2}^{2}+\frac{1}{2}\left[\ |\alpha_{0}|^{2}+|\beta_{0}|^{2}\ \right]
+o\left(\ \left[|\alpha_{0}|^{2}+|\beta_{0}|^{2}\right]\ \right).
\end{equation}
\end{enumerate}
\end{theorem}

The following result applies to the case where $\sigma>1$.
\begin{theorem}\label{THM:MainTheorem2}
Assume the general nonlinearity $\sigma >1$ . Then  statements (A)-(D) of Theorem \ref{THM:MassTransfer} hold  provided, in addition to the assumptions of Theorem \ref{THM:MassTransfer}, we assume:
 \begin{itemize}
\item[(1)] in the case where the neutral modes are degenerate ($N>1$), the potential $V$ is spherically symmetric and the eigenvectors $\xi^{lin}_{n},\ n=1,2,\cdots, N=d,$ admit the form $\xi_{n}^{lin}=\frac{x_{n}}{|x|}\xi(|x|)$ for some function $\xi.$
\item[(2)] $|\alpha_{0}|^{2}+|\beta_{0}|^{2}\leq [\|\phi^{\lambda_{0}}\|_{2}]^{C(\sigma)}$ for some sufficiently large constant $C(\sigma).$
 \end{itemize}
\end{theorem}
The statements (B) and (C) are obtained in ~\cite{GaWe}: all except ~\eqref{eq:stability} are taken from Theorem 7.1, p.  284. Equation ~\eqref{eq:stability} is from the proof (Line 18, p. 306) that $\mathcal{R}_{4}(T):=\displaystyle\max_{0\leq t\leq T} \|\vec{R}(t)\|_{{H}^2}\ll 1$;   $\mathcal{R}_{4}$ is defined  in (11.2) .

 The bounds on $S_\lambda(t)$ and $S_z(t)$, ~\eqref{new-lam-z}, of  statement (A) were proved in \cite{GaWe}; see equations (8-9) and (8-11) p. 286. For the estimate $|Remainder|\lesssim (1+t)^{-\frac{19}{10}}$, see line 9, p. 306. The remaining assertions in (A) will be reformulated as Theorems ~\ref{THM:Zequation} and ~\ref{THM:KeyTerm}, and proved in Section ~\ref{sec:ProveMainTHM}. Statement (D) is  proved just below.
\begin{remark}
{\bf Mass equipartition:}
It is straightforward to interpret  ~\eqref{eq:Mass} as implying  equipartition of the neutral mode mass. Indeed,  since  $\phi^{\lambda}$ is orthongal to $ \xi_{m}$ (see ~\eqref{eq:Orthogonality}) and since mass is conserved for NLS/GP, i.e. $\|\psi(t)\|_2=\|\psi_0\|_2$ , we have
\begin{equation}
\|\psi(\cdot,t)\|_{2}^{2}=\|\psi_0\|_{2}^{2}=\|\phi^{\lambda_{0}}\|_{2}^{2}+|\alpha_{0}|^{2}+|\beta_{0}|^{2}+o(|\alpha_{0}|^{2}+|\beta_{0}|^{2}),\ \ {\rm for\ all}\ t.\nn\end{equation}
The theorem implies that $\psi(t,\cdot)$ has a weak-$L^2$ limit, $\phi^{\lambda_\infty}$, whose mass is given by ~\eqref{eq:Mass}. Thus,  half of the mass of the neutral modes is transferred to the ground states while the other half is radiated to infinity.\\ \\
\end{remark}

\nit We now use Statement (A) of Theorem
\ref{THM:MassTransfer} to prove Statement (D).\\ \\
\nit{\bf Proof of Mass equipartition:} Twice the first plus the second equation in  (\ref{new-lam-z}) yields:
\begin{equation}
\frac{d}{dt}\left(\ 2\left\|\phi^{\lambda(t)}\right\|_2^2\ +\  |z(t)|^2\ \right)\ =
\ 2\cS_\lambda(t)\ + \cS_z(t).
\label{lincomen}\end{equation}
Integration of (\ref{lincomen}) with respect to $t$ from zero to infinity
and use of the decay of $z(t)$, (\ref{z-decay}), imply
\begin{equation}
2\left\|\phi^{\lambda(\infty)}\right\|_2^2\ =\ 2\left\|\phi^{\lambda(0)}\right\|_2^2\ +\  |z(0)|^2\ +\ \int_0^\infty\left(\ 2\cS_\lambda(t')\ + \cS_z(t')\ \right)\ dt'.
\nn\end{equation}
Dividing by two and estimating the integral, using (\ref{SL1}), completes the proof of Statement D.

\begin{remark}
{\bf Generic data in a neighborhood of the origin:}\\ For the  case of cubic nonlinearity, $\sigma=1$, the condition $|\alpha_{0}|^{2}+|\beta_{0}|^{2}\ll \|\phi^{\lambda_{0}}\|_{2}^{2}$ can be improved to a state about generic (low energy) initial conditions satisfying $$|\alpha_{0}|^{2}+|\beta_{0}|^{2}\approx \|\phi^{\lambda_{0}}\|_{2}^2.$$ We impose the stronger condition in the present paper to simplify the treatment and to apply directly the results in \GW\ \cite{GaWe}. We refer to ~\cite{SW:04,MR1992875}. See also our remarks ~\ref{remark:remark3} and ~\ref{Remark} below.\\ \\
{\bf The generality of the nonlinearity in ~\eqref{eq:NLS}} \\ Our results hold not only for focusing nonlinearity, i.e., $-|\psi|^{2\sigma}\psi$ in ~\eqref{eq:NLS}. In fact all of the results in the Theorems ~\ref{THM:MassTransfer} and ~\ref{THM:MainTheorem2} can be transferred to the general cases $g|\psi|^{2\sigma}\psi,\ g\in \mathbb{R}\backslash\{0\}$ without difficulty. We restrict to the present consideration in order not to clutter our arguments by discussing various constants.
\end{remark}

\subsection{Relation to previous work}
Theorems \ref{THM:MassTransfer} and \ref{THM:MainTheorem2} are derived from a  refinement of the analysis of \cite{GaWe} and a generalization to arbitrary nonlinearity parameter $\sigma\ge1$. In this subsection we explain this.

The overall plan for proofs of asymptotic stability can be broken into two parts, motivated by a view of the soliton as an interaction between discrete and continuum modes:
\begin{itemize}
\item[{\bf Part 1}]:\
a) We seek a natural decomposition of the solution into a component evolving along the manifold of solitons and a component which is dispersive.  However, since the linearization about the soliton may have neutral modes, non-decaying time periodic states, we incorporate these degrees of freedom among the discrete degrees of freedom in the Ansatz. The dispersive components of the evolution lie in the subspace bi-orthogonal, in fact symplectic-orthogonal, to the discrete modes. The result is a {\it strongly-coupled} system  governing the discrete degrees of freedom and dispersive dispersive wave field, $R(t)$. Mathematically we decomposed the solution $\psi$ as in ~\eqref{eq:decom}, and by the orthogonal conditions ~\eqref{eq:Orthogonality} and ~\eqref{s-Rorthogonal} we derive equations for $\dot\lambda$, $\dot\gamma$, $z$ and $\vec{R}$. These are taken  from ~\cite{GaWe} and displayed in Appendix ~\ref{sec:decom}.\\ \\
b) We solve explicitly for the leading order components of $R(t)$, which arise due to resonant forcing by new, nonlinearity-generated, discrete mode frequencies. To achieve this we find the leading order, that is second order in $z$ and $\bar{z}$ contributions to $R(x,t)$. This is presented in Appendix ~\eqref{eq:Rform}. \\ \\
c) This leading order behavior is substituted into the equations governing the discrete modes, leading to a (to leading order) closed equation for the discrete modes, implying estimates for $\dot\lambda$ and $\dot\gamma$. This is Proposition ~\ref{Prop:Majorants}.\\ \\
d) The latter is put into a normal form, via a finite sequence of near-identity changes of variables, in which the energy transfer mechanisms are made explicit. This is achieved via the introduction of  $z\mapsto a_1(z,\bar{z}),\
a_2(z,\bar{z}),\ p(z,\bar{z})$ and $q(z,\bar{z})$ in Appendix ~\ref{sec:NormalFormExp}.
\item[{\bf Part 2}]:
The full coupled system is now in a form of:\\
a finite dimensional  system of (normal form) ODEs, with non-resonant terms removed by near identity changes of variables, with rapidly time- decaying corrections, determined by the dispersive part, {\it weakly-coupled} to a  dispersive PDE, with rapidly decaying and/or oscillating source terms, coming from the discrete components of the solution. The latter is  essentially treatable by low-energy scattering methods.

In \GW\ ~\cite{GaWe} we proved that the neutral mode mass and $\lambda(t)$, which through $\|\phi^{\lambda(t)}\|_2^2$ controls the ground state mass, is governed by
\begin{equation}\label{eq:crudeLamb}
\frac{d}{dt} \lambda(t) =\ {\rm Rem}_\lambda(t),
\end{equation}
\begin{equation}\label{old-lam-z2}
\frac{d}{dt} |z(t)|^2 = -  2z^*\Gamma( z ,\bar{ z })z + {\rm Rem}_z(t),
\end{equation}
where $\Rem_\lambda(t)$ and $\Rem_z(t)$ satisfy an estimate of the form:
\begin{equation}\label{Rem-ests}
\begin{array}{lll}
|\ \Rem_{\lambda}(t) |\ &\lesssim& |z(t)|^4+\ \|\langle x\rangle^{-4}R(t)\|_{H^2}^2+\ \|R(t)\|_\infty^2+ |z(t)|\ \|\langle x\rangle^{-4}\tilde{R}(t)\|_{2}\\
& &\\
|\ \Rem_{z}(t)\ | &\lesssim& |z(t)|^5+|z(t)|\ \|\langle x\rangle^{-4}R(t)\|_{H^2}^2+|z(t)|\ \|R(t)\|_\infty^2+ |z(t)|^2\ \|\langle x\rangle^{-4}\tilde{R}(t)\|_{2}
\end{array}
\end{equation}
where, $\tilde{R}$ is defined in ~\eqref{eq:difTildeR}, and for
$t\gg1$ we have 
\begin{equation}
|z(t)|\sim t^{-\frac{1}{2}},\  \|\langle x\rangle^{-4}R(t)\|_{H^2}\sim t^{-1},\ \|R(t)\|_\infty\sim t^{-1},\ \|\langle x\rangle^{-4}\tilde{R}(t)\|_{2}\sim t^{-\frac{7}{5}}.
\nn\end{equation}
\end{itemize}
Since $\Rem_z(t)=\cO(t^{-2-\tau}),\ \tau>0$, $\Rem_z(t)$ is dominated by
 the  first term on the right hand side of \eqref{old-lam-z2}, which is $\cO(t^{-2})$ and strictly negative, by the Fermi Golden Rule resonance hypothesis {\bf (FGR)}.
 Furthermore, $\Rem_\lambda(t)$  is integrable in $t$, $\lambda(t)$ has a limit, $\lambda_\infty$.\\ \\
\section{Refinements of the analysis and outline of the proof}\label{sec:ProveMainTHM}

In view of the results of \GW, we focus on the refinements required. These concern
 the terms $\mathcal{S}_{\lambda}$ and $\mathcal{S}_{z}$ in ~\eqref{eq:IncreaseLambda} and ~\eqref{eq:DecayZ} and their estimation in ~\eqref{SL1}, for the proofs of Theorems ~\ref{THM:MassTransfer} and ~\ref{THM:MainTheorem2}. In this section we derive $\mathcal{S}_{\lambda}$ and $\mathcal{S}_{z}$ and estimate them.

Technically the main effort in the present paper is to improve the estimates for the various terms on the right hand side of ~\eqref{eq:crudeLamb} and ~\eqref{old-lam-z2}. It is relatively easier to improve the estimate for $\partial_{t}|z|^2$, since the term $-2z^* \Gamma(z,\bar{z})z$ already measures the decreasing of $|z|^2$. What is left is to prove the term $Rem_{z}$ is indeed a small correction in certain sense.

To improve the estimates of the terms on the right hand side of ~\eqref{eq:crudeLamb} is more involved. From ~\eqref{eq:crudeLamb} we can not tell the increasing or decreasing of the parameter $\lambda.$ For that purpose we expand the right hand side of ~\eqref{eq:crudeLamb} to {\it fourth order} in $z$ and $\bar{z}$ to find some sign. This in turn is achieved by expansion of the function $R$ or $\tilde{R}$ further to third order in $z$ and $\bar{z}$. For that purpose we define the third order terms in ~\eqref{eq:difRgeq3} and introduce remainder by $R_{\geq 4}$ in ~\eqref{dif:R4}.

We next present some  precise estimates on $R_{\geq 4},\ z$ and $\dot\lambda$, which are defined in Appendices
~\ref{sec:decom}-~\ref{sec:NormalFormExp}.
To facilitate later discussions we define the constant $\delta_{\infty}$ by:
\begin{equation}\label{eq:defDelta}
\delta_{\infty}:=\|\phi^{\lambda_{\infty}}\|_{L^2}
=\mathcal{O}(|\lambda_{\infty}+e_0|^{1-\frac{1}{2\sigma}})=\mathcal{O}(\delta(\lambda(t)))\ \text{for any time}\ t
\end{equation}
where the last estimate follows from the fact the soliton manifold is stable (see \cite{GaWe}). Recall the constant $\delta(\lambda)\equiv \delta$ defined and estimated in ~\eqref{eqn:perturb}, and recall $\displaystyle\lim_{t\rightarrow\infty}\lambda(t)= \lambda_{\infty}$ in Theorem ~\ref{THM:MassTransfer}.
 We have:
\begin{proposition}\label{prop:useful}
Suppose that $\frac{|z_{0}|}{\delta_{\infty}}\ll 1$ for $\sigma=1$ and $|z_{0}|\leq\delta_{\infty}^{C(\sigma)}$ for $\sigma>1$ and $C(\sigma)$ is a sufficiently large constant. Then the following results hold: there exists a constant $C>0$ such that for any time $t\geq 0$
\begin{equation}\label{eq:uppZ}
|z(t)|\leq\ C(|z_{0}|^{-2}+\delta_{\infty}^{4\sigma-2}t)^{-\frac{1}{2}};
\end{equation}
if $\sigma=1$ then
\begin{equation}\label{eq:estR4}
\|\langle x\rangle^{-4 }R_{\geq 4}\|_{2}\lesssim |z_{0}|^{2}(1+t)^{-\frac{3}{2}}+\delta_{\infty}|z_{0}|^{2}|z(t)|^2;
\end{equation}
\begin{equation}\label{eq:estJunk}
|\dot\lambda| \lesssim \delta_{\infty}\ |z(t)|^{4}+\delta_{\infty}^2 |z_0|^2 |z(t)|^3+\delta_{\infty}|z_0|^{4}(1+t)^{-3}+\delta_{\infty}|z_0|^2 |z(t)|(1+t)^{-\frac{3}{2}},
\end{equation}
and if $\sigma>1$ then
\begin{equation}
\|\langle x\rangle^{-4 }R_{\geq 4}\|_{2}\lesssim |z_{0}|^{2}(1+t)^{-\frac{3}{2}}+[|z_{0}|^{2}+|z_{0}|^{2\sigma-1}]|z(t)|^2,
\end{equation}
\begin{equation}
|\dot\lambda |\lesssim |z|^{2\sigma+1}+|z(t)|^{4}+\delta_{\infty}|z_0|^2 |z(t)|^3+|z_0|^4 (1+t)^{-3}+|z_0|^2 |z(t)|(1+t)^{-\frac{3}{2}}.
\end{equation}
\end{proposition}
This proposition will be formulated as different parts of Propositions ~\ref{Prop:KeyEst} and ~\ref{Prop:Majorants} in Appendix ~\ref{SEC:DetailInfo}.\

In the next two subsections we find and estimate the functions $S_{z}$ and $S_{\lambda}$ of ~\eqref{eq:IncreaseLambda} and ~\eqref{eq:DecayZ}.
\subsection{Definition of $S_{z}$ and its estimate}
In this part we define and estimate the function $S_{z}$ in ~\eqref{eq:DecayZ}.

It was proved in ~\cite{GaWe},  p. 293 (and can also be derived from ~\eqref{eq:z1} and ~\eqref{eq:z2}) that $z$ satisfies the equation
\begin{align}\label{eq:ZNequation}
\partial_{t}z+iE(\lambda)z=-\Gamma(z,\bar{z})z+\Lambda(z,\bar{z})z+\mathcal{K}
\end{align} where
$\Gamma(z,\bar{z})$ is positive definite and $\Lambda(z,\bar{z})$ is skew symmetric,
$\mathcal{K}=(\mathcal{K}_{1},\cdots, \mathcal{K}_{N})^{T}$ is defined as $$
\begin{array}{lll}
\mathcal{K}_n&:=&-[\partial_{t}p_{n}+iE(\lambda)\displaystyle\sum_{k+l=2,3}(k-l)P^{(n)}_{k,l}]-i [\partial_{t}q_{n}+iE(\lambda)\displaystyle\sum_{k+l=2,3}(k-l)Q^{(n)}_{k,l}]\\
& &-\left< JN(\vec{R},p,z)-\displaystyle\sum_{m+n=2,3} JN_{m,n},
\left(\begin{array}{lll}\eta_{n}\\-i\xi_{n}\end{array}\right)\right>+\Upsilon_{1,1}[\left< q\cdot \eta, \eta_{n}\right>-\left< i p\cdot\xi, \xi_{n}\right>]\\
& &+[\dot\gamma-\Upsilon_{1,1}][\langle (\beta+q)\cdot\eta, \eta_{n}\rangle-i\langle (\alpha+p)\cdot\xi, \xi_{n}\rangle]\\
& &\\
& &-\dot\lambda[a_1\langle \partial_{\lambda}^{2}\phi^\lambda,\eta_n\rangle+\langle (\alpha+p)\cdot\partial_{\lambda}\xi,\eta_n\rangle+ia_2\langle \partial_{\lambda}\phi^\lambda,\xi_n\rangle+i\langle (\beta+q)\cdot\partial_{\lambda}\eta,\xi_n\rangle]\\
& &+\left\langle \vec{R},\dot\lambda \left(
\begin{array}{lll}
\partial_{\lambda}\eta_{n}\\
-i\partial_{\lambda}\xi_{n}
\end{array}\right)-\dot\gamma \left(
\begin{array}{lll}
i\xi_n\\
\eta_n
\end{array}
\right)
\right\rangle.
\end{array}
$$

Recall that $|z|^2$ measures   the neutral mode mass. By direct computation we find
$$
\begin{array}{lll}
\frac{d}{dt}|z|^2&=&-2 z^*\Gamma(z,\bar{z})z+2Re z^*\cdot \mathcal{K}\ =\ -2 z^*\Gamma_0(z,\bar{z})z+S_{z}
\end{array}
$$ with the function $S_z$ defined by
\begin{equation}\label{eq:defSz}
S_{z}:=-2 z^*\Gamma(z,\bar{z})z+ 2 z^*\Gamma_0(z,\bar{z})z+2 Re z^*\cdot \mathcal{K}.
\end{equation}

We now estimate different terms on the right hand side of ~\eqref{eq:defSz}.
\begin{lemma}\label{LM:Junkz}
For $\sigma\geq 1$
\begin{equation}\label{eq:approxPos}
z^* \Gamma(z,\bar{z}) z=z^* \Gamma_0(z,\bar{z}) z+\mathcal{O}(\delta_{\infty}^{4\sigma-1}|z|^4).
\end{equation}
If $\sigma=1$ then
\begin{equation}\label{eq:ref1}
|\mathcal{K}|\lesssim \delta_{\infty} |z(t)|^{4}+\delta_{\infty}\|\langle x\rangle^{-4 }R_{\geq 4}\|_{2}^{2}+\delta_{\infty}|z(t)|\|\langle x\rangle^{-4 }R_{\geq 4}\|_{2};
\end{equation}
if $\sigma>1$ then
\begin{equation}\label{eq:ref2}
|\mathcal{K}|\lesssim |z(t)|^{2\sigma+1}+|z(t)|^{4}+\|\langle x\rangle^{-4 } R_{\geq 4}\|_{2}^{2}+|z(t)|\|\langle x\rangle^{-4 }R_{\geq 4}\|_{2}.
\end{equation}
\end{lemma}
 Equation ~\eqref{eq:approxPos} will be proved in Appendix ~\ref{sec:approxPos}, ~\eqref{eq:ref1} and ~\eqref{eq:ref2} will be incorporated into Proposition ~\ref{Prop:Majorants}.
By above estimates  we have
\begin{theorem}\label{THM:Zequation}
\begin{equation}\label{eq:estSz}
\int_{0}^{\infty}S_{z}(s)\ ds=o(|z_0|^2).
\end{equation}
\end{theorem}
\begin{proof}
The following two estimates together with Lemma ~\ref{LM:Junkz} are sufficient to prove the theorem:
\begin{equation}\label{eq:estzK}
\int_{0}^{\infty} |z|(s)|\mathcal{K}(s)|\ ds=o(|z_0|^2)
\end{equation}
and
\begin{equation}\label{eq:z02}
\int_{0}^{\infty}\delta_{\infty}^{4\sigma-2} |z|^4(s)\ ds\leq C|z_0|^2.
\end{equation}

We next focus on proving the two inequalities \eqref{eq:estzK} and \eqref{eq:z02}.
 The proof of ~\eqref{eq:z02} is relatively easy; it  follows  applying the estimate of $z$ in ~\eqref{eq:uppZ} and direct computation.

We now turn to ~\eqref{eq:estzK}.
For $\sigma=1$ we use ~\eqref{eq:ref1} and ~\eqref{eq:estR4} to obtain
$$|\mathcal{K}|\lesssim \delta_{\infty}|z|^4+ \delta_{\infty}|z_0|^2 |z| (1+t)^{-\frac{3}{2}}+\delta_{\infty}^2 |z|^3 |z_0|^2 +|z_0|^4\delta_{\infty} (1+t)^{-3}.$$
Together with the assumption on the initial condition $|z|\ll \delta_{\infty}=\mathcal{O}(\delta_{0})$ (see ~\eqref{eq:defDelta}) and ~\eqref{eq:uppZ} we have
\begin{equation}\label{eq:intRemaind}
\int_{0}^{\infty}|z||\mathcal{K}|(s)\ ds=o(|z_{0}|^2).
\end{equation}

\nit For the case, $\sigma>1$, the estimate is easier to obtain by applying the stronger condition $|z_0|\leq O(\delta_{0}^{C(\sigma)})=\mathcal{O}(\delta_{\infty}^{C(\sigma)})$ with $C(\sigma)$ being sufficiently large.
 This completes the proof.
\end{proof}

\subsection{Definition of $S_{\lambda}$ and its estimate}
After expanding the dispersive part $\vec{R}$ into the third order in $z$ and $\bar{z}$, we derive in Appendix ~\ref{sec:deriv} an equation for $\frac{d}{dt}\|\phi^{\lambda}\|_2^2$:
\begin{equation}\label{eq:DetailLambda}
\frac{d}{dt}\|\phi^{\lambda}\|_2^2
=- z^* \Gamma_0(z,\bar{z})z+S_{\lambda}
\end{equation}
with $S_{\lambda}$ defined as
\begin{equation}
S_{\lambda}:=
2\Psi\ +\ \left(\ 2\Pi_{2,2} + z^* \Gamma_0(z,\bar{z})z\ \right)
\ +\ 2\sum_{m+n=4,5;\ m\ne n}
\Pi_{m,n} \nn\end{equation}
here $\displaystyle\sum_{m+n=4,5}\Pi_{m,n}$ is a collection of fourth and fifth order terms
$$
\begin{array}{lll}
\displaystyle\sum_{m+n=4,5}\Pi_{m,n}&:=
&-\left\langle \displaystyle\sum_{m+n=4}JN_{m,n}\ ,\ \left(
\begin{array}{lll}
\phi^{\lambda}\\
0
\end{array}\right)
\right\rangle+\Upsilon_{1,1}\left\langle \displaystyle\sum_{m+n=2,3}R_{m,n},\left(
\begin{array}{lll}
0\\
\phi^{\lambda}
\end{array}\right)
\right\rangle+\Upsilon_{1,1}\left\langle q\cdot \eta,\phi^{\lambda}\right\rangle\\
&&\nn\\
& &+\langle \phi^{\lambda},\partial_{\lambda}\phi^{\lambda}\rangle\ \left[\partial_{z}a_1\cdot Z_{2,1}+\partial_{\bar{z}}a_1\cdot\overline{Z_{2,1}}\right],
\end{array}
$$
and $Z_{2,1}:=-\Gamma(z,\bar{z})z+\Lambda(z,\bar{z})z$ with the latter defined in ~\eqref{new-nf1};
and $\Psi$ is defined as
$$
\begin{array}{lll}
\Psi&:=&(\dot\gamma -\Upsilon_{1,1})\left[\ \left\langle \vec{R},
\left(
\begin{array}{lll}
0\\
\phi^{\lambda}
\end{array}
\right)
\right\rangle+\left\langle (\beta+q)\cdot\eta,\phi^{\lambda}\right\rangle\ \right]+
\Upsilon_{1,1}\left\langle R_{\geq 4},\left(
\begin{array}{lll}
0\\
\phi^{\lambda}
\end{array}
\right)\right\rangle\\
& &\\
& &+\dot\lambda\left\langle \vec{R},
\left(
\begin{array}{lll}
\partial_{\lambda}\phi^{\lambda}\\
0
\end{array}
\right)
\right\rangle-\dot\lambda a_{1}\langle \partial_{\lambda}^{2}\phi^{\lambda},\phi^\lambda\rangle-\dot\lambda (\alpha+p)\langle \partial_{\lambda}\xi,\phi^{\lambda}\rangle\\
& &+\langle \phi^{\lambda},\partial_{\lambda}\phi^{\lambda}\rangle[\partial_{t}a_1+iE(\lambda)\displaystyle
\sum_{m+n=2,3}(m-n)A_{m,n}^{(1)}-\partial_{z}a_1\cdot Z_{2,1}-\partial_{\bar{z}}a_1\cdot\overline{Z_{2,1}}]\\
& &+\left\langle JN-\displaystyle\sum_{m+n=2}^{4}JN_{m,n}\ ,\ \left(
\begin{array}{lll}
\phi^{\lambda}\\
0
\end{array}
\right)\right\rangle,
\end{array}
$$ Here we used the convention made in ~\eqref{eq:convention} and the definitions of  Appendix ~\ref{sec:NormalFormExp}.

To control these terms in $S_{\lambda}$ we use the following results: (Recall  $\delta_{\infty}=\|\phi^{\lambda_{\infty}}\|_2$, defined in ~\eqref{eq:defDelta}.)
\begin{lemma}\label{LM:junk}
\begin{align}\label{eq:KeyTermRef}
|\Psi| &\lesssim |z|\delta_{\infty}^{2\sigma-1}\|\langle x\rangle^{-4}R_{\geq 4}\|_{2}+\|\langle x\rangle^{-4}R_{\geq 4}\|_{2}^2+\delta_{\infty}^{2\sigma-1}|z|^5\\ &\nn\\
\label{eq:KeyTerm}
2\Pi_{2,2}+ z^* \Gamma_0(z,\bar{z})z &= \mathcal{O}(\delta_{\infty}^{4\sigma-1}|z|^4),
\\ &\nn\\
\label{eq:small}
\sum_{
\begin{subarray}{lll}
m+n=4,5,\\
m\not=n
\end{subarray}
}\int_0^\infty \Pi_{m,n}(s)\ ds &\lesssim \sum_{
\begin{subarray}{lll}
m+n=4,5,\\
m\not=n
\end{subarray}
}\int_0^\infty |\partial_{\lambda}\Pi_{m,n}||\dot\lambda|(s)+|\partial_{z}\Pi_{m,n}||\dot{z}+iE(\lambda)z|(s)\ ds\\ &+\ o(|z_0|^2) \nn
\end{align}
\end{lemma}

\nit The bound ~\eqref{eq:KeyTermRef} will be proved in Appendix ~\ref{SEC:DetailInfo}, ~\eqref{eq:small}  in Section ~\ref{sec:periodic} and ~\eqref{eq:KeyTerm} in Section ~\ref{sec:compare}. We now briefly present the ideas in the proof.
\begin{itemize}
\item[(1)] $\Psi$ is defined in term of functions $\dot\lambda,\ \dot\gamma$, $z$ and $\vec{R}$. They satisfy a coupled system. This system must be put in matrix form and decoupled.  In the end, we bound the functions $\dot\lambda$ and $\dot\gamma$ by the functions of $\vec{R}$ (or $R_{\geq 4}$) and $z$.
\item[(2)] All the integrands in ~\eqref{eq:KeyTerm} are of order $|z|^4$ in $z$ and $\bar{z}.$ What makes the terms different is the sizes of the coefficients. These depend smoothly on the functions $\phi^{\lambda}, \ \partial_{\lambda}\phi^\lambda,\ \xi,\ \eta$, which in turn depend smoothly  on the small parameter $\delta(\lambda)=\mathcal{O}(\delta_{\infty})$; see Proposition ~\ref{Prop:Parameters}. The estimate~\eqref{eq:KeyTerm} follows from a perturbation expansion in the parameter $\delta(\lambda).$
\item[(3)] For ~\eqref{eq:small} the important observation is that, if $m\ne n$,  then function $\Pi_{m,n}$ is a sum of the functions of the form $C(\lambda)z^{m}\bar{z}^n=C(\lambda)\prod_{k}z^{m_{k}}_{k}\prod_{l}\bar{z}^{n_{l}}_{l}$ with $m=\sum_{k}m_k,\ n=\sum_{l}n_{l}$. These are ``almost periodic" with period $2\pi(E(\lambda)(m-n))^{-1}\ne0$ since $z$ satisfies the equation $\dot{z}=-iE(\lambda)z+\cdots$. This non-trivial oscillation enables us to integrate by parts in the variable $s$ to derive smallness. The term $o(|z_0|^2)$ in ~\eqref{eq:small} is due to a  boundary term obtained in this way.
\end{itemize}
Based on the estimates in Lemma ~\ref{LM:junk} we will prove
\begin{theorem}\label{THM:KeyTerm}
$S_{\lambda}$ satisfying the estimate in ~\eqref{eq:IncreaseLambda}, i.e.
$$\int_{0}^{\infty} S_{\lambda}(s)\ ds=o(|z_0|^2).$$
\end{theorem}
\begin{proof}
The result follows directly from Lemma ~\ref{LM:junk} and the following two estimates:
\begin{equation}\label{eq:KeyTerm2}
\int_{0}^{\infty}|\Psi|(s)\ ds=o(|z_0|^2);
\end{equation}
\begin{equation}\label{eq:KeyTerm3}
\sum_{
\begin{subarray}{lll}
m+n=4,5,\\
m\not=n
\end{subarray}
}\int_0^\infty |\partial_{\lambda}\Pi_{m,n}||\dot\lambda|(s)+|\partial_{z}\Pi_{m,n}||\dot{z}+iE(\lambda)z|(s)\ ds=o(|z_0|^2)\ .
\end{equation}

We next prove estimates \eqref{eq:KeyTerm2} and  \eqref{eq:KeyTerm3}.
In the proof we consider the case $\sigma=1$. That of $\sigma>1$  is different, but easier due to  the stronger condition $|z_{0}|\leq\delta_{\infty}^{C(\sigma)}$ for some sufficiently large $C(\sigma)$, and hence omit the details.

We start with ~\eqref{eq:KeyTerm2}, by estimating three different terms in the estimate of $\Psi$ in ~\eqref{eq:KeyTermRef} on the right hand side.
By applying the estimates for $z$ in ~\eqref{eq:uppZ}
\begin{equation}\label{eq:estPsi}
\int_{0}^{\infty}\delta_{\infty}|z|^5(s)\ ds\lesssim \int_{0}^{\infty}\delta_{\infty}(|z_0|^{-2}+\delta_\infty^2 s)^{-\frac{5}{2}}\ ds=\frac{2}{3}\delta_{\infty}^{-1}|z_0|^3=o(|z_0|^2)
\end{equation} where the assumption on the initial condition $|z_0|\ll \delta_{0}=\cO(\delta_{\infty})$ was used.

By the estimate of $R_{\geq 4}$ in ~\eqref{eq:estR4} and $|z(t)|$ in ~\eqref{eq:uppZ}
$$
\begin{array}{lll}
& &\int_{0}^{\infty} \delta_{\infty}|z(s)|\|\langle x\rangle^{-4 }R_{\geq 4}(s)\|_{2}\ ds\\
&\lesssim & \delta_{\infty}|z_{0}|^{2}\int_{0}^{\infty}(1+s)^{-\frac{3}{2}}(|z_{0}|^{-2}+\delta_{\infty}^{2} s)^{-\frac{1}{2}}\ ds+
\delta_{\infty}^2|z_{0}|^{2}\int_{0}^{\infty}(|z_{0}|^{-2}+\delta_{\infty}^{2} s)^{-\frac{3}{2}}\ ds\\
&=&o(|z_0|^2).
\end{array}
$$
The third term can be similarly estimated.
Assembling the above estimates yields $$\int_{0}^{\infty}|\Psi(t)|\ dt=o(|z_0|^2).$$

To prove ~\eqref{eq:KeyTerm3} we use the equations for $\dot{z}$ and $\dot\lambda$ in ~\eqref{eq:ZNequation} and ~\eqref{eq:estJunk} to find that if $m+n=4,5$ and $m\not=n$ then
$$|\partial_{\lambda}\Pi_{m,n}||\dot\lambda(s)|+|\partial_{z}\Pi_{m,n}||\dot{z}+iE(\lambda)z|(s)
\lesssim |z||\dot\lambda|+ |z|^6+|z|^3|\mathcal{K}|.
$$  Using the estimates in ~\eqref{eq:ref1} and ~\eqref{eq:ref2} for $\mathcal{K}$, and the estimate ~\eqref{eq:estJunk} and similar techniques above we prove ~\eqref{eq:KeyTerm3}. This is  straightforward, but tedious, hence we omit the details.
\end{proof}
\begin{remark}\label{remark:remark3}
%

In the last step of ~\eqref{eq:estPsi} we used $|z_0|\ll \delta_\infty$ to control $\int_{0}^{\infty}\delta_{\infty}|z|^5(s)\ ds.$ If $\sigma=1$ this can be relaxed to $|z_{0}|\leq \|\phi^{\lambda_0}\|_2=\mathcal{O}( \delta_{\infty})$ by inspecting closely the terms forming $\delta_{\infty}|z|^5$. The term actually is a part of $\left\langle JN_{\geq 5},\left(
\begin{array}{lll}
\phi^{\lambda}\\
0
\end{array}
\right)\right\rangle $, and can be written as $\displaystyle\sum_{m+n=5}K_{m,n}$ for some properly defined $K_{m,n}$. To evaluate $\displaystyle\sum_{m+n=5}\int_{0}^{\infty} K_{m,n}(s)\ ds$ we observe that $K_{m,n},\ m+n=5,$ are ``almost periodic" as $\Lambda_{m,n}$ of ~\eqref{eq:small}. Hence by integrating by parts as in the proof of ~\eqref{eq:small} it is easy to obtain the desired estimate $$\sum_{m+n=5}\int_{0}^{\infty} K_{m,n}(s)\ ds=o(|z_0|^2).$$

Note that the terms $K_{m,n},\ m+n=5,$ may not be well defined if $\sigma\not\in \mathbb{N}.$
\end{remark}

\section{Extension to the case of nearly degenerate neutral modes}\label{SEC:summary}
In ~\cite{GaWe} and the main part of the present paper we have proved that if the neutral modes are degenerate and their eigenvalues are sufficiently close to the essential spectrum then the ground state is asymptotically stable and its mass will grow by half of that of the neutral modes.

 In what follows we extend the results to the cases where the neutral modes are nearly degenerate, {\it i.e.}  a cluster of approximately equal eigenfrequencies. For technical simplicity, we consider the case of  cubic nonlinearity, $\sigma=1$. The main result is Theorem ~\ref{THM:main3} below. The key ideas of the proof will be presented after its statement.
\subsection{New assumptions on the spectrum and definition of FGR}
  As in Subsection ~\ref{Vassumptions}
we assume that the linear operator $-\Delta+V$ has
the following properties:
\begin{enumerate}
\item[(V1)]
$V$ is real valued and decays sufficiently rapidly, {\it e.g.} exponentially,  as $|x|$ tends to infinity.\\
\item[(V2)] The linear operator $-\Delta+V$ has $N+1$ (counting multiplicity if degenerate) eigenvalues $e_{0}, \ e_{k},\ k=1,2,\cdot,N,$ with $e_{0}<e_{k},$\\
 $e_{0}$ is the lowest eigenvalue with
ground state $\phi_{lin}>0$, the eigenvalues $\{e_{k}\}_{k=1}^{N}$ are possibly degenerate
with eigenvectors
$\xi_{1}^{lin},\xi_{2}^{lin},\cdot\cdot\cdot,\xi_{N}^{lin}.$
\item[(V3)] Moreover, for any $k=1,2,\cdots, N$
 we assume
 \begin{equation}  2e_k-e_0>0.
 \end{equation}
\end{enumerate}
Then the nonlinear equation ~\eqref{eq:NLS} admits a family of ground states solution $e^{i\lambda t}\phi^{\lambda}$ with properties as described in Proposition ~\ref{bif-of-gs}. The linearized operators about the ground states,  $L(\lambda)$,  takes the same form as in ~\eqref{eq:opera}. The excited states of $-\Delta+V$ bifurcate to the neutral modes $\left(
\begin{array}{lll}
\xi_k\\
\pm i\eta_k
\end{array}
\right)$ of $L(\lambda)$ with eigenvalues $\pm i E_{k}(\lambda),\ k=1,\cdots, N.$ The ground states $\phi^{\lambda}$ and neutral modes satisfy all the estimates in Lemma ~\ref{LM:NearLinear} and the estimates ~\eqref{eq:LambdaPhi2}-~\eqref{eq:asympto}.

Assumption {\bf (SA)} on the spectrum of $L(\lambda)$ is generalized, in the case where near-degeneracy is admitted, as
\begin{enumerate}
 \item[{\bf (SA)}] The discrete spectrum of the linearized operator $L(\lambda)$ consists of:
 the  eigenvalue $0$
    with
generalized eigenvectors $\left(
\begin{array}{lll}
0\\
\phi^{\lambda}
\end{array}
\right)$ and $\left(
\begin{array}{lll}
\partial_{\lambda}\phi^{\lambda}\\
0
\end{array}
\right)$ and eigenvalues $\pm iE_k (\lambda),\ E_k(\lambda)>0,\ k=1,2,\cdots,N$.
\end{enumerate}

A consequence of nonzero neutral mode frequency-differences is a  slightly different system for the neutral mode amplitudes, $z(t)$. The solution $\psi(t)$ is decomposed as in ~\eqref{eq:decom}. Following the  same procedure as in ~\cite{GaWe}, we derive
\begin{equation}
\partial_{t}z =-iE(\lambda) z -\Gamma( z ,\bar{ z }) z +\Lambda( z ,\bar{ z })
z\ +\cdots
\end{equation}
where $E(\lambda)=Diag[E_{1}(\lambda),\cdots,\ E_{n}(\lambda)]$ is a diagonal $N\times N$ matrix, $\Gamma$ is symmetric and $\Lambda$ is skew symmetric.

We now describe the matrix $\Gamma$, which  takes a different form from the degenerate case:
Define vector functions $G(k,m),\ k,m=1,2,\cdot\cdot\cdot, N,$ as
\begin{equation}
G(k,m):=\left(
\begin{array}{lll}
B(k,m)\\
D(k,m)
\end{array}
\right)
\end{equation} with the functions $B(k,m)$ and $D(k,m)$ defined as $$\begin{array}{lll}
B(k,m)&:=&-i\phi^{\lambda} \ \ \left[\ z_m \xi_{m}\ \eta_{k}+z_m\eta_m\ \xi_{k}\ \right]\ , \\
D(k,m)&:=&-\phi^{\lambda}
\left[\ 3z_m\xi_m\ \xi_{k}-z_m\eta_m\ \eta_{k}\ \right]\ .
\end{array}$$
In terms of the column 2-vector, $G(k,m)$, we define
a $N \times N$ matrix $Z(z,\bar{z})$ as
\begin{equation}
Z(z,\bar{z})=(Z^{(k,l)}(z,\bar{z})),\ \ 1\le k,l\le N
\end{equation} and
\begin{equation}
Z^{(k,l)}\ =\ -\left\langle
\sum_{m=1}^{N}(L(\lambda)+iE_l (\lambda)+iE_{m}(\lambda)-0)^{-1}P_{c}G(l,m), iJP_c\sum_{m=1}^{N}G(k,m)\right\rangle
\end{equation}
Finally, we define $\Gamma(z,\bar{z})$ as follows:
\begin{equation}
\Gamma(z,\bar{z})\ :=\ \frac{1}{2}[Z(z,\bar{z})+Z^{*}(z,\bar{z})].
\end{equation}

We shall require the following Fermi Golden Rule hypothesis: Let $P_{c}^{lin}$ be the projection onto the essential spectrum of $-\Delta+V$ then
\begin{enumerate}
\item[{\bf (FGR)}] We assume there exists a constant $C>0$ such that $$-Re\langle i[-\Delta+V+\lambda-2E_1(\lambda)-i0]^{-1}P_{c}^{lin}\phi_{lin}(z\cdot \xi^{lin})^{2},\phi_{lin}(z\cdot \xi^{lin})^{2}\rangle\ge C|z|^4$$ for any $z\in \mathbb{C}^{N}.$
\end{enumerate} The assumption FGR implies that there exist constants $C_1>0$ and $\delta_0>0$ such that if $\sup_{k,l}|E_{k}(\lambda)-E_{l}(\lambda)|\leq \delta_{0}$ then for any $z\in \mathbb{C}^{N}$
\begin{equation}\label{eq:FGR2}
z^*\ \Gamma(z,\bar{z})\ z\geq C_1\|\phi^{\lambda}\|_{\infty}^{2} |z|^{4}.
\end{equation}

We now introduce the leading order contribution to $\Gamma(z,\bar{z})$.
 For each fixed $z$ we use the fact $|\lambda+e_0|$ being small and use ~\eqref{eqn:perturb} and ~\eqref{eq:GoToNear} to find the leading term in $z^*\ \Gamma(z,\bar{z})\ z$ is $z^* \Gamma_0(z,\bar{z})\ z$ defined as
\begin{equation}\label{eq:Gamma2}
\begin{array}{lll}
& &z^* \Gamma_0(z,\bar{z})\ z\\
&=&-
8\delta^{2}(\lambda)\Re\langle i\displaystyle\sum_{m,n\leq N}[-\Delta+V+\lambda-E_m(\lambda)-E_n-i0]^{-1}P_{c}^{lin}\phi^{lin}(z_m \xi_{m}^{lin})(z_n \xi_n^{lin}),
\phi^{lin}(z\cdot\xi^{lin})^{2}\rangle.
\end{array}
\end{equation}
\subsection{Main Theorem in nearly degenerate case  and strategy of proof}
 Recall that we only consider the case $\sigma=1$, i.e. the cubic nonlinearity.
\begin{theorem}\label{THM:main3}
There exists a constant $\delta_0$ independent of the initial condition $\psi_0$ of ~\eqref{eq:NLS} such that if $\displaystyle\max_{k,l}|E_k(\lambda)-E_l(\lambda)|\leq \delta_0$ then all the results in Theorem ~\ref{THM:MassTransfer} hold with $ z^* \Gamma_0(z,\bar{z})\ z$ replaced by the expression in ~\eqref{eq:Gamma2}. Moreover all the remainder estimates in ~\eqref{eq:IncreaseLambda}-~\eqref{eq:Mass} hold and are independent of the size of $\delta_0.$
\end{theorem}
In the next we show how to recover all the estimates. To simplify the treatment we only consider the case $N=2$ with eigenfrequencies $E_1(\lambda)$ and $E_2(\lambda).$

There are some differences between the degenerate and the nearly degenerate cases. Among them, the most outstanding one are terms, which previously vanished identically, which now need to be estimated.  These terms  include, for example,  $\langle ImN_{1,1},\phi^{\lambda}\rangle$, which was proved to be zero in  ~\cite{GaWe} Lemma 9.4, p. 291 which we see below is non-zero in the nearly degenerate case. To treat such terms, the key observation is that these terms have a factor $E_1(\lambda)-E_2(\lambda)$ in their coefficient  enabling  us to  re-express $[E_1(\lambda)-E_2(\lambda)]z_1 \bar{z}_2$ as $-i\frac{d}{dt}z_1 \bar{z}_2+o(|z|^4)$. Thus, these terms can be removed via integration by parts and a redefinition of the normal form transformation.
\subsection{Normal Form Transformation and Asymptotic Stability of Ground States}
We decompose the initial condition in exactly the same way as in ~\eqref{eq:decom}. All equations ~\eqref{eq:decom}-~\eqref{eq:Gamma11}, ~\eqref{eq:gamma} and ~\eqref{eq:lambda} hold. The equations for $\dot{z}$ are slightly different since  $z_j$ each have  different associated frequencies. Consequently instead of ~\eqref{eq:z1} and ~\eqref{eq:z2} we have $$\partial_{t}(\alpha_{n}+p_{n})-E_n(\lambda)(\beta_{n}+q_{n})+\cdots,\ \ \
\partial_{t}(\beta_{n}+q_{n})+E_n(\lambda)(\alpha_{n}+p_{n})+\cdots,$$
requiring a different near-indentity / normal form transformation.

To illustrate the main difference in the calculation we study the equation for $\dot\lambda.$ Recall that the function $\dot\lambda$ satisfies the equation $$\dot\lambda+\partial_{t}a_1=-\frac{1}{\langle \phi^{\lambda},\partial_{\lambda}\phi^{\lambda}\rangle}\langle ImN(\vec{R},z),\phi^{\lambda}\rangle+\cdots$$ and we want to remove the second and third order terms in $z$ and $\bar{z}$ from the equation by defining some polynomial $a_1$ in $z$ and $\bar{z}$: $$a_1:=\sum_{m+n=2,3}A_{m,n}^{(1)}.$$ In the degenerate case we set $A_{1,1}^{(1)}=0$ (see ~\eqref{eq:pkmn}) due to the fact $\langle ImN_{1,1},\phi^{\lambda}\rangle=0.$ When the latter no longer holds $A_{1,1}^{(1)}$ has to be redefined. Following steps in ~\cite{GaWe}, p. 291, we use the fact $\left(
\begin{array}{lll}
\xi_{n}\\
i\eta_{n}
\end{array}
\right),\ n=1,2,$ are eigenvectors of $L(\lambda)$ to obtain
\begin{equation}\label{eq:IMN11}
\begin{array}{lll}
\langle ImN_{1,1},\phi^{\lambda}\rangle&:=&\frac{1}{2i}\displaystyle\sum_{n=1}^{2}\sum_{m=1}^{2}\bar{z}_{n}z_{m}\int (\phi^{\lambda})^2 (\xi_n\eta_m-\xi_m\eta_n)\\
&=&\frac{1}{4i}\displaystyle\sum_{n=1}^{2}\sum_{m=1}^{2}\bar{z}_{n}z_{m}[\langle (L_{-}-L_{+})\xi_n,\eta_m\rangle-\langle (L_{-}-L_{+})\xi_m,\eta_{n}\rangle]\\
&=&\frac{1}{4i}[E_1(\lambda)-E_2(\lambda)][z_1\bar{z}_2-z_2\bar{z}_1][\langle \eta_1,\eta_2\rangle+\langle \xi_1,\xi_2\rangle].
\end{array}
\end{equation}
To remove ~\eqref{eq:IMN11} from the equation of $\dot\lambda$ we define
\begin{align}\label{eq:A11}
A_{1,1}^{(1)}&:=&-\frac{1}{4\langle \phi^{\lambda},\partial_{\lambda}\phi^{\lambda}\rangle}[z_1\bar{z}_2+z_2\bar{z}_1][\langle \eta_1,\eta_2\rangle+\langle \xi_1,\xi_2\rangle]\
= \mathcal{O}\left(\|\phi^{\lambda}\|_2^{2}\ |z|^2\right),
\end{align}
where in the last step the estimate \eqref{eq:asympto} and the fact $\xi_1^{lin}\perp \xi_{2}^{lin}$ are used.

For the other terms in $a_1$ we only re-define $A_{2,0}^{(1)}$ to illustrate the differences: Decompose $\frac{1}{\langle \phi^{\lambda},\partial_{\lambda}\phi^{\lambda}\rangle}\langle ImN_{2,0},\phi^{\lambda}\rangle$ as $K_1 z_1^2+K_2 z_1 z_2+K_3 z_2^2,$ then instead of the definition in ~\eqref{eq:A1} we define $$A_{2,0}^{(1)}=-\frac{i}{2E_1(\lambda)}K_1 z_1^2-\frac{i}{E_1(\lambda)+E_2(\lambda)}K_2 z_1z_2-\frac{i}{2E_2(\lambda)}K_2 z_2^2.$$

The new normal forms enable the proof of asymptotic stability of the ground states to go through,  as well as  all results in Section ~\ref{sec:CorrectionNormal}, {\it i.e.} all the statements in Theorem ~\ref{THM:MassTransfer} except (A) and (D), which we discuss in the next subsection.

\subsection{Equipartition of Energy}
In this subsection we recover the Statements (A) and (D). Most of the arguments proved in the degenerate regime still hold. As presented above certain newly-nonzero terms enter different places. In what follows we present the strategy to handle such terms.

To illustrate the idea we only study one term whose counterpart is $H_{2,2}$ in ~\eqref{eq:H22}
$$D:=\sum_{
\begin{subarray}{lll}
m+n=2\\
m'+n'=2
\end{subarray}}D(m,n,m',n')
$$ with $D(m,n,m',n')$ being a real function:
$$D(m,n,m',n'):=Re\langle i(-\Delta+V+\lambda+m E_1(\lambda)+n E_2(\lambda))^{-1}P_{c}^{lin}\phi^{\lambda}(z_1 \xi_1)^{m} (z_2 \xi_2)^{n}, \phi^{\lambda}(z_1\xi_1)^{m'}(z_2\xi_2)^{n'}\rangle.$$ If $E_1(\lambda)=E_{2}(\lambda)$ then we use the observation in ~\eqref{eq:H22} to prove $$
D(m,n,m',n')=\overline{D(m,n,m',n')}=-D(m',n',m,n)\ \text{implies}\ D=0.
$$ When $E_1(\lambda)\not=E_2(\lambda)$ we use the following result to recover the desired estimate
\begin{lemma}
\begin{equation}\label{eq:switch}
\int_{0}^{\infty}D(s)\ ds=o(|z_0|^2).
\end{equation}
\end{lemma}
\begin{proof}
The facts that $D(m',n',m,n)$ is real and $(-\Delta+V+\lambda+m' E_1(\lambda)+n' E_2(\lambda))^{-1}$ is self-adjoint imply
$$
\begin{array}{lll}
D(m',n',m,n)&=&\overline{D(m',n',m,n,)}\\
&=&-Re\langle i(-\Delta+V+\lambda+m' E_1(\lambda)+n' E_2(\lambda))^{-1}P_{c}^{lin}\times\\
& &\phi^{\lambda}(z_1 \xi_1)^{m} (z_2 \xi_2)^{n}, \phi^{\lambda}(z_1\xi_1)^{m'}(z_2\xi_2)^{n'}\rangle.
\end{array}
$$
The crucial step is to find the presence of $E_1(\lambda)-E_2(\lambda)$ in the coefficient:
\begin{equation}\label{eq:realPart}
\begin{array}{lll}
D(m,n,m',n')+D(m',n',m,n)
&=&-[(m-m')E_1(\lambda)+(n-n')E_2(\lambda)] Re H\\
&=&-[m-m'][E_{1}(\lambda)-E_{2}(\lambda)]ReH
\end{array}
\end{equation} where $H$ is defined as
$$
\begin{array}{lll}
H&:=&\langle i[-\Delta+V+\lambda+m' E_1(\lambda)+n' E_2(\lambda)]^{-1}[-\Delta+V+\lambda+m E_1(\lambda)+n E_2(\lambda)]^{-1}\times\\
& &P_{c}^{lin}\phi^{\lambda}(z_1 \xi_1)^{m} (z_2 \xi_2)^{n}, \phi^{\lambda}(z_1\xi_1)^{m'}(z_2\xi_2)^{n'}\rangle,
\end{array}
$$ and in the last step the fact $m+n=m'+n'=2$ is used.

~\eqref{eq:realPart} enables us to use the same trick as in ~\eqref{eq:small}, namely integration by parts, to obtain the desired estimate
\begin{align}
&\int_0^{\infty}D(m,n,m',n')+D(m',n',m,n) ds\\
&=\int_{0}^{\infty}\frac{d}{ds}\ \Re\ \left( iH\ \right)\ ds+\int_0^{\infty} \mathcal{O}(|\dot\lambda||z|^4+\|\phi^{\lambda}\|_{2}^{2}|z|^6)\ ds
+o(|z_0|^2).
\end{align}
The proof is complete.
\end{proof}

\nit In summary, as outlined above,  all the estimates obtained in the degenerate can be proved in the nearly degenerate case.

\appendix
\section{Decomposition of the solution $\psi$}\label{sec:decom}
This section is based upon ~\cite{GaWe}, pp. 286-287. As stated in ~\eqref{Decom}, for any time $t$ the solution $\psi(x,t)$ can be decomposed as
\begin{equation}\label{eq:decom}
\begin{array}{lll}
\psi(x,t)&=&e^{i\gamma(t)}e^{i\int_0^t \lambda(s)ds}[\phi^{\lambda(t)}(x)+a_{1}(t)\partial_{\lambda}\phi^{\lambda(t)}(x)+ia_2(t) \phi^{\lambda(t)}(x)\\
 & &+\displaystyle\sum_{n=1}^N [\alpha_n(t)+p_n(t)]\xi_n^{\lambda(t)}(x)+i\sum_{n=1}^N [\beta_n(t)+q_n(t)]\eta_n^{\lambda(t)}(x)+R(x,t)]
\end{array}
\end{equation} for some polynomials $a_1,\ a_2,\ p_n,\ q_n$ (will be defined explicitly in Appendix ~\ref{sec:NormalFormExp}) and the function $R$ satisfies the symplectic orthogonality conditions ~\eqref{s-Rorthogonal}. By this $\vec{R}:=\left(
 \begin{array}{lll}
 R_1\\
 R_2
 \end{array}
 \right):=\left(
\begin{array}{lll}
ReR\\
ImR
\end{array}
\right)\in P_{c}(L(\lambda)) L^2$ satisfies the equation
\begin{equation}\label{RAfProj}
\frac{d}{dt}\vec{R}\ =\
 L( \lambda(t) )\vec{R}\ -\ P_{c}^{\lambda(t)}J\vec{N}(\Vec{R},z)\ +\ L_{(\dot\lambda,\dot\gamma)}\vec{R}\ +\ \mathcal{G}.
\end{equation}
Here,
\begin{align}
J\vec{N}(\Vec{R},z)\ &:=\ \left(
\begin{array}{lll}
ImN(\Vec{R},z)\\
-ReN(\Vec{R},z)
\end{array}
\right),\label{JvecN}\\
ImN(\Vec{R},z)&:=
|\phi^{\lambda}+I_{1}+iI_{2}|^{2\sigma}I_{2}-(\phi^{\lambda})^{2\sigma}I_{2},\nn\\
Re N(\Vec{R},z)&:=
[|\phi^{\lambda}+I_{1}+iI_{2}|^{2\sigma}-(\phi^{\lambda})^{2\sigma}](\phi^{\lambda}+I_{1})
-2\sigma (\phi^{\lambda})^{2\sigma}I_{1},\nn\\
I_{1}\ &:=\ \alpha\cdot\xi+a_{1}\partial_{\lambda}\phi^{\lambda}+
p\cdot\xi+R_{1},\nn \\
I_{2}\ &:=\ \beta\cdot \eta\ +\ a_{2}\phi^{\lambda}\ +\ q\cdot\eta+\ R_{2}\nn .
\nn\end{align}
The operator $L_{(\dot\lambda,\dot\gamma)}$ and the vector function $\mathcal{G}$ are defined
as
\begin{align}
& L_{(\dot\lambda,\dot\gamma)}\ :=\ \dot\lambda
(\partial_\lambda P_{c}^{\lambda(t)})+\dot\gamma P_{c}^{\lambda(t)}J,\label{Lgdld}\\
& \mathcal{G}\ :=\ P_{c}^{\lambda(t)}\left(
\begin{array}{lll}
[\dot\gamma-\Upsilon_{1,1}] (\beta+q)\cdot\eta-\dot{\lambda}a_{1}
\partial_{\lambda}^{2}\phi^{\lambda}
-\dot\lambda (\alpha+p)\cdot\partial_{\lambda}\xi\\
-[\dot\gamma-\Upsilon_{1,1}] (\alpha+p)\cdot\xi-\dot{\lambda}a_{2}\phi^{\lambda}_{\lambda}-\dot\lambda
(\beta+q)\cdot\partial_{\lambda}\eta
\end{array}
\right)+\Upsilon_{1,1}P_{c}^{\lambda(t)}\left(
\begin{array}{lll}
(\beta+q)\cdot\eta\\
-(\alpha+p)\cdot\xi
\end{array}
\right)\label{Gdef}
\end{align} where, $\Upsilon_{1,1}$ is defined as
\begin{equation}\label{eq:Gamma11}
\Upsilon_{1,1}:=\frac{\langle
(\phi^{\lambda})^{2\sigma-1}[(2\sigma^2+\sigma)|z\cdot\xi|^2+\sigma |z\cdot\eta|^2],\partial_{\lambda}\phi^{\lambda}\rangle}{2\langle\phi^{\lambda},\partial_{\lambda}\phi^{\lambda}\rangle}
\end{equation} with $$z:=(z_{1},\cdots, z_{N})^{T},\ z_{n}:=\alpha_{n}+i\beta_{n},\ n=1,\cdots, N,$$ and $$\xi:=(\xi_{1},\cdots,\xi_{N})^{T},\ \eta:=(\eta_{1},\cdots,\eta_{N})^{T}.$$

By the orthogonality conditions ~\eqref{s-Rorthogonal} and ~\eqref{eq:Orthogonality}  we derive equations for $
\dot\lambda,\ \dot\gamma$ and $z_{n}=\alpha_n+i\beta_n,\ n=1,\dots,N, $ as
\begin{equation}\label{eq:z1}
\partial_{t}(\alpha_{n}+p_{n})-E(\lambda)(\beta_{n}+q_{n})+\langle ImN(\Vec{R},z),\eta_{n}\rangle=F_{1n};
\end{equation}
\begin{equation}\label{eq:z2}
\partial_{t}(\beta_{n}+q_{n})+E(\lambda)(\alpha_{n}+p_{n})-\langle ReN(\Vec{R},z),\xi_{n}\rangle=F_{2n};
\end{equation}
\begin{equation}\label{eq:gamma}
\dot\gamma+\partial_{t}a_{2}-a_{1}-\frac{1}{\langle
\phi^{\lambda},\phi^{\lambda}_{\lambda}\rangle}\langle
ReN(\Vec{R},z),\partial_{\lambda}\phi^{\lambda}\rangle=F_{3};
\end{equation}
\begin{equation}\label{eq:lambda}
\dot\lambda+\partial_{t}a_{1}+\frac{1}{\langle
\phi^{\lambda},\phi^{\lambda}_{\lambda}\rangle}\langle
ImN(\Vec{R},z),\phi^{\lambda}\rangle=F_{4}
\end{equation}
where the scalar functions $F_{j,n},\ j=1,2,\ n=1,2,\cdots,N,\  F_3,
 F_4,$ are defined as
\begin{align}
F_{1n}\ &=\ \dot\gamma\langle
(\beta+q)\cdot\eta,\eta_{n}\rangle-\dot{\lambda}a_{1}\langle
\partial_{\lambda}^{2}\phi^{\lambda},\eta_{n}\rangle
-\dot\lambda\langle
(\alpha+p)\cdot\partial_{\lambda}\xi,\eta_{n}\rangle
-\dot\gamma\langle
R_{2},\eta_{n}\rangle+\dot\lambda\langle
R_{1},\partial_{\lambda}\eta_{n}\rangle,\nn\\
F_{2n}\ &=\ -\dot\gamma\langle
(\alpha+p)\cdot\xi,\xi_{n}\rangle
-\dot{\lambda}a_{2}\langle\phi^{\lambda}_{\lambda},
\xi_{n}\rangle
-\dot\lambda\langle(\beta+q)\cdot\partial_{\lambda}\eta,\xi_{n}\rangle+\dot\gamma\langle
R_{1},\xi_{n}\rangle+\dot\lambda\langle
R_{2},\partial_{\lambda}\xi_{n}\rangle,\nn\\
F_{3}\ &=\ \frac{1}{\langle
\phi^{\lambda},\phi^{\lambda}_{\lambda}\rangle}
\left[\ \dot\lambda \langle
R_{2},\phi_{\lambda\lambda}^{\lambda}\rangle -\dot\gamma\langle
R_{1},\phi_{\lambda}^{\lambda}\rangle
-\langle\dot\gamma(\alpha+p)\cdot\xi+
\dot{\lambda}a_{2}\phi^{\lambda}_{\lambda}
+\dot\lambda (\beta+q)\cdot\partial_{\lambda}\eta,\phi^{\lambda}_{\lambda}\rangle\ \right],\nn\\
F_{4}\ &=\ \frac{1}{\langle
\phi^{\lambda},\phi^{\lambda}_{\lambda}\rangle}
\left[\ \dot\lambda\langle
R_{1},\phi_{\lambda}^{\lambda}\rangle +\dot\gamma\langle
R_{2},\phi^{\lambda}\rangle
+\langle
\dot\gamma (\beta+q)\cdot \eta-\dot{\lambda}a_{1}
\partial_{\lambda}^{2}\phi^{\lambda}-\dot\lambda
(\alpha+p)\cdot\partial_{\lambda}\xi,\phi^{\lambda}\ \rangle\ \right].
\end{align}

\section{The Normal Form Expansion}\label{sec:NormalFormExp}
All the results in this Appendix, except the definitions of $R_{m,n}, \ JN_{m,n},\ m+n=3$, are taken from ~\cite{GaWe}. Specifically the definitions of $a_1, \ a_2,\ p_{k},\ q_{k},\ k=1,2,\cdots, N,$ are taken from (9-12) and (9-13), p. 288; the definitions of $R_{m,n}, \ JN_{m,n},\ m+n=2,$ from (9-18)-(9-21), p. 290.

Before defining various functions we introduce the following convention on notations:
\textbf{we always use $z$
to stand for a complex $N$-dimensional vector $z=(z_{1},z_{2},\cdot\cdot\cdot,z_{N})$ and an upper case letter or a
Greek letter with two subindices, for example $Q_{m,n}$, to represent
\begin{equation}\label{eq:convention}
Q_{m,n}(\lambda)=\sum_{
 |a|=m,\ |b|=n
}q_{a,b}(\lambda)\
 \prod_{k=1}^{N}z_{k}^{a_{k}}\bar{z}_{k}^{b_{k}},
\end{equation} where $a,b\in \mathbb{N}^N$, $|a|:=\displaystyle\sum_{k=1}^{N}a_{k}$.
 We refer to this kind term as $(m,n)$ term.}

In what follows we define $R_{m,n},\ m+n=2,3,$ $JN_{m',n'},\ m'+n'=2,3,4,$ and the polynomials $a_1,\ a_2,\ p_k,\ q_k,\ k=1,2,\cdots, N,$ by induction.
\begin{flushleft}
{\bf{Definitions of Polynomials $a_1,\ a_2,\ p_k,\ q_k,\ k=1,2,\cdots, N$}}
\end{flushleft}
We define the polynomials $a_{1}$, $a_{2}$, $p_{k}$ and $q_{k},\
k=1,2,\cdot\cdot\cdot,N,$ in (~\ref{eq:decom}) as
\begin{equation}\label{eq:pkmn}
\begin{array}{lll}
a_{k}(z,\bar{z})&:=&\displaystyle\sum_{m+n=2,3,m\not=n}A^{(k)}_{m,n}(\lambda),
\
k=1,2,\\
p_{k}(z,\bar{z})&:=&\displaystyle\sum_{m+n=
2,3}P^{(k)}_{m,n}(\lambda),
\ k=1,2,\cdot\cdot\cdot,N,\\
q_{k}(z,\bar{z})&:=&\displaystyle\sum_{m+n=
2,3}Q^{(k)}_{m,n}(\lambda),
\ k=1,2,\cdot\cdot\cdot,N\\
\end{array}
\end{equation}
where the terms on the right hand side take the form:
\begin{equation}\label{eq:A1}
\begin{array}{lll}
2iE(\lambda)A_{2,0}^{(1)}&:=&\frac{1}{\langle
\phi^{\lambda},\partial_{\lambda}\phi^{\lambda}\rangle}\langle
N^{Im}_{2,0},\phi^{\lambda}\rangle;\\
3iE(\lambda)A_{3,0}^{(1)}&:=&\frac{1}{\langle
\phi^{\lambda},\partial_{\lambda}\phi^{\lambda}\rangle}\langle
N^{Im}_{3,0},\phi^{\lambda}\rangle;\\
i E(\lambda)A_{2,1}^{(1)}&:=&\frac{1}{\langle
\phi^{\lambda},\partial_{\lambda}\phi^{\lambda}\rangle}[\langle
N^{Im}_{2,1},\phi^{\lambda}\rangle-\frac{i}{2}\Upsilon_{1,1}\langle
z\cdot\eta,\phi^{\lambda}\rangle];
\end{array}
\end{equation}
\begin{equation}
\begin{array}{lll}
-2iE(\lambda)A_{2,0}^{(2)}-A_{2,0}^{(1)}&:=&\frac{1}{\langle
\phi^{\lambda},\partial_{\lambda}\phi^{\lambda}\rangle}\langle
N^{Re}_{2,0},\partial_{\lambda}\phi^{\lambda}\rangle;\\
-3iE(\lambda)A_{3,0}^{(2)}-A_{3,0}^{(1)}&:=&\frac{1}{\langle
\phi^{\lambda},\partial_{\lambda}\phi^{\lambda}\rangle}\langle
N^{Re}_{3,0},\partial_{\lambda}\phi^{\lambda}\rangle;\\
-iE(\lambda)A_{2,1}^{(2)}-A_{2,1}^{(1)}&:=&\frac{1}{\langle
\phi^{\lambda},\partial_{\lambda}\phi^{\lambda}\rangle}[\langle
N^{Re}_{2,1},\partial_{\lambda}\phi^{\lambda}\rangle
-\frac{1}{2}\Upsilon_{1,1}\langle
z\cdot\xi,\partial_{\lambda}\phi^{\lambda}\rangle];
\end{array}
\end{equation}
$$A_{k,l}^{(n)}:=\overline{A_{l,k}^{(n)}}\ \text{for}\ n=1,2,\ k+l=2,3,
\ k\not=l;$$ and
\begin{equation}\label{eq:pq20and02}
\begin{array}{lll}
-2iE(\lambda)P_{2,0}^{(n)}-E(\lambda)Q_{2,0}^{(n)}&:=&-\langle
N_{2,0}^{Im},\eta_{n}\rangle,\\
-2iE(\lambda)Q_{2,0}^{(n)}+E(\lambda)P_{2,0}^{(n)}&:=&\langle
N_{2,0}^{Re},\xi_{n}\rangle,\\
-3iE(\lambda)P_{3,0}^{(n)}-E(\lambda)Q_{3,0}^{(n)}&:=&-\langle
N_{3,0}^{Im},\eta_{n}\rangle,\\
-3iE(\lambda)Q_{3,0}^{(n)}+E(\lambda)P_{3,0}^{(n)}&:=&\langle
N_{3,0}^{Re},\xi_{n}\rangle,\\
\end{array}
\end{equation}
\begin{equation}
\begin{array}{lll}
2iE(\lambda)P_{1,2}^{(n)}-2E(\lambda)Q_{1,2}^{(n)}&:=&-\langle
N_{1,2}^{Im},\eta_{n}\rangle+i\langle N_{1,2}^{Re},\xi_{n}\rangle\\
& &+i\Upsilon_{1,1}\displaystyle\sum_{k=1}^{d}\bar{z}_{k}[\langle\eta_{k},\eta_{n}\rangle-\langle\xi_{k},\xi_{n}\rangle]\\
-E(\lambda)Q_{1,1}^{(n)}&:=&-\langle N_{1,1}^{Im},\eta_{n}\rangle,\\
E(\lambda)P_{1,1}^{(n)}&:=&\langle N_{1,1}^{Re},\xi_{n}\rangle.
\end{array}
\end{equation}
$$P_{k,l}^{(n)}:=\overline{P_{l,k}^{(n)}},\ Q_{l,k}^{(n)}:=\overline{Q_{k,l}^{(n)}}.$$
The functions $JN_{m,n}=\left(
\begin{array}{lll}
N^{Im}_{m,n}\\
-N^{Re}_{m,n}
\end{array}
\right)$ used above will be defined in the next subsection.
\begin{flushleft}
{\bf{ Expansion of $J\vec{N}$ and $\vec{R}$}}
\end{flushleft}
For $m+n=2$ we define
\begin{equation}\label{eq:Rform}
R_{m,n}:=\left(
\begin{array}{lll}
R_{m,n}^{(1)}\\
R_{m,n}^{(2)}
\end{array}
\right):=[L(\lambda)+iE(\lambda)(m-n)-0]^{-1}P_{c}JN_{m,n}.
\end{equation}
We denote the remainder of the second order expansion by $\tilde{R},$ i.e.,
\begin{equation}\label{eq:difTildeR}
\tilde{R}=\vec{R}-\sum_{m+n=2}R_{m,n}.
\end{equation}
For $m+n=3$ we define
\begin{equation}\label{eq:difRgeq3}
R_{m,n}:=[L(\lambda)+iE(\lambda)(m-n)-0]^{-1}P_{c}[JN_{m,n}+X_{m,n}]
\end{equation} where $\displaystyle\sum_{m+n=3}X_{m,n}:=\Upsilon_{1,1}\left(
\begin{array}{lll}
-\beta\cdot\eta \\
\alpha\cdot\xi
\end{array}
\right)$ and recall the definition of $\Upsilon_{1,1}$ in ~\eqref{eq:Gamma11}.

We define the quadratic terms $JN_{m,n},\ m+n=2,$ as
\begin{equation}\label{eq:SecondOrderTerm}
\displaystyle\sum_{m+n=2}\ JN_{m,n}\ =\ \sigma(\phi^{\lambda})^{2\sigma-1}\left(
\begin{array}{ccc}
2\ (\alpha\cdot\xi)(\beta\cdot\eta)\\
-[2\sigma+1](\alpha\cdot\xi)^{2}-(\beta\cdot\eta)^{2}
\end{array}\right),
\end{equation}
and define $JN_{m,n},$ $m+n=3,$ by
\begin{equation}\label{eq:ThirdOrderJN}
\sum_{m+n=3}JN_{m,n}:=\sum_{m+n=3}\left(
\begin{array}{lll}
N^{Im}_{m,n}\\
-N^{Re}_{m,n}
\end{array}
\right)
\end{equation} where $N^{Im}_{m,n}$ and $N^{Re}_{m,n}$ are defined as $$
\begin{array}{lll}
\displaystyle\sum_{m+n=3}N^{Im}_{m,n}&:=&2\sigma (\phi^{\lambda})^{2\sigma-1}(\alpha\cdot \xi)\displaystyle\sum_{m+n=2}[A^{(2)}_{m,n}\phi^{\lambda}+Q_{m,n}\cdot\eta+R^{(2)}_{m,n}]\\
& &+2\sigma (\phi^{\lambda})^{2\sigma-1}(\beta\cdot\eta)\displaystyle\sum_{m+n=2}[A_{m,n}^{(1)}\partial_{\lambda}\phi^{\lambda}+P_{m,n}\cdot\xi+R^{(1)}_{m,n}]\\
& &+\sigma (\phi^{\lambda})^{2\sigma-2}[(\alpha\cdot \xi)^2+(\beta\cdot\eta)^2][\beta\cdot\eta]+2\sigma(\sigma-1) (\phi^{\lambda})^{2\sigma-2}(\alpha\cdot\xi)^2 (\beta\cdot\eta),
\end{array}
$$ and
$$
\begin{array}{lll}
\displaystyle\sum_{m+n=3}N^{Re}_{m,n}&:=& 2\sigma(2\sigma+1)(\phi^{\lambda})^{2\sigma-1}(\alpha\cdot\xi)\displaystyle\sum_{m+n=2}[A^{(1)}_{m,n}\partial_{\lambda}\phi^{\lambda}+P_{m,n}\cdot\xi +R_{m,n}^{(1)}]\\
& &+2\sigma (\phi^{\lambda})^{2\sigma-1}(\beta\cdot\eta)\displaystyle\sum_{m+n=2}[A_{m,n}^{(2)}\phi^{\lambda}+Q_{m,n}\cdot\eta+R_{m,n}^{(2)}]\\
& &+\sigma[2\sigma-1+\frac{4}{3}(\sigma-1)(\sigma-2)](\phi^{\lambda})^{2\sigma-2}(\alpha\cdot\xi)^3+\sigma(2\sigma-1)(\phi^{\lambda})^{2\sigma-2}(\alpha\cdot\xi)(\beta\cdot\eta)^2.
\end{array}
$$
Now we expand $J\vec{N}$ to fourth order:
\begin{equation}\label{eq:FourthOrderJN}
\sum_{m+n=4}JN_{m,n}:=\sum_{m+n=4}\left(
\begin{array}{lll}
N^{Im}_{m,n}\\
-N^{Re}_{m,n}
\end{array}
\right)
\end{equation} with $N^{Im}_{m,n}$ and $N^{Re}_{m,n}$ defined as $$
\begin{array}{lll}
& &\displaystyle\sum_{m+n=4}N^{Im}_{m,n}\\
&:=& 2\sigma(\phi^{\lambda})^{2\sigma-1}(\alpha\cdot\xi)\displaystyle\sum_{m+n=3}[A^{(2)}_{m,n}\phi^{\lambda}
+Q_{m,n}\cdot\eta+R_{m,n}^{(2)}]\\
& &+2\sigma(\phi^{\lambda})^{2\sigma-1}(\beta\cdot\eta)\displaystyle\sum_{m+n=3}[A^{(1)}_{m,n}\partial_{\lambda}\phi^{\lambda}
+P_{m,n}\cdot\xi+R_{m,n}^{(1)}]\\
& &+2\sigma(\phi^{\lambda})^{2\sigma-1}\displaystyle\sum_{m+n=2}[A^{(1)}_{m,n}\partial_{\lambda}\phi^{\lambda}
+P_{m,n}\cdot\xi+R_{m,n}^{(1)}]
\displaystyle\sum_{m'+n'=2}[A^{(2)}_{m',n'}\phi^{\lambda}
+Q_{m',n'}\cdot\eta+R_{m',n'}^{(2)}]\\
& &+\sigma(\phi^{\lambda})^{2\sigma-2}[(2\sigma-1)(\alpha\cdot\xi)^2+3(\beta\cdot\eta)^2]\displaystyle\sum_{m+n=2}[A^{(2)}_{m,n}\phi^{\lambda}
+Q_{m,n}\cdot\eta+R_{m,n}^{(2)}]\\
& &+2\sigma(2\sigma-1)(\phi^{\lambda})^{2\sigma-2}(\beta\cdot\eta)(\alpha\cdot\xi) \displaystyle\sum_{m+n=2}[A^{(1)}_{m,n}\partial_{\lambda}\phi^{\lambda}
+P_{m,n}\cdot\xi+R_{m,n}^{(1)}]\\
& &+2\sigma(\sigma-1)(\phi^{\lambda})^{2\sigma-3}[\frac{2\sigma-1}{3}(\alpha\cdot\xi)^3(\beta\cdot\eta)+
(\alpha\cdot\xi)(\beta\cdot\eta)^3]
\end{array}
$$ and $$
\begin{array}{lll}
\displaystyle\sum_{m+n=4}N^{Re}_{m,n}
&:=&2\sigma(2\sigma+1)(\phi^{\lambda})^{2\sigma-1}(\alpha\cdot\xi)\displaystyle\sum_{m+n=3}[A^{(1)}_{m,n}\partial_{\lambda}\phi^{\lambda}
+P_{m,n}\cdot\xi+R_{m,n}^{(1)}]\\
& &+2\sigma(\phi^{\lambda})^{2\sigma-1}(\beta\cdot\eta)\displaystyle\sum_{m+n=3}[A^{(2)}_{m,n}\phi^{\lambda}
+Q_{m,n}\cdot\eta+R_{m,n}^{(2)}]\\
& &+\sigma(2\sigma+1)(\phi^{\lambda})^{2\sigma-1}[\displaystyle\sum_{m+n=3}(A^{(1)}_{m,n}\partial_{\lambda}\phi^{\lambda}
+P_{m,n}\cdot\xi+R_{m,n}^{(1)})]^2\\
& &+\sigma(\phi^{\lambda})^{2\sigma-1}[\displaystyle\sum_{m+n=2}(A^{(2)}_{m,n}\phi^{\lambda}
+Q_{m,n}\cdot\eta+R_{m,n}^{(2)})]^2\\
& &+\sigma(\phi^{\lambda})^{2\sigma-2}[3C_{1}(\sigma)(\alpha\cdot\xi)^2+C_2(\sigma)(\beta\cdot\eta)^2]\displaystyle\sum_{m+n=2}[A^{(1)}_{m,n}\partial_{\lambda}\phi^{\lambda}
+P_{m,n}\cdot\xi+R_{m,n}^{(1)}]\\
& &+2\sigma(\phi^{\lambda})^{2\sigma-2}C_{3}(\sigma)(\beta\cdot\eta)(\alpha\cdot\xi)\displaystyle\sum_{m+n=2}[A^{(2)}_{m,n}\phi^{\lambda}
+Q_{m,n}\cdot\eta+R_{m,n}^{(2)}]\\
& &+(\phi^{\lambda})^{2\sigma-3}[C_{4}(\sigma)(\alpha\cdot\xi)^4+C_{5}(\sigma)(\beta\cdot\eta)^4
+C_{6}(\sigma)(\alpha\cdot\xi)^2(\beta\cdot\eta)^2]
\end{array}
$$ where $C_{k}(\sigma),\ k=1,2,\cdots,6$ are real constants, $C_k(1)=1 $ if $k=1,2,3,$ and $C_{l}(1)=0$ if $l=4,5,6.$
\begin{remark}\label{remark2}
Note that if in ~\eqref{eq:NLS} $\sigma>1$ the function $JN_{m,n},\ m+n=4$ might NOT be well defined: in the last lines of definitions of $N_{m,n}^{Im}$ and $N_{m,n}^{Re}$ we have terms of the form $(\phi^{\lambda})^{2\sigma-3}(\alpha\cdot \xi)^{4}$ where the power $2\sigma-3$ of $\phi^{\lambda}$ might be negative. Still the definitions are useful because later we will take inner production with $JN_{m,n},\ m+n=4$ and $\left(
\begin{array}{lll}
\phi^{\lambda}\\
0
\end{array}
\right)$, see ~\eqref{eq:JN5sigma}.
\end{remark}

To facilitate later discussions we define
\begin{equation}\label{dif:R4}
R_{\geq 4}:=\vec{R}-\sum_{m+n=2,3}R_{m,n}
\end{equation}
and
\begin{equation}\label{def:JN5}
JN_{\geq 5}:=J\vec{N}(\vec{R},z)-\displaystyle\sum_{m+n=2}^{4}JN_{m,n}.
\end{equation}

\section{Derivation of Equation ~\eqref{eq:DetailLambda}}\label{sec:deriv}
By the equation for $\dot\lambda(t)$ in ~\eqref{eq:lambda} we derive the following modulation equation
\begin{equation}\label{eq:lambda2}
\frac{1}{2}\ \frac{d}{dt}\ \left\|\phi^{\lambda(t)}\right\|^2_{2}\ =\
\langle
\phi^{\lambda},\phi^{\lambda}_{\lambda}\rangle\dot\lambda\ =\  - \langle
ImN(\Vec{R},z),\phi^{\lambda}\rangle+\langle
\phi^{\lambda},\phi^{\lambda}_{\lambda}\rangle F_{4}\ -\langle
\phi^{\lambda},\phi^{\lambda}_{\lambda}\rangle\partial_{t}a_{1}.
\end{equation}
To see the increasing of the mass on the ground state, or $\left\|\phi^{\lambda(t)}\right\|^2_{2}$ we resort to expand the terms on the right hand side to fourth order in $z$ and $\bar{z}$:
\begin{itemize}
\item[(1)]
The definitions of $a_{1}$ in ~\eqref{eq:pkmn} and ~\eqref{eq:A1} imply
\begin{equation}\label{eq:a1detail}
\begin{array}{lll}
\partial_{t}a_1
&=&\partial_{t}a_1+iE(\lambda)\displaystyle
\sum_{m+n=2,3}(m-n)A_{m,n}^{(1)}\\
& &-\frac{1}{\langle \phi^{\lambda},\partial_{\lambda}\phi^{\lambda}\rangle}[\langle \displaystyle\sum_{m+n=2,3} JN_{m,n},\left(
\begin{array}{lll}
\phi^{\lambda}\\
0
\end{array}
\right)\rangle-\Upsilon_{1,1}\langle \beta\cdot\eta,\phi^{\lambda}\rangle]
\end{array}
\end{equation}
where in the second step the fact $\langle N_{1,1}^{Im},\phi^{\lambda}\rangle=0$, proved in ~\cite{GaWe} Lemma 9.4 on p. 291, is used. Extracting the lower order terms, we find
 \begin{equation}\label{eq:a1detail2}
\begin{array}{lll}
\partial_{t}a_1
&=&\partial_{z}a_1\cdot Z_{2,1}+\partial_{\bar{z}}a_1\cdot\overline{Z_{2,1}}
-\frac{1}{\langle \phi^{\lambda},\partial_{\lambda}\phi^{\lambda}\rangle}[\langle \displaystyle\sum_{m+n=2,3} JN_{m,n},\left(
\begin{array}{lll}
\phi^{\lambda}\\
0
\end{array}
\right\rangle-\Upsilon_{1,1}\langle \beta\cdot\eta,\phi^{\lambda}\rangle]\\
& &+\partial_{t}a_1+iE(\lambda)\displaystyle
\sum_{m+n=2,3}(m-n)A_{m,n}^{(1)}-\partial_{z}a_1\cdot Z_{2,1}-\partial_{\bar{z}}a_1\cdot\overline{Z_{2,1}}
\end{array}
\end{equation} with $Z_{2,1}=-\Gamma(z,\bar{z})z+\Lambda(z,\bar{z})z$ defined in ~\eqref{eq:ZNequation}.
\item[(2)]
Separate the terms of order $|z|^3$, $|z|^4$ and $|z|^5$ from $F_{4}$ and obtain
\begin{equation}\label{eq:F4}
\begin{array}{lll}
\langle \phi^{\lambda},\partial_{\lambda}\phi^{\lambda}\rangle F_{4}&=&\Upsilon_{1,1}\displaystyle\sum_{m+n=2,3}\langle R_{m,n},\left(
\begin{array}{lll}
0\\
\phi^{\lambda}
\end{array}
\right)\rangle+\Upsilon_{1,1}\langle (\beta+q)\cdot\eta,\phi^{\lambda}\rangle\\
& &+(\dot\gamma -\Upsilon_{1,1})[\langle \vec{R},
\left(
\begin{array}{lll}
0\\
\phi^{\lambda}
\end{array}
\right)
\rangle+\langle (\beta+q)\cdot\eta,\phi^{\lambda}\rangle]+\Upsilon_{1,1}\langle R_{\geq 4},\left(
\begin{array}{lll}
0\\
\phi^{\lambda}
\end{array}
\right)\rangle\\
& &\\
& &+\dot\lambda\langle \vec{R},
\left(
\begin{array}{lll}
\partial_{\lambda}\phi^{\lambda}\\
0
\end{array}
\right)
\rangle-\dot\lambda a_{1}\langle \partial_{\lambda}^{2}\phi^{\lambda},\phi^\lambda\rangle-\dot\lambda (\alpha+p)\langle \partial_{\lambda}\xi,\phi^{\lambda}\rangle.
\end{array}
\end{equation}
\item[(3)]
The definitions of $JN_{m,n},\ m+n=2,3,4,$ and $JN_{\geq 5}$ in ~\eqref{eq:SecondOrderTerm}-~\eqref{eq:FourthOrderJN} and ~\eqref{def:JN5} imply that
\begin{equation}\label{eq:imN}
\langle ImN(\vec{R},z),\phi^{\lambda}\rangle=\langle\sum_{m+n=2}^{4}JN_{m,n},\left(
\begin{array}{lll}
\phi^{\lambda}\\
0
\end{array}
\right)\rangle+\langle JN_{\geq 5}, \left(
\begin{array}{lll}
\phi^{\lambda}\\
0
\end{array}
\right)\rangle.
\end{equation}
\end{itemize}
Equations ~\eqref{eq:lambda2}-~\eqref{eq:imN} and  cancellation of  terms in sum leads to ~\eqref{eq:DetailLambda}.


\section{Estimates on the Eigenvectors of $L(\lambda)$ and the Parameters of Normal Form Transformation}\label{sec:CorrectionNormal}
Precise estimations of  $S_{z}$ and $S_{\lambda}$ defined in Theorems ~\ref{THM:Zequation} and ~\ref{THM:KeyTerm} require control on coefficients depending on norms of $\phi^\lambda$,  its derivatives
 as well as such norms of the neutral modes.
 Recall the definition of $\delta_{\infty}$ and the fact $\delta_{\infty}=\mathcal{O}(\delta(\lambda(t)))$ in ~\eqref{eq:defDelta}.
 
The result is
\begin{proposition}\label{Prop:Parameters}
There exist constants $C_0,\ C_1, \ C_2\in\mathbb{R}$ such that in the space $\langle x\rangle^{-4}{H}^{2}$
\begin{equation}\label{eq:LambdaPhi2}
\begin{array}{lll}
\phi^{\lambda}&=&C_0 |e_{0}+\lambda|^{\frac{1}{2\sigma}} \phi_{lin}+\cO(|e_{0}+\lambda|^{1+\frac{1}{2\sigma}})=\cO(\delta_{\infty})\\
& &\\
\partial_{\lambda}\phi^{\lambda}&=&C_1 |e_{0}+\lambda|^{\frac{1}{2\sigma}-1}\phi_{lin}+\cO(|e_{0}+\lambda|^{\frac{1}{2\sigma}})=\cO(\delta_{\infty}^{-(2\sigma-1)}),\\
& &\\
\partial_{\lambda}^2 \phi^{\lambda}&=&C_2 |e_{0}+\lambda|^{\frac{1}{2\sigma}-2}\phi_{lin}+\cO(|e_{0}+\lambda|^{\frac{1}{2\sigma}-1})=\cO(\delta_{\infty}^{-(4\sigma-1)})
.
\end{array}
\end{equation}
\begin{equation}\label{eq:denominator}
\frac{1}{\langle \phi^{\lambda},
\partial_{\lambda}\phi^{\lambda}\rangle}\lesssim \delta_{\infty}^{2\sigma-2}.
\end{equation}
For the neutral modes we have
\begin{equation}\label{eq:derXi}
\|\langle x\rangle^{4}\partial_{\lambda}\xi_{n}\|_{2},\ \|\langle x\rangle^{4}\partial_{\lambda}\eta_{n}\|_{2}\lesssim 1;
\end{equation}
\begin{equation}\label{eq:asympto}
\|\langle x\rangle^{4}(\eta_{m}-\xi_{m}^{lin})\|_{H^2},\ \|\langle x\rangle^{4}(\xi_{m}-\xi_{m}^{lin})\|_{H^2},\ \|\langle x\rangle^{4}(\xi_{m}-\eta_{m})\|_{H^2}=\cO(\delta_{\infty}^{2\sigma}).
\end{equation} Recall $P_{c}^{lin}$ is the orthogonal project onto the essential spectrum of $-\Delta+V$
\begin{equation}\label{eq:projection}
P_{c}^{\lambda}=P_{c}^{lin}\left(
\begin{array}{ll}
1&0\\
0&1
\end{array}
\right)+\mathcal{O}(\delta_{\infty}^{2\sigma}).
\end{equation} The function $\Upsilon_{1,1}$ in ~\eqref{eq:Gamma11} satisfies the estimate
\begin{equation}\label{eq:upsilon11}
|\Upsilon_{1,1}|\lesssim \delta_{\infty}^{2\sigma-2}|z|^2.
\end{equation}
In what follows we estimate various functions defined in Appendix ~\ref{sec:NormalFormExp}.\\
For $m+n=2,$ $$\|\langle x\rangle^{4}JN_{m,n}\|_{2},\ \|\langle x\rangle^{-4}R_{m,n}\|_{2}\lesssim \delta_{\infty}^{2\sigma-1}|z|^{2},$$
$$|A^{(1)}_{m,n}|\lesssim \delta_{\infty}^{2(2\sigma-1)}|z|^2,$$
$$ |A^{(2)}_{m,n}|\lesssim \delta_{\infty}^{2\sigma-2}|z|^2,$$ $$|P_{m,n}^{(k)}|,\ |Q_{m,n}^{(k)}|\lesssim \delta_{\infty}^{2\sigma-1}|z|^{2},\ k=1,2,\cdots,N.$$
For $m+n=3,$ $$\|\langle x\rangle^{4}JN_{m,n}\|_{2},\ \
 \|\langle x\rangle^{-4 }R_{m,n}\|_{2}\lesssim \delta_{\infty}^{2\sigma-2}|z|^{3}.$$
$$|A^{(1)}_{m,n}|\lesssim \delta_{\infty}^{4\sigma-3}|z|^3,$$ $$|A^{(2)}_{m,n}|\lesssim \delta_{\infty}^{2\sigma-3}|z|^3,$$ $$|P_{m,n}^{(k)}|,\ |Q_{m,n}^{(k)}|\lesssim \delta_{\infty}^{2\sigma-2}|z|^{3},\ k=1,2,\cdots,N.$$
For $m+n=4$ and $\sigma=1$ $$\|\langle x\rangle^{4}JN_{m,n}\|_{2}\lesssim \delta_{\infty}|z|^{4}.$$
\end{proposition}
\begin{proof}
Since all the functions are defined in term of $\phi^{\lambda},\ \xi,\ \eta$ and their derivatives, we start with deriving estimates for them, or proving ~\eqref{eq:LambdaPhi2}-~\eqref{eq:asympto}.

The key observation is these functions can be constructed perturbatively, as can be found in the known results in the spaces ${H}^{k},k=1,2,$ (see e.g.\cite{TsaiYau02}). In what follows we re-do the proof in the desired space.

We start with ~\eqref{eq:LambdaPhi2} by decompose $\phi^{\lambda}$ as $$\phi^{\lambda}=\delta \phi_{lin}+\phi_{Re} \ \text{with}\ \langle \phi_{lin}, \phi_{Re}\rangle_{L^2}=0.$$ On the subspace parallelling to $\phi^{\lambda}$ and its orthogonal we derive two equations
\begin{equation}\label{eq:existence2}
\begin{array}{lll}
\delta^{-1}\phi_{Re}&=&\delta^{2\sigma}(-\Delta+V+\lambda)^{-1}P_c ( \phi_{lin}+\delta^{-1}\phi_{Re})^{2\sigma+1}\\
& &\\
(\lambda+e_0) &=& \delta^{2\sigma}\langle \phi_{lin}, ( \phi_{lin}+\delta^{-1}\phi_{Re})^{2\sigma+1}\rangle_{L^2}.
\end{array}
\end{equation} where, $\lambda+e_0$ is a fixed small positive constant, and $1-P_c$ is the projection onto the $L^2$ 1-dimensional subspace $\{\phi_{lin}\}$.

We prove the existence of the solutions in appropriate Sobolev spaces by applying the contraction mapping  theorem. Its applicability is fairly routine except observing the map $$(-\Delta+V+\lambda)^{-1}P_c=(-\Delta+\lambda)^{-1}P_c-(-\Delta+\lambda)V(-\Delta+V+\lambda)^{-1}P_c:\ \langle x\rangle^{-4}H^2\rightarrow \langle x\rangle^{-4}H^2 $$ is bounded. By the contraction mapping theorem it is easy to construct the small solutions $\delta^{-1}\phi_{Re}$ and $\delta^{2\sigma}$ and find they are functions of $\lambda+e_0$ with differentiability $\mathcal{C}^3$. (Actually if $\sigma$ is an integer then the functions are analytic in $\lambda+e_0$.) The dependence on $\lambda+e_0$ can be displayed by rewriting  ~\eqref{eq:existence2}
$$
\begin{array}{lll}
\delta^{-1}\phi_{Re}&=&\delta^{2\sigma}(-\Delta+V+\lambda)^{-1}P_c[(\phi_{lin})^{2\sigma+1}+\delta^{2\sigma}\kappa](1+\cO(\delta^{2\sigma}))\\
\delta^{2\sigma}&=&(\lambda+e_0)(\int \phi_{lin}^{2\sigma})^{-1}+C (\lambda+e_0)^2+ \cO(\lambda+e_0)^3
\end{array}
$$ with $\kappa$ being some function in $\langle x\rangle^{-4}{H}^2$ and $C$, a constant. From these it is easy to derive
~\eqref{eq:LambdaPhi2}. The estimates ~\eqref{eq:derXi} and ~\eqref{eq:asympto} can be proved similarly.

~\eqref{eq:denominator} is implied by ~\eqref{eq:LambdaPhi2}. The estimates ~\eqref{eq:projection} and ~\eqref{eq:upsilon11} are implied by their definitions ~\eqref{eq:PdProjection} and ~\eqref{eq:Gamma11} and the estimates ~\eqref{eq:LambdaPhi2}-~\eqref{eq:asympto}.

Recall the definitions of $JN_{m,n}$, $m+n=2,3,4,$ in ~\eqref{eq:SecondOrderTerm}- ~\eqref{eq:FourthOrderJN}. The desired estimates are simple applications of ~\eqref{eq:LambdaPhi2}-~\eqref{eq:upsilon11}. Since all the other functions are defined in terms of $JN_{m,n}$ and the estimates are straightforward, we omitted the details here. This completes the  proof.
\end{proof}

\section{The Estimate on $JN_{\geq 5}:=J\vec{N}-\displaystyle\sum_{m+n=2,3,4}JN_{m,n}$}\label{sec:Remainder}
In the section we estimate the remainder of $JN$ after expanding it to the fourth order.
Recall the definitions of $JN_{m,n},\ m+n=2,3, 4$, in  ~\eqref{eq:SecondOrderTerm}-~\eqref{eq:FourthOrderJN}.
\begin{proposition}\label{Prop:NonlinearRem} If $\sigma=1$ then
\begin{equation}\label{eq:HiRemainder}
JN_{\geq 5}=Loc+NonLoc
\end{equation} with $NonLoc:=[R_{1}^2+R_2^2] J\vec{R}$ and
\begin{equation}\label{eq:Loc}
\|\langle x\rangle^{4 }Loc\|_{L^1},\ \|\langle x\rangle^{4 }Loc\|_{L^2}\lesssim  |z|^5+|z|(\delta_{\infty}+|z|)\|\langle x\rangle^{-4 }R_{\geq 4}\|_{2}+(\delta_{\infty}+|z|)\|\langle x\rangle^{-4 }\vec{R}\|_{2}^2,
\end{equation}
and
\begin{equation}\label{eq:InnerJN5}
\left| \left\langle JN_{\geq 5},\left(
\begin{array}{lll}
\phi^{\lambda}\\
0
\end{array}
\right) \right\rangle\right|\lesssim \delta_{\infty}|z|^5 +\delta_{\infty}|z|\|\langle x\rangle^{-4}R_{\geq 4}\|_{2}+\delta_{\infty}\|\langle x\rangle^{-4}R_{\geq 4}\|_{2}^2.
\end{equation}
\end{proposition}
\begin{proof}
Recall the decomposition of $\psi$ in ~\eqref{eq:decom} and the fact that the nonlinearity of ~\eqref{eq:NLS} is cubic if $\sigma=1$. Hence each term in $Loc$ must be product of three terms taken from $\phi^{\lambda}, \ a_{1}\phi^{\lambda}_{\lambda},\ a_2 \phi^\lambda,\ \displaystyle\sum_{n=1}^N (\alpha_n+p_n)\xi_n,\ i\sum_{n=1}^N (\beta_n+q_n)\eta_n,\ R_{m,n}$ and $R_{\geq 4}.$ By considering all the possibilities and using Proposition ~\ref{Prop:Parameters} we obtain ~\eqref{eq:Loc}. The procedure is tedious but not difficult, hence is omitted here. ~\eqref{eq:InnerJN5} follows easily easily from ~\eqref{eq:Loc} and the fact $\phi^{\lambda}=\mathcal{O}(\delta_{\infty})$. This completes the proof.
\end{proof}

Now we study the cases $\sigma>1.$
\begin{proposition}\label{Prop:NonlinearRem2}
If $\sigma>1$ then
\begin{equation}\label{eq:SigmaGeq1}
\begin{array}{lll}
& &\left\|J\vec{N}-\displaystyle\sum_{m+n=2,3}JN_{m,n}\right\|_{L^1}+\left\|J\vec{N}-\displaystyle\sum_{m+n=2,3}JN_{m,n}\right\|_{L^2}\\
&\lesssim &|z|^{2\sigma+1}+|z|^{4}+|z|\|\langle x\rangle^{-4} R_{\geq 4}\|_{2}+|z|^{\sigma+1}(\|\langle x\rangle^{-4}\vec{R}\|_{\sigma}^{\sigma}+\|\langle x\rangle^{-4}\vec{R}\|_{2\sigma}^{\sigma})+ \|\vec{R}\|_{2\sigma+1}^{2\sigma+1}+\|\vec{R}\|_{2(2\sigma+1)}^{2\sigma+1}
\end{array}
\end{equation}
and
\begin{equation}\label{eq:JN5sigma}
\left|\left\langle J\vec{N}(\vec{R},z)-\sum_{m+n=2}^{4}JN_{m,n},\left(
\begin{array}{lll}
\phi^{\lambda}\\
0
\end{array}
\right)\right\rangle\right|\lesssim |z|^{2\sigma+2}+|z|^5 +\delta_{\infty}^{2\sigma-1}|z|\|\langle x\rangle^{-4}R_{\geq 4}\|_{2}+\|\langle x\rangle^{-4}R_{\geq 4}\|_{2}^2.
\end{equation}
\end{proposition}
\begin{proof}
As in the proof of the case $\sigma=1$ the basic idea in proving ~\eqref{eq:SigmaGeq1} and ~\eqref{eq:JN5sigma} is to Taylor expand the function $JN$ in $z$ and $\bar{z}$.
What makes the present situation different is that if $\sigma\not\in \mathbb{N}$, then the nonlinearity $|\psi|^{2\sigma}\psi$ is not smooth at $\psi=0$. Technically the decomposition of $
\psi $ in ~\eqref{eq:decom} makes $|\psi|^{2\sigma}$ to be of the form $$|\psi|^{2\sigma}(x,t)=|(\phi^{\lambda})^2(x)+\epsilon(x,t)|^{\sigma}$$ for some $\epsilon\in H^2(\mathbb{R}).$ Since the inequality $|\epsilon|<(\phi^{\lambda})^2$ does not hold for all $x\in \mathbb{R}^{3},$ we find that after expanding in $\epsilon$ to certain orders some undesired negative powers of $\phi^{\lambda}$ will be encountered if $\sigma\not\in \mathbb{N}$. To prevent the negative powers from appearing in the final form we have to adopt some tricks, namely compare the sizes of $
\phi^{\lambda}$ and $\epsilon$ and discuss several regimes.

Now we start proving the proposition. Recall that $|\psi|^{2\sigma}=[(\phi^{\lambda})^2+2\phi^{\lambda}(\alpha\cdot\xi)+I]^{\sigma}$ with $I:=2\phi^\lambda (I_1-\alpha\cdot\xi)+I_1^2+I_2^2$, $I_1$ and $I_2$ defined in ~\eqref{JvecN}.
To control the remainder of the expansion around $\phi^{\lambda}$ we consider separately two regimes $$(\phi^{\lambda})^2 \geq 4|\phi^{\lambda}(\alpha\cdot\xi)|+2|I|\ \text{and}\ (\phi^{\lambda})^2 < 4|\phi^{\lambda}(\alpha\cdot\xi)|+2|I|.$$

For the second regime we have
\begin{equation}\label{eq:secRegime}
\begin{array}{lll}
|J\vec{N}-\displaystyle\sum_{m+n=2,3}JN_{m,n}|
&\leq& |J\vec{N}|+|\displaystyle\sum_{m+n=2,3}JN_{m,n}|\\
&\leq &|z|^{2\sigma+1}+|z|^{4}+ |\vec{R}|^{2\sigma+1}.
\end{array}
\end{equation}

For the first regime i.e.
\begin{equation}\label{eq:bounds}
(\phi^{\lambda})^2 \geq 4|\phi^{\lambda}(\alpha\cdot\xi)|+2|I|.
\end{equation} we only study one term $(\phi^{\lambda})^{2\sigma+1}\mathcal{O}(\epsilon^4)$ with  $$\epsilon:=(\phi^{\lambda})^{-2}[2\phi^{\lambda}(\alpha\cdot\xi)+I]\leq \frac{1}{2}.$$ It is easy to see that this term is the remainder after expanding $\epsilon$ to the third order.

We claim that
\begin{equation}\label{eq:desired2}
|(\phi^{\lambda})^{2\sigma+1}\mathcal{O}(\epsilon^4)|\lesssim |\alpha\cdot \xi|^{2\sigma+1}+ |\alpha\cdot\xi|^4+|I|^{\sigma+\frac{1}{2}}.
\end{equation} We compute directly to obtain
$$
\begin{array}{lll}
|(\phi^{\lambda})^{2\sigma+1} \mathcal{O}(|\epsilon|^4)|&\lesssim &(\phi^{\lambda})^{\sigma-7} I^4+ (\phi^{\lambda})^{2\sigma-3}(\alpha\cdot\xi)^4.
\end{array}
$$ Apply ~\eqref{eq:bounds} to control the first term on the right hand side $$(\phi^{\lambda})^{\sigma-7} I^4\lesssim |I|^{\sigma+\frac{1}{2}}.$$ To bound the second term we have two possibilities: $2\sigma-3\geq 0$ and $2\sigma-3<0.$ For the first we have
$$(\phi^{\lambda})^{2\sigma-3}|(\alpha\cdot\xi)^4|\lesssim (\alpha\cdot\xi)^4,$$ hence ~\eqref{eq:desired2} holds trivially; for the second apply ~\eqref{eq:bounds} to obtain
$$(\phi^{\lambda})^{2\sigma-3}(\alpha\cdot\xi)^4=|\alpha\cdot\xi|^{2\sigma+1} (\phi^{\lambda})^{2\sigma-3}  |\alpha\cdot\xi|^{3-2\sigma}\lesssim |\alpha\cdot \xi|^{2\sigma+1}.$$ By collecting the estimates above we prove ~\eqref{eq:desired2}.

This completes the proof for the first regime, and moreover this together with ~\eqref{eq:secRegime} completes the proof for ~\eqref{eq:SigmaGeq1}.

Now we turn to ~\eqref{eq:JN5sigma}. Here the function $JN$ has to be expanded to one more order to obtain the desired estimate. This is enabled by the fact the function was taken inner product with function $\phi^{\lambda}$. The technique is similar to the proof of ~\eqref{eq:SigmaGeq1}, hence we omit the details here. The proof is complete.
\end{proof}

\section{Proof of Proposition ~\ref{prop:useful}, Equations ~\eqref{eq:KeyTermRef}, ~\eqref{eq:ref1} and ~\eqref{eq:ref2}}\label{SEC:DetailInfo}
Proposition ~\ref{prop:useful} and Equations  ~\eqref{eq:KeyTermRef}, ~\eqref{eq:ref1} and ~\eqref{eq:ref2} are included in Propositions ~\ref{Prop:KeyEst} and ~\ref{Prop:Majorants} below.
 Recall the definitions of $\delta_{\infty}$, $R_{\geq 4}$ in ~\eqref{eq:defDelta}, ~\eqref{dif:R4} respectively.
\begin{proposition}\label{Prop:KeyEst} If  $\frac{|z_{0}|}{\delta_{\infty}}\ll 1$ for $\sigma=1$ and for $\sigma>1$, $|z_{0}|\leq\delta_{\infty}^{C(\sigma)}$ with $C(\sigma)$ sufficiently large,
 then the following results hold: $\vec{R}$ and $R_{\geq 4}$ satisfy the estimates
\begin{align}
\|\vec{R}\|_{2}^{2}&\lesssim  |z_{0}|^{2}\label{eq:L2NormR},\\
\|\vec{R}\|_{\infty}&\lesssim |z_{0}|^{2}(1+t)^{-\frac{3}{2}}+\delta_{\infty}^{2\sigma-1}|z(t)|^2,\label{eq:LinftyNormR}\\
\|\langle x\rangle^{-4 }\vec{R}\|_{H^{2}} &\lesssim |z_{0}|^{2}(1+t)^{-\frac{3}{2}}+\delta_{\infty}^{2\sigma-1}
|z(t)|^2.\label{eq:Sobolev}
\end{align} If $\sigma=1$ then
\begin{equation}
\|\langle x\rangle^{-4 }R_{\geq 4}\|_{2}\lesssim |z_{0}|^{2}(1+t)^{-\frac{3}{2}}+\delta_{\infty}|z_{0}|^{2}|z(t)|^2\label{eq:BetterEstR};
\end{equation} and if $\sigma>1$ then
\begin{equation}
\|\langle x\rangle^{-4 }R_{\geq 4}\|_{2}\lesssim |z_{0}|^{2}(1+t)^{-\frac{3}{2}}+[|z_{0}|^{2}+|z_{0}|^{2\sigma-1}]|z(t)|^2\label{eq:BetEstRsigma}.
\end{equation}
\end{proposition}
The proposition will be proved shortly.

We prepare for the proof by defining some functions. To bound various functions in the proposition we define the following estimating functions:
\begin{equation}
\begin{array}{lll}
\mathcal{M}_{1}(T)&:=&\displaystyle\max_{0\leq t\leq T}\|\vec{R}(t)\|_{\infty}[|z_{0}|^{2}(1+t)^{-\frac{3}{2}}+\delta_{\infty}^{2\sigma-1}
|z(t)|^2]^{-1},\\
\mathcal{M}_{2}(T)&:=&\displaystyle\max_{0\leq t\leq T}\|\langle x\rangle^{-4 }\vec{R}(t)\|_{H^2}[|z_{0}|^{2}(1+t)^{-\frac{3}{2}}+\delta_{\infty}^{2\sigma-1} |z(t)|^2]^{-1},\\
\mathcal{M}_{4}(T)&:=&\displaystyle \max_{0\leq t\leq T} \|\vec{R}\|_{2}|z_{0}|^{-1}.
\end{array}
\end{equation}
For $\sigma=1$ we define $$\mathcal{M}_{3}(T):=\displaystyle\max_{0\leq t\leq T}\|\langle x\rangle^{-4 }R_{\geq 4}(t)\|_{2}[|z_{0}|^{2}(1+t)^{-\frac{3}{2}}+\delta_{\infty}|z_{0}|^2|z(t)|^2]^{-1};$$
for $\sigma>1$
$$\mathcal{M}_{3}(T):=\displaystyle\max_{0\leq t\leq T}\|\langle x\rangle^{-4 }R_{\geq 4}(t)\|_{2}[|z_{0}|^{2}(1+t)^{-\frac{3}{2}}+(|z_{0}|^2+|z_{0}|^{2\sigma-1})
|z(t)|^2]^{-1}.$$

In the present paper we use $|z|$ as a gauge to measure sizes of different functions. This makes it necessary to obtain lower and upper bounds for $|z|$.
Recall the definition $\Gamma(z,\bar{z})\equiv\Gamma^{\lambda}(z,\bar{z})$ in ~\eqref{energy-id}.
By Equation ~\eqref{eq:LambdaPhi2} and the assumption (FGR) there exist constants $C_\pm$ such that
\begin{equation}\label{eq:LowerUpperBound}
C_{+}\delta_{\infty}^{2(2\sigma-1)} |z|^4\leq 2  z^* \Gamma^{\lambda}(z,\bar{z})z\leq C_{-}\delta_{\infty}^{2(2\sigma-1)}|z|^4.
\end{equation} We define the upper and lower bounds $z_{\pm}$ by
\begin{equation}\label{eq:difZpm}
z_{\pm}(t):=(|z_{0}|^{-2}+C_{\pm}\delta_{\infty}^{2(2\sigma-1)} t)^{-\frac{1}{2}}.
\end{equation}

Recall the equations for $\dot\lambda$, the definitions for $\Psi$ and $\mathcal{K}$ in ~\eqref{eq:DetailLambda} and ~\eqref{eq:ZNequation}, and recall the equation for $\dot\gamma$ in ~\eqref{eq:gamma}, and the definition of $\Upsilon_{1,1}$ in ~\eqref{eq:Gamma11}. In the rest of the paper we use $Remainder$ to represent different terms satisfying the estimate
\begin{equation}\label{eq:Remainder}
|Remainder|\lesssim |z(t)|^{4}+\|\langle x\rangle^{-4 }\vec{R}\|_{2}^{2}+|z(t)|\|\langle x\rangle^{-4 }R_{\geq 4}\|_{2}.
\end{equation}

Define a constant $\epsilon_{\infty}$ by $$\epsilon_{\infty}:=\displaystyle\max_{0\leq t< \infty}\|\vec{R}(t)\|_{H^2}.$$ By ~\eqref{eq:stability} it is a small constant. The result is
\begin{proposition}\label{Prop:Majorants}
Suppose that $\frac{|z_{0}|}{\delta_{\infty}}\ll 1$ for $\sigma=1$ and $|z_{0}|=\delta_{\infty}^{C(\sigma)}$ for $\sigma>1$ with $C(\sigma)$ being sufficiently large. Then $$\text{if}\ \displaystyle\sum_{k=1}^{4}\mathcal{M}_{k}\leq (\delta_\infty+\epsilon_{\infty})^{-\frac{1}{2}}\ \text{and}\ 10|z_{+}|\geq |z|\geq \frac{1}{10}|z_{-}|\ \text{in the time interval}\ [0,\delta],$$ then in the same interval the following results hold:
\begin{itemize}
\item[(1)]
The function $|z|$ admits lower and upper bounds
\begin{equation}\label{eq:ControlZ}
\frac{1}{5}|z_{-}(t)|\leq |z(t)|\leq\ 5|z_{+}(t)|\ \text{for any time}\ t.
\end{equation}
\item[(2)]
The functions $\mathcal{K}$, $\Psi$, $\dot\lambda$ and $\dot\gamma-\Upsilon_{1,1}$ satisfy the following estimates
\begin{equation}\label{eq:KeyTerm1}
|\Psi|\lesssim |z|\delta_{\infty}^{2\sigma-1}\|\langle x\rangle^{-4}R_{\geq 4}\|_{2}+\|\langle x\rangle^{-4}R_{\geq 4}\|_{2}^2+\delta_{\infty}^{2\sigma-1}|z|^5,
\end{equation}
if $\sigma=1$ then
\begin{equation}\label{eq:RoughEstLambda}
|\dot\lambda|,\ |\mathcal{K}| \lesssim \delta_{\infty}\ Remainder,
\end{equation}
\begin{equation}\label{eq:RoughEstGamma}
|\dot\gamma-\Upsilon_{1,1}| \lesssim  Remainder.
\end{equation}
If $\sigma>1$ then
\begin{equation}\label{eq:sigmaG1}
|\dot\lambda |,\ |\dot\gamma-\Upsilon_{1,1}|,\ |\mathcal{K}|\lesssim |z|^{2\sigma+1}+Remainder.
\end{equation}
\item[(3)]
\begin{align}
\mathcal{M}_{1}(T)&\lesssim 1+\delta_{\infty}[\mathcal{M}_{2}(T)+\mathcal{M}_{2}^2(T)]+\epsilon_{\infty}\mathcal{M}_{1}(T),\label{eq:EstM1}\\
\mathcal{M}_{2}(T)&\lesssim 1+\delta_{\infty}[\mathcal{M}_{2}(T)+\mathcal{M}_{2}^2(T)]+\epsilon_{\infty}\mathcal{M}_{1}(T),\\
\mathcal{M}_{3}(T)&\lesssim 1+\mathcal{M}_{4}(T)\mathcal{M}_{1}^2(T)+\mathcal{M}_4^2(T)\mathcal{M}_1(T)+\delta_{\infty}(\mathcal{M}_{3}(T)+\mathcal{M}_{3}^2(T))
\label{eq:EstM3},\\
\mathcal{M}_4^2(T) &\lesssim 1+|z_{0}|[\mathcal{M}_{3}^2(T)+\mathcal{M}_1^{2\sigma}(T)\mathcal{M}_4^2(T)]\label{eq:EstM4}.
\end{align}
\end{itemize}
\end{proposition}
The estimates ~\eqref{eq:ControlZ}-~\eqref{eq:sigmaG1}, ~\eqref{eq:EstM3} and ~\eqref{eq:EstM4} of the proposition will be proved in later sections. The proofs of the other two estimates are almost the same to the corresponding estimates of ~\cite{GaWe}, specifically Propositions 11.3 and 11.5, pp. 299 and 300 respectively, hence are omitted here.
\bigskip

{\bf{Proof of Proposition ~\ref{Prop:KeyEst}}}
By the local well-posedness of ~\eqref{eq:NLS} there exists some interval $[0,\delta]$ such that for any $T\in [0,\delta]$ $$\displaystyle\sum_{k=1}^{4}\mathcal{M}_{k}(T)\leq (\delta_\infty+\epsilon_{\infty})^{-\frac{1}{2}}$$ and $$10|z_{+}|\geq |z|\geq \frac{1}{10}|z_{-}|.$$

Then the estimates ~\eqref{eq:ControlZ}-~\eqref{eq:EstM4} hold in this time interval, which include \begin{equation}\label{eq:zFinal}
 5|z_{+}|\geq |z|\geq \frac{1}{5}|z_{-}|.
\end{equation}

Now we turn to the estimates on $\mathcal{M}_{k},\ k=1,2,3,4.$
By substituting estimates of $\mathcal{M}_1$ and $\mathcal{M}_4$ of ~\eqref{eq:EstM1} and ~\eqref{eq:EstM4} into ~\eqref{eq:EstM3} we obtain $$ \mathcal{M}_{3}(T)\leq 1+(\epsilon_{\infty}+\delta_{\infty}+|z_{0}|)P(\mathcal{M}_1(T),\cdots,\mathcal{M}_4(T))$$ with $P(x_1,\cdot,x_4)$ being some polynomial of positive coefficient. This together with the other estimates implies
\begin{equation}\label{eq:M}
\mathcal{M}\lesssim 1+(\epsilon_{\infty}+\delta_{\infty}+|z_{0}|) P_1(M)
\end{equation} with $\mathcal{M}:=\sum_{k=1}^{4}\mathcal{M}_{k}$ and $P_{1}$ being some polynomial.

By the assumption on the initial condition we find $\displaystyle\sum_{n=1}^4 \mathcal{M}_{n}(0)\lesssim 1$, which together with ~\eqref{eq:M} implies
\begin{equation}\label{eq:BoundMajorant}
\mathcal{M}\lesssim 1\ \text{in}\ [0,\delta].
\end{equation}

~\eqref{eq:BoundMajorant}, ~\eqref{eq:zFinal} and the local wellposedness of ~\eqref{eq:NLS} imply that ~\eqref{eq:ControlZ}-~\eqref{eq:EstM4} hold in a larger interval $[0,\delta_2],$ so do ~\eqref{eq:BoundMajorant} and ~\eqref{eq:zFinal}.

By induction on the time interval we prove ~\eqref{eq:BoundMajorant} holds on the interval $[0,K]$ for any $K>0$, hence ~\eqref{eq:BoundMajorant} and  ~\eqref{eq:ControlZ}-~\eqref{eq:sigmaG1} hold in $[0,\infty).$

To conclude the proof the fact $\mathcal{M}\lesssim 1$ in $[0,\infty)$ implies Proposition ~\ref{Prop:KeyEst}.

The proof is complete.
\begin{flushright}
$\square$
\end{flushright}

\subsection{Proof of ~\eqref{eq:KeyTerm1}-~\eqref{eq:sigmaG1}}\label{subsec:equations}
In this subsection we study the derivatives of the normal form transformation, which appear in $\mathcal{K}$ of ~\eqref{eq:ZNequation} and $\Psi$ of ~\eqref{eq:DetailLambda}.

Recall the equation for $z$ in ~\eqref{new-nf1}. Define
\begin{equation}\label{eq:defZ21}
Z_{2,1}:=-\Gamma(z,\bar{z})z+\Lambda(z,\bar{z})z
\end{equation}
and define
$$A_1:=\partial_{t}a_1+iE(\lambda)\displaystyle\sum_{m+n=2,3}(m-n)A_{m,n}^{(1)}+\partial_{z}a_1\cdot Z_{2,1}+\partial_{\bar{z}}a_1\cdot \overline{Z_{2,1}},$$
$$A_2:=\partial_{t}a_2+iE(\lambda)\displaystyle\sum_{m+n=2,3}(m-n)A_{m,n}^{(2)}+\partial_{z}a_2\cdot Z_{2,1}+\partial_{\bar{z}}a_2\cdot \overline{Z_{2,1}},$$ $$P_k:=\partial_{t}p_k+iE(\lambda)\displaystyle\sum_{m+n=2,3}(m-n)P_{m,n}^{(k)}+\partial_{z}p_k\cdot Z_{2,1}+\partial_{\bar{z}}p_k\cdot \overline{Z_{2,1}},$$ $$Q_k:= \partial_{t}q_{k}+iE(\lambda)\displaystyle\sum_{m+n=2,3}(m-n)Q_{m,n}^{(k)}+\partial_{z}q_k\cdot Z_{2,1}+\partial_{\bar{z}}q_k\cdot \overline{Z_{2,1}}.$$

The result is
\begin{lemma}\label{LM:apremainder}
\begin{equation}\label{eq:apq2}
|\mathcal{K}|,\ |A_1|,\ |A_2|,\ |P_k|,\ |Q_k|\lesssim |z|(|\dot\lambda|+|\dot\gamma-\Upsilon_{1,1}|)+\delta_{\infty}^{C(\sigma)}Remainder
\end{equation} with $C(1)=1$, $C(\sigma)=0$ if $\sigma>1,$ and $k=1,2,\cdots,N.$
\end{lemma}
\begin{proof}
The definitions of $p_k$ in ~\eqref{eq:pkmn}, the equation for $z$ in ~\eqref{eq:ZNequation} imply
$$|P_k|\leq |\dot\lambda||\partial_\lambda p_k|+ |\partial_z p_k||\mathcal{K}|\leq |\dot\lambda||z|+|z||\mathcal{K}|$$ and $$|Q_k|\leq |\dot\lambda||\partial_\lambda q_k|+ |\partial_z q_k||\mathcal{K}|\leq |\dot\lambda||z|+|z||\mathcal{K}|.$$ Apply Propositions ~\ref{Prop:Parameters}, ~\ref{Prop:NonlinearRem} and ~\ref{Prop:NonlinearRem2} to obtain
\begin{equation}\label{eq:estK}
\begin{array}{lll}
|\mathcal{K}|&\leq &(|z|+\|\langle x\rangle^{-4}\vec{R}\|_2)(|\dot\lambda|+|\dot\gamma-\Upsilon_{1,1}|)+\delta_{\infty}^{C(\sigma)}Remainder
\end{array}
\end{equation} with $C(1)=1$ and $C(\sigma)=0$ if $\sigma>1$,
which together with the estimate above implies the estimates for $\mathcal{K}$, $P_{k}$ and $Q_{k}$ in ~\eqref{eq:apq2}.

By almost identical arguments we produce the estimates for $A_{1}$ and $A_{2}.$

The proof is complete.
\end{proof}
\begin{flushleft}
{\bf{Proof of ~\eqref{eq:KeyTerm1}-~\eqref{eq:sigmaG1}}}
\end{flushleft}
\begin{proof}
We start with estimating $\dot\lambda,\ \dot\gamma-\Upsilon_{1,1}$: as in ~\cite{GaWe} pp. 291-293, we put the equations into a matrix form to find
\begin{equation}\label{eq:lambdagamma}
[Id+\Pi(z,\vec{R},a,p,q)]\left(
\begin{array}{lll}
\dot\lambda\\
\dot\gamma-\Upsilon_{1,1}
\end{array}
\right)=\Omega+Remainder\left(
\begin{array}{lll}
\delta_{\infty}\\
1
\end{array}
\right),
\end{equation} where, recall the definition of $Remainder$ in ~\eqref{eq:Remainder}, and the terms on the right hand side were obtained by the following arguments:
\begin{itemize}
\item[(1)]
the matrix $\Omega$ is defined as
$$
\Omega:=-\left(
\begin{array}{lll}
\partial_{t}a_1+iE(\lambda)\displaystyle\sum_{m+n=2,3}(m-n)A_{m,n}^{(1)},\\
\partial_{t}a_2+iE(\lambda)\displaystyle\sum_{m+n=2,3}(m-n)A_{m,n}^{(2)}
\end{array}
\right)
$$ and is controlled by applying the results in Lemma ~\ref{LM:apremainder}
\begin{equation}
\Omega=\cO[|z|(|\dot\lambda|+|\dot\gamma-\Upsilon_{1,1}|)+\delta_{\infty}Remainder]+\left(
\begin{array}{lll}
\mathcal{O}(\delta_{\infty}|z|^4)\\
\mathcal{O}(|z|^4)
\end{array}
\right);
\end{equation}
\item[(3)]
$Id$ is the $2\times 2$ identity matrix, the matrix $\Pi$ is defined as
$$\Pi(z,\vec{R},a,p,q):=\Pi_1+\Pi_{s} $$ where $\Pi_{s}$ is a matrix depending on $z,\vec{R},a,p$ and $q,$
and satisfies the estimate
$$\|\Pi_{s}(z,\vec{R},a,p,q)\|\lesssim |z_0|+\delta_{\infty}$$
by the conditions $\displaystyle\sum_{k=1}^{4}\mathcal{M}_{k}(t)\leq (\delta_{\infty}+\epsilon_{\infty})^{-\frac{1}{2}}$ of Proposition ~\ref{Prop:Majorants};
the matrix $\Pi_1$ is defined and estimated as $$\Pi_1:=\left(
\begin{array}{lll}
0&0\\
-\frac{\langle R_2,\partial_{\lambda}^2\phi^{\lambda}\rangle}{\langle \phi^{\lambda},\partial_{\lambda}\phi^{\lambda}\rangle},&0
\end{array}
\right)=\mathcal{O}(|z_0|^{\frac{1}{2}}).$$ To prove this we used the observations that $R_2=ReR\perp \partial_{\lambda}\phi^{\lambda}$ in ~\eqref{s-Rorthogonal} and $\partial_{\lambda}\phi^{\lambda}$ and $\partial_{\lambda}^{2}\phi^{\lambda}$ are almost colinear to each other, proved in ~\eqref{eq:LambdaPhi2}. By these and ~\eqref{eq:denominator} we obtain $\frac{1}{\langle \phi^{\lambda},\partial_{\lambda}\phi^{\lambda}\rangle}|\langle R_2,
\partial_{\lambda}^2\phi^{\lambda}\rangle|\leq \delta_{\infty}^{1-2\sigma}\|\langle x\rangle^{-4}\vec{R}\|_{2}.$ The assumption $\mathcal{M}_2\leq \delta_{\infty}^{-\frac{1}{2}}$ in Proposition ~\ref{Prop:Majorants} implies that $\|\langle x\rangle^{-4}\vec{R}\|_{2}\leq |z_0|^{\frac{3}{2}}.$ Consequently
\begin{equation}\label{eq:almostOrtho}
\frac{1}{\langle \phi^{\lambda},\partial_{\lambda}\phi^{\lambda}\rangle}|\langle R_2,
\partial_{\lambda}^2\phi^{\lambda}\rangle|\lesssim \delta_{\infty}^{1-2\sigma}|z_0|^{\frac{3}{2}}\leq |z_0|^{\frac{1}{2}}.
\end{equation}
\item[(2)]
the term $Remainder\left(
\begin{array}{lll}
\delta_\infty\\
1
\end{array}
\right)$ is produced by
$$\frac{1}{\langle\phi^{\lambda},\partial_{\lambda}\phi^{\lambda}\rangle}\left(
\begin{array}{lll}
\Upsilon_{1,1}\langle R_{2},\phi^{\lambda}\rangle+\Upsilon_{1,1}\langle q\cdot \eta,\phi^{\lambda}\rangle-\langle
ImN-\displaystyle\sum_{m+n=2,3}ImN_{m,n},\phi^{\lambda}\rangle\\
-\Upsilon_{1,1}\langle R_{1},\partial_{\lambda}\phi^{\lambda}\rangle-\Upsilon_{1,1}\langle p\cdot \xi,\partial_{\lambda}\phi^{\lambda}\rangle+\langle
ReN-\displaystyle\sum_{m+n=2,3}ReN_{m,n},\partial_{\lambda}\phi^{\lambda}\rangle
\end{array}
\right).$$ To prove this we use the results in Propositions ~\ref{Prop:Parameters} and ~\ref{Prop:NonlinearRem} and the fact $\langle ImN_{1,1},\phi^{\lambda}\rangle=0$. Here the ``almost orthogonality" or ``almost colinear condition" between functions $\xi_{k}$ and $\eta_{k}$, $\phi^{\lambda}$ and $\partial_{\lambda}\phi^{\lambda},$ implied by Proposition ~\ref{Prop:Parameters} were used to approximate the orthogonal conditions ~\eqref{eq:Orthogonality} and ~\eqref{s-Rorthogonal}. An example is in proving ~\eqref{eq:almostOrtho} above. We omit the details here.
\end{itemize}
Now inverting the matrix $$[Id+\Pi]^{-1}=Id-\Pi+\mathcal{O}(\Pi^2)$$ in ~\eqref{eq:lambdagamma} we obtain the desired estimates on $\dot\lambda$ and $\dot\gamma-\Upsilon_{1,1}$, which are ~\eqref{eq:RoughEstLambda}-~\eqref{eq:sigmaG1} except these on $\mathcal{K}$.

The estimate on $\mathcal{K}$ are implied by ~\eqref{eq:apq2} and ~\eqref{eq:RoughEstLambda}-~\eqref{eq:sigmaG1}.

By similar arguments we prove ~\eqref{eq:KeyTerm1}.

The proof is complete.
\end{proof}
\subsection{Proof of Equation ~\eqref{eq:ControlZ}}
\begin{proof} As usual we only prove the case $\sigma=1$, the cases $\sigma>1$ is easier by using the stronger condition $|z_{0}|\leq \delta_{\infty}^{C(\sigma)}$ for some sufficiently large $C(\sigma).$

By ~\eqref{eq:ZNequation} and the estimate for $\mathcal{K}$ in ~\eqref{eq:apq2}
\begin{equation}\label{eq:estiZ}
\frac{d}{dt}|z|^{2}=-2Re z^* \Gamma(z,\bar{z})z+2Re z^* \cdot Remainder
\end{equation} with $Re z^* \cdot Remainder
$ satisfying the estimate $$
\begin{array}{lll}
|Re z^* \cdot Remainder
|&\lesssim& \delta_{\infty}|z| Remainder\\
&\lesssim& \delta_{\infty}|z|^{5}+\delta_{\infty}|z|^2[|z_{0}|^2(1+t)^{-\frac{3}{2}}+
\delta_{\infty}|z_{0}|^2|z(t)|^2]\mathcal{M}_{3}\\
& &\\
& &+\delta_{\infty}|z|[|z_{0}|^4(1+t)^{-3}+\delta_{\infty}^2 |z_{0}|^4|z(t)|^4]\mathcal{M}_{3}^2.
\end{array}
$$

Observe that $|Re z^* \cdot Remainder
|$ is NOT a higher order correction of $Re z^* \Gamma(z,\bar{z})z=\mathcal{O}(\delta_{\infty}^2|z|^4)$ in a neighborhood of $t=0$. This forces us to divide the region $t\in [0,\infty)$ into two parts $t\leq \delta_{\infty}^{-2}|z_{0}|^{-2}$ and $t> \delta_{\infty}^{-2}|z_{0}|^{-2}.$

In the finite time interval we define $$|\tilde{z}|^2(t):=|z|^2(t) e^{-\int_{0}^{t}\frac{2 Re z^* \cdot Remainder
}{|z|^2(s)}ds}\approx |z|^2(t).$$ The assumptions $|z|\geq \frac{1}{10}|z_{-}(t)|\gtrsim (|z_{0}|^{-2}+\delta_{\infty}^2 t)^{-\frac{1}{2}}$ and $\mathcal{M}_{3}\leq \delta_{\infty}^{-\frac{1}{2}}$ imply that for any time $t\leq \delta_{\infty}^{-2}|z_{0}|^{-2}$,
\begin{equation}\label{eq:remark}
0\leq \int_{0}^{t}\frac{2 |Re z^* \cdot Remainder
|}{|z|^{2}}ds=\mathcal{O}(\frac{|z_0|}{\delta_{\infty}}) \ll 1.
 \end{equation} Moreover by ~\eqref{eq:estiZ} and the estimates of $Re z^* \Gamma(z,\bar{z})z$ in ~\eqref{eq:LowerUpperBound}
\begin{equation}\label{eq:UpperLower}
\frac{3}{4} C_{-}\delta_{\infty}^2 \leq \frac{d}{dt}|\tilde{z}|^{-2}\leq \frac{5}{4} C_{+}\delta_{\infty}^2.
\end{equation} Consequently when $t\leq \delta_{\infty}^{-2}|z_{0}|^{-2}$ we have the desired estimate
\begin{equation}\label{eq:claimZ}
2|z_{+}(t)|\geq |\tilde{z}(t)|\geq \frac{1}{2}|z_{-}(t)|,\ \text{hence}\ 3|z_{+}(t)|\geq |z(t)|\geq \frac{1}{3}|z_{-}(t)|
\end{equation} where, recall the definitions of $z_{\pm}(t)=(|z_{0}|^{-2}+C_{\pm}\delta_{\infty}^2 t)^{-\frac{1}{2}}$ in Equation ~\eqref{eq:difZpm}.

When $t\geq \delta_{\infty}^{-2}|z_{0}|^{-2}$ we consider ~\eqref{eq:estiZ} in the regime $[\delta_{\infty}^{-2}|z_{0}|^{-2},\infty)$ with the initial condition satisfying the estimate in ~\eqref{eq:claimZ}. This is easier by the fact the second term in ~\eqref{eq:estiZ} is a true correction to the first: Use $(1+t)^{-1}\leq 2(|z_{0}|^{-2}\delta_{\infty}^{-2}+t)^{-1}=2\delta_{\infty}^2(|z_{0}|^{-2}+\delta_{\infty}^2 t)^{-1}$ to find $\delta_{\infty}^2|z|^4\gg |Re z^* \cdot Remainder
|. $ Thus there exists an $0\leq \epsilon\ll 1$ such that $$-(2+\epsilon)Re z^* \Gamma(z,\bar{z})z\leq \frac{d}{dt}|z|^{2}\leq -(2-\epsilon)Re z^* \Gamma(z,\bar{z})z$$ which together with the condition in ~\eqref{eq:claimZ} at $t=\delta_{\infty}^{-2}|z_{0}|^{-2}$ enables us to obtain ~\eqref{eq:ControlZ}.

The proof is complete.
\end{proof}
\begin{remark}\label{Remark} In the proof above, specifically ~\eqref{eq:remark}, we used $|z_{0}|\ll \delta_{\infty}=\mathcal{O}(\|\phi^{\lambda_0}\|_{2})$ to show $\delta_{\infty}|z|^5\ll 2Re z^* \Gamma(z,\bar{z})z$. Actually this condition can be weaken to be $|z_0|\leq \|\phi^{\lambda_0}\|_2$ by refining normal form transformation: Namely examine closely the equation of $z$ to find that
$$\partial_{t}z+iE(\lambda)z=-\Gamma(z,\bar{z})z+\Lambda(z,\bar{z})z+\sum_{m+n=4}Z_{m,n}+Remainder2$$ where $Remainder2$ satisfies the estimate $$|Remainder2|\lesssim \delta_{\infty}|z|^5+\delta_{\infty}|z| \|\langle x\rangle^{-4 }R_{\geq 4}\|_{2}+\delta_{\infty}\|\langle x\rangle^{-4 }R_{\geq 4}\|_{2}^2.$$  The fourth order term $Z_{m,n}$ can be removed by choosing a new parameter $\tilde{z}$ by $$\tilde{z}=z+\displaystyle\sum_{m+n=4}\frac{1}{iE(\lambda)(m-n-1)}Z_{m,n}\approx z.$$ By studying the equation for $\tilde{z}$ we obtain the desired estimate.
\end{remark}

\subsection{The Estimate of $\|\vec{R}\|_{2}^{2}$: Proof of ~\eqref{eq:EstM4}}
We only prove the case $\sigma=1$, the cases $\sigma>1$ is easier by using the stronger condition $|z_{0}|\leq \delta_{\infty}^{C(\sigma)}$ for some sufficiently large $C(\sigma).$

By taking time derivative on $\|\vec{R}\|_{2}^{2}$ and using the equation for $\vec{R}$ in ~\eqref{RAfProj} we find
$$
\frac{d}{dt}\langle
\vec{R},\vec{R}\rangle
=K_{1}+K_2
$$ with $K_{n}$, $n=1,2,$ defined as $$K_{1}:=\langle
(L(\lambda)+\dot\gamma
J)\vec{R},\vec{R}\rangle +\langle \vec{R},
(L(\lambda)+\dot\gamma J)\vec{R}\rangle+\dot\lambda\langle
P_{c\lambda}\vec{R},\vec{R}\rangle+\dot\lambda\langle\vec{R},P_{c\lambda}\vec{R}\rangle;$$
$$K_{2}:=-\langle
P_{c}^{\lambda}JN(\vec{R},z),\vec{R}\rangle-\langle
\vec{R},P_{c}^{\lambda}JN(\vec{R},z)\rangle+\langle P_{c}^{\lambda}\mathcal{G},
\vec{R}\rangle+\langle
\vec{R},P_{c}^{\lambda}\mathcal{G}\rangle.$$

By the observation $J^{*}=-J$ and the fact
that $JL(\lambda)$ is self-adjoint we cancel all the nonlocalized terms in $K_{1}$ and obtain: $$|K_{1}|\lesssim \|\langle x\rangle^{-4}\vec{R}\|_{2}^2.$$
Recall the definition of
$P_{c}^{\lambda}\mathcal{G}$ in (~\ref{RAfProj}). By various estimates in Proposition ~\ref{Prop:Parameters} and the estimates in previous subsections we obtain
$$|K_{2}|\lesssim \delta_{\infty}|z|^{2}\|\langle x\rangle^{-4}\vec{R}\|_{2}+\|\langle x\rangle^{-4 }\vec{R}\|_{2}^{2}+\|\vec{R}\|_{4}^{4}.$$

Consequently
$$
\begin{array}{lll}
|\frac{d}{dt}\langle \vec{R},\vec{R}\rangle|&\lesssim & \delta_{\infty}|z|^{2}\|\langle x\rangle^{-4 }\vec{R}\|_{2}+\|\langle x\rangle^{-4 }\vec{R}\|_{2}^{2}+\|\vec{R}\|_{4}^{4}\\
&\lesssim &\delta_{\infty}^{2}|z|^{4}+\|\langle x\rangle^{-4 }\vec{R}\|_{2}^{2}+\|\vec{R}\|_{4}^{4}.
\end{array}
$$
Integrate the equation from $0$ to $T$ and use the fact $\|\vec{R}(0)\|_{2}^2\lesssim |z_0|^2$ to obtain
\begin{equation}\label{eq:L2Control}
 \|\vec{R}(T)\|_{2}^{2}\lesssim |z_0|^2+\int_{0}^{T}\ \delta_{\infty}^{2}|z|^{4}(s)+\|\langle x\rangle^{-4 }\vec{R}(s)\|_{2}^{2}+\|\vec{R}(s)\|_{4}^{4}\ ds.
\end{equation}

Now we estimate the different terms inside integral.

We start with the first term. By the assumption $10 |z_+|\geq |z|\geq \frac{1}{10}|z_{-}|$ we obtain $$\int_{0}^{T}\delta_{\infty}^{2}|z(s)|^{4}ds\lesssim \int_{0}^{T}\delta_{\infty}^{2}(|z_{0}|^{-2}+\delta_{\infty}^{2}s)^{-2} ds\leq |z_{0}|^{2}.$$

For the second term we use the definition $\vec{R}=\displaystyle\sum_{m+n=2,3}R_{m,n}+ R_{\geq 4}$ to obtain $$\|\langle x\rangle^{-4 }\vec{R}\|_{2}\leq \sum_{m+n=2,3}\|\langle x\rangle^{-4 }R_{m,n}\|_{2}+\|\langle x\rangle^{-4 }R_{\geq 4}\|_{2}.$$

The two terms on the right hand sides admit the following estimate.
\begin{enumerate}
\item[(1)]
By $\|\langle x\rangle^{-4 }R_{m,n}\|_{2}\lesssim \delta_{\infty}|z|^{2}$ proved in Proposition ~\ref{Prop:Parameters} and the assumption $10|z_{+}|\geq |z|\geq \frac{1}{10}|z_{-}|$  $$\int_{0}^{T}\|\langle x\rangle^{-4 }R_{m,n}(s)\|_{2}^{2}\ ds\lesssim \delta_{\infty}^{2}\int_{0}^{\infty}(|z_{0}|^{-2}+\delta_{\infty}^{2}s)^{-2}\ ds=|z_{0}|^{2};$$
\item[(2)]By the definition of $\mathcal{M}_3$ $$\int_{0}^{T} \|\langle x\rangle^{-4 }R_{\geq 4}(s)\|_{2}^{2}\ ds \leq \mathcal{M}_{3}^2(T) |z_{0}|^3.$$
\end{enumerate}

For the third term $\int_{0}^T\|\vec{R}(s)\|_{4}^{4}\ ds$
$$\int_{0}^{T}\|\vec{R}(s)\|_{4}^{4}\ ds\leq \int_{0}^{T} \|\vec{R}(s)\|_{\infty}^{2} \|\vec{R}(s)\|_{2}^2\ ds\leq \mathcal{M}_{1}^{2}(T)\mathcal{M}_{4}^2(T)|z_{0}|^3.$$

Collecting the estimates above we complete the proof.

\subsection{The Estimate of $\|\langle x\rangle^{-4 }R_{\geq 4}\|_{2}$: Proof of ~\eqref{eq:EstM3}}
We only prove the case $\sigma=1$, the cases $\sigma>1$ is different, but easier by using the stronger condition $|z_{0}|\leq \delta_{\infty}^{C(\sigma)}$ for some sufficiently large $C(\sigma).$

Use the definition of $R_{\geq 4}$ in ~\eqref{dif:R4} and the equation ~\eqref{RAfProj} to derive an equation for $R_{\geq 4}$
\begin{equation}\label{eq:Rge4}
\partial_{t}R_{\geq 4}=L(\lambda)R_{\geq 4}+L_{(\dot\lambda,\dot\gamma)}R_{\geq 4}-P_{c}^{\lambda}JN_{\geq 4}+\mathcal{G}_{1}
\end{equation} with the terms $JN_{\geq 4}$ and $\mathcal{G}_1$ defined as: $$JN_{\geq 4}:=JN(\vec{R},p,z)-\sum_{m+n=2,3}JN_{m,n},$$
$$\mathcal{G}_{1}:=\mathcal{G}-\Upsilon_{1,1}P_{c}^{\lambda(t)}\left(
\begin{array}{lll}
\beta\cdot \eta\\
-\alpha\cdot\xi
\end{array}\right)+P_{c}S_{>4},$$ and
$$P_{c}S_{>4}:=L_{(\dot\lambda,\dot\gamma)}\sum_{m+n=2,3}R_{m,n}+\sum_{m+n=2,3}[L(\lambda)R_{m,n}-\frac{d}{dt}R_{m,n}-P_{c}^{\lambda}JN_{m,n}].$$
The functions in $\mathcal{G}_1$ take certain forms
\begin{lemma} There exist some functions $\phi(m,n,k)$ such that
\begin{equation}\label{eq:G1}
\mathcal{G}_{1}=\displaystyle\sum_{
\begin{subarray}{lll}
m+n=2,3,\\
k=0,1,2
\end{subarray}
}[L(\lambda)+iE(\lambda)(m-n)-0]^{-k}\phi(m,n,k)
\end{equation} where $\phi(m,n,k)$ are smooth functions admitting the estimate
\begin{equation}\label{eq:b11}
\|\langle x\rangle^{4}\phi(m,n,k)\|_{2}\lesssim |\dot\lambda ||z|+|z||\dot\gamma-\Upsilon_{1,1}|+\delta_{\infty}|z|^{4}.
\end{equation}
\end{lemma}
\begin{proof}
~\eqref{eq:G1} is taken from ~\cite{GaWe}, Theorem 8.1, p. 285.

~\eqref{eq:b11} is an improvement from ~\cite{GaWe} by the fact one has to find that the lowest order term in $\phi(m,n,k)$ is of the order $\delta_{\infty}|z|^4$. By Proposition ~\ref{Prop:Parameters} and direct computation we find it is generated by $\Upsilon_{1,1}\left(
\begin{array}{lll}
q\cdot\eta\\
-p\cdot\xi
\end{array}
\right),$ $\Upsilon_{1,1}\displaystyle\sum_{m+n=2}JR_{m,n}$ and part of $\displaystyle\sum_{m+n=2}[L(\lambda)R_{m,n}-\frac{d}{dt}R_{m,n}-P_{c}^{\lambda}JN_{m,n}]$. The procedure is tedious, but easy. We omit the detail here.

The proof is complete.
\end{proof}

Rewrite the equation for $R_{\geq 4}$ in ~\eqref{eq:Rge4} as
\begin{equation}\label{eq:R4}
\partial_{t}R_{\geq 4}=L(\lambda)R_{\geq 4}+L_{(\dot\lambda,\dot\gamma)}R_{\geq 4}-P_{c}^{\lambda}(Loc+NonLoc)+P_{c}^{\lambda}\mathcal{G}_{1}
\end{equation} where, recall the definitions of $Loc$ and $NonLoc$ in ~\eqref{eq:HiRemainder}.

Following the steps in ~\cite{GaWe}, p. 302, we now derive an integral equation for $P_c^{\lambda_1}R_{\geq 4}$. Rewrite $L(\lambda(t))$ as
$L(\lambda(t))=L(\lambda_{1})+L(\lambda(t))-L(\lambda_{1})$ with $\lambda_1:=\lambda(T)$ for some fixed time $T$, and rewrite (~\ref{eq:R4}) one more time to obtain
\begin{equation}\label{estimater21r22}
\begin{array}{lll}
\frac{d}{dt}P_{c}^{\lambda_{1}}R_{\geq 4}&=&L(\lambda_{1})P_{c}^{\lambda_{1}}
R_{\geq 4}+[\dot{\gamma}+\lambda-\lambda_{1}]i(P_{+}-P_{-})R_{\geq 4}\\
&
&+P_{c}^{\lambda_{1}}O_{1}R_{\geq 4}+P_{c}^{\lambda_{1}}\mathcal{G}_1-P_{c}^{\lambda_{1}}[Loc+NonLoc].
\end{array}
\end{equation}
Here for the terms on the right hand side we have
\begin{itemize}
\item[(1)]
$O_{1}$ is the operator defined by
$$
O_{1}:=\dot{\lambda}P_{c\lambda}+L(\lambda)-L(\lambda_{1})+
\dot{\gamma}P_{c}^{\lambda}J-[\dot{\gamma}+\lambda-\lambda_{1}]i(P_{+}-P_{-})
$$ and the function $O_1 R_{\geq 4}$ satisfies the estimate: when $t\leq T$ then apply Proposition ~\ref{Prop:Majorants} to obtain
\begin{equation}\label{EstO1}
\begin{array}{lll}
 \|\langle x\rangle^{4} O_1 R_{\geq 4}\|_2 &\lesssim& [|\dot\lambda|+|\dot\gamma|+|\lambda-\lambda_1|]\|\langle x\rangle^{-4} R_{\geq 4}\|_2\\
&\lesssim &[|z_{0}|^{2}(1+s)^{-\frac{3}{2}}+\delta_{\infty}|z_{0}|^{2}(|z_{0}|^{-2}+\delta_{\infty}^2 s)^{-1} [1+\delta_{\infty}\mathcal{M}_{3}+\delta_{\infty}^3\mathcal{M}_{3}^2]
\end{array}
\end{equation}
\item[(2)]
Recall that $L(\lambda)$ has two branches of essential
spectrum $[i\lambda,i\infty)$ and $(-i\infty,-i\lambda]$, we use
$P_{+}$ and $P_{-}$ to denote the projection operators onto these
two branches of the essential spectrum of $L(\lambda(T))$. 
\end{itemize}
Then we
have
\begin{lemma}\label{LM:ApproOp} For any function $h$ and any large constant $\nu>0$ we have
\begin{equation}\label{eq:PosiNegBran}
\left\|\ \langle x\rangle^{\nu}\left(\ P_{c}^{\lambda_{1}}J-i(P_{+}-P_{-})\ \right)h\ \right\|_{2}\leq
c\ \left\|\ \langle x\rangle^{-4 }h\ \right\|_{2}.
\end{equation}
\end{lemma}

The following estimates are taken from ~\cite{GaWe}, Theorem 5.7, p. 280.
\begin{lemma} There exists a constant $c$ such that if the parameter $\lambda_1$ satisfies the estimate $|\lambda_1+e_0|\ll 1$, then for any function $h$ and $t\geq 0$ we have
\begin{equation}\label{eq:SingularEst}
\|\langle x\rangle^{-4}e^{tL(\lambda_1)}(L(\lambda_1)\pm
ikE(\lambda_1)-0)^{-n}P_{\pm}h\|_{2}
\ \ \ \ \leq c(1+t)^{-\frac{3}{2}}    \|\langle x\rangle^{4}h\|_{2}
\end{equation} with $n=0,1,2,\ k=2,3$;
\begin{equation}
\|\langle x\rangle^{-4}e^{tL(\lambda_1)} P_{\pm} h\|_{2}\lesssim
(1+|t|)^{-\frac{3}{2}}(\|h\|_{1}+\|h\|_{2}).
\label{eq:Regular}
\end{equation}
\end{lemma}

Apply the Duhamel's principle on Equation (~\ref{estimater21r22}) and use the observation that the
operators $P_{+},$ $P_{-}$ and $L(\lambda_{1})$ commute with each
other to find
\begin{equation}\label{rewriter21r22}
\|\langle x\rangle^{-4}P_{c}^{\lambda_{1}}R_{\geq 4}\|_{2}
\leq \sum_{l=1}^{4}A_{l},
\end{equation} with $$A_{1}:=\|\langle x\rangle^{-4}e^{tL(\lambda_{1})+i a(t,0)(P_{+}-P_{-})}P_{c}^{\lambda_{1}}R_{\geq 4}(0)\|_{2}$$ where $a(t,s)\ := \int_{s}^{t}
[\dot{\gamma}(\tau)+\lambda(\tau)-\lambda_{1}]
\ d\tau\in \mathbb{R},$
$$A_{2}:=\int_{0}^{t}\|\langle x\rangle^{-4}e^{(t-s)L(\lambda_{1})+ia(t,s)(P_{+}-P_{-})}P_{c}^{\lambda_{1}}[O_{1}R_{\geq 4}-P_{c}^{\lambda}Loc]\|_{2}ds,$$
$$A_{3}:=\int_{0}^{t}\|\langle x\rangle^{-4}e^{(t-s)L(\lambda_{1})+ia(t,s)(P_{+}-P_{-})}P_{c}^{\lambda_{1}}P_{c}^{\lambda}\mathcal{G}_{1}(s)\|_{2}ds,$$
and
$$A_{4}:=\int_{0}^{t}\|\langle x\rangle^{-4}e^{(t-s)L(\lambda_{1})+ia(t,s)(P_{+}-P_{-})}P_{c}^{\lambda_{1}}NonLoc\|_{2} ds.$$
\\ \\

Now we estimate $A_{k}, k=1,2,3,4.$ In what follows we use repeatedly the assumption $10 |z_+|\geq |z|\geq \frac{1}{10}|z_{-}|$ and the fact
$$
e^{ i a(t_{2},t_{1})(P_{+}-P_{-})}
=e^{i a(t_{2},t_{1})}P_{+}+e^{-i a(t_{2},t_{1})}P_{-}:
\langle x\rangle^{k} L^{2} \rightarrow \langle x\rangle^{k} L^{2}
$$ is uniformly bounded for any $k\in \mathbb{R}$.
\begin{itemize}
\item[(1)] By the propagator estimate ~\eqref{eq:SingularEst} we have
$$
\begin{array}{lll}
A_{1}&\leq &\|\langle x\rangle^{-4} e^{tL(\lambda(T))}P_{\pm}\vec{R}(0)\|_{2}+\|\langle x\rangle^{-4} e^{tL(\lambda(T))}P_{\pm}\displaystyle\sum_{m+n=2,3}R_{m,n}(0)\|_{2}\\
&\lesssim & (1+t)^{-\frac{3}{2}}[\|\langle x\rangle^{4}\vec{R}(0)\|_2+|z_0|^2]\\
&\lesssim &(1+t)^{-\frac{3}{2}}|z_{0}|^2.
\end{array}
$$
\item[(2)]
Equation ~\eqref{eq:IntegralKernel} below and the estimates of $O_{1}R_{\geq 4}$ in ~\eqref{EstO1}, $Loc$ in Proposition ~\ref{Prop:NonlinearRem} imply that for any time $t\leq T$ $$
\begin{array}{lll}
A_{2}&\lesssim &\int_{0}^{t}(1+t-s)^{-\frac{3}{2}}[|z_{0}|^{2}(1+s)^{-\frac{3}{2}}+\delta_{\infty}|z_{0}|^{2}(|z_{0}|^{-2}+\delta_{\infty}^2 s)^{-1} [1+\delta_{\infty}\mathcal{M}_{3}+\delta_{\infty}^3\mathcal{M}_{3}^2]\ ds\\
& \lesssim & [|z_{0}|^{2}(1+t)^{-\frac{3}{2}}+\delta_{\infty}|z_{0}|^{2}(|z_{0}|^{-2}+\delta_{\infty}^2 t)^{-1}] [1+\delta_{\infty}\mathcal{M}_{3}+\delta_{\infty}^3\mathcal{M}_{3}^2].
\end{array}
$$
\item[(3)]
Equation ~\eqref{eq:SingularEst} and $\mathcal{G}_{1}$ of Equation ~\eqref{eq:G1} imply
$$
\begin{array}{lll}
A_{3}&\lesssim&
\int_{0}^{t}{(1+t-s)^{-\frac{3}{2}}}[\delta_{\infty}|z(s)|^{4}+Remainder(s)|z(s)|] ds\\
& \lesssim &[|z_{0}|^{2}(1+s)^{-\frac{3}{2}}+\delta_{\infty}|z_{0}|^{2}(|z_{0}|^{-2}+\delta_{\infty}^2 s)^{-1}] [1+\delta_{\infty}\mathcal{M}_{3}+\delta_{\infty}^3\mathcal{M}_{3}^2] .
\end{array}
$$
\item[(4)] We control $A_{4}$ by the form of $NonLoc$ in Proposition ~\ref{Prop:NonlinearRem} and the propagator estimate ~\eqref{eq:Regular}:
$$
\begin{array}{lll}
A_{4}&\lesssim &\int_{0}^{t}(1+t-s)^{-\frac{3}{2}}[\|\vec{R}(s)\|_{3}^{3}+\|\vec{R}(s)\|_{6}^3]ds\\
&\lesssim &\int_{0}^{t}(1+t-s)^{-\frac{3}{2}}[\|\vec{R}(s)\|_{2}^{2}\|\vec{R}\|_{\infty}+\|\vec{R}(s)\|_{\infty}^2\|\vec{R}\|_{2}]ds\\
&\lesssim & |z_{0}|^{2}\int_{0}^{t}(1+t-s)^{-\frac{3}{2}}[(1+s)^{-3/2}+\delta_{\infty}(|z_{0}|^{-2}+\delta_{\infty}^{2}s)^{-1}]ds [\mathcal{M}_{4}^2\mathcal{M}_1+ \mathcal{M}_4\mathcal{M}_{1}^2].
\end{array}
$$
Apply ~\eqref{eq:IntegralKernel} to obtain
$$A_{4}\lesssim |z_{0}|^2[(1+t)^{-\frac{3}{2}}+\delta_{\infty} (|z_{0}|^{-2}+\delta_{\infty}^{2}t)^{-1}] [\mathcal{M}_{4}^2\mathcal{M}_1+ \mathcal{M}_4\mathcal{M}_{1}^2].$$
\end{itemize}
Collecting the estimates above we have $$
\|\langle x\rangle^{-4 }R_{\geq 4}\|_{2}\lesssim [|z_{0}|^2 (1+t)^{-\frac{3}{2}}+ \delta_{\infty}|z_{0}|^{2}(|z_{0}|^{-2}+\delta_{\infty}^{2}t)^{-1}][1+\mathcal{M}_{4}\mathcal{M}_{1}^2+\mathcal{M}_4^2\mathcal{M}_1+\delta_{\infty}(\mathcal{M}_{3}+\mathcal{M}_{3}^2)].
$$

The proof is complete by the definition of $\mathcal{M}_{3}$ in Equation ~\eqref{eq:EstM3}.

\begin{flushright}
$\square$
\end{flushright}
In the proof we used the following lemma.
\begin{lemma}
\begin{equation}\label{eq:IntegralKernel}
\int_{0}^{t} \frac{1}{(1+t-s)^{\frac{3}{2}}} (|z_{0}|^{-2}+\delta_{\infty}^{2}s)^{-1}\ ds\lesssim (|z_{0}|^{-2}+\delta_{\infty}^{2}t)^{-1}.
\end{equation}
\end{lemma}
\begin{proof}
We divide the regime into two parts $0\leq s\leq \frac{t}{2}$ and $\frac{t}{2}\leq s\leq t.$ For the latter
$$\int_{\frac{t}{2}}^{t} \frac{1}{(1+t-s)^{\frac{3}{2}}} (|z_{0}|^{-2}+\delta_{\infty}^{2}s)^{-1}\ ds\leq 2(|z_{0}|^{-2}+\delta_{\infty}^{2}t)^{-1} \int_{\frac{t}{2}}^{t} \frac{1}{(1+t-s)^{\frac{3}{2}}}ds\lesssim (|z_{0}|^{-2}+\delta_{\infty}^{2}t)^{-1}.$$

The estimate of the first is more involved. By direct computation we have
$$
\begin{array}{lll}
I:&=&\int_{0}^{\frac{t}{2}} \frac{1}{(1+t-s)^{\frac{3}{2}}} (|z_{0}|^{-2}+\delta_{\infty}^{2}s)^{-1}\ ds \\
& \leq& 2\sqrt{2} (1+t)^{-\frac{3}{2}} \int_{0}^{\frac{t}{2}}  (|z_{0}|^{-2}+\delta_{\infty}^{2}s)^{-1}\ ds\\
&=&2\sqrt{2}\delta_{\infty}^{-2} (1+t)^{-\frac{3}{2}}ln[1+\frac{1}{2}|z_{0}|^2\delta_{\infty}^2 t].
\end{array}
$$ Change variable $u=\delta_{\infty}^2 |z_{0}|^2 t$ to obtain $$
\begin{array}{lll}
I&\lesssim&  \delta_{\infty}^{-2}(1+\delta_{\infty}^{-2}|z_{0}|^{-2}u)^{-\frac{3}{2}}ln(1+u)\\
&\lesssim &|z_{0}|^2 u^{-1}(1+u)^{-\frac{1}{2}}ln(1+u)\\
&\lesssim& |z_{0}|^2(1+u)^{-1}=(|z_{0}|^{-2}+\delta_{\infty}^2 t)^{-1}.
\end{array}
$$
The proof is complete.
\end{proof}
\section{Proof of ~\eqref{eq:approxPos}}\label{sec:approxPos}
\begin{proof}
In the next we study $z^* \Gamma(z,\bar{z})z$ defined in ~\eqref{energy-id}. To prepare for the proof we list a few estimates.
\begin{itemize}
\item[(1)]
If $n=-2,-1,0,1,2$ then
\begin{equation}\label{mapping}
[L(\lambda)+in E(\lambda)+0]^{-1}P_{c}=[(-\Delta+V+\lambda)J+inE(\lambda)+0]^{-1}P_{c}^{lin}\left(
\begin{array}{ll}
1&0\\
0&1
\end{array}
\right)+\mathcal{O}(\|\langle x\rangle^{4}\phi^{\lambda}\|^{2\sigma}_{H^{2}})
\end{equation} as operators mapping from space $\langle x\rangle^{-4}L^{\infty}$ to space $\langle x\rangle^{-4}L^{\infty}.$ This is resulted from the estimate of $P_{c}-P_{c}^{lin}$ in ~\eqref{eq:projection} and the fact $L(\lambda)$ defined in ~\eqref{eq:opera} is of the form $L(\lambda)=(-\Delta+V+\lambda)J+\mathcal{O}((\phi^{\lambda})^{2\sigma}).$
\item[(2)]
To diagonalize the matrix $J$ we define a unitary $2\times 2$ matrix $U$ as
\begin{equation}\label{eq:MatrixA}
U:=\frac{1}{\sqrt{2}}\left(
\begin{array}{lll}
1&i\\
i&1
\end{array}
\right)
\end{equation} which makes $$U^*J U=i\sigma_{3}$$ with $\sigma_{3}$ being the Pauli matrix.
\item[(3)]
Apply ~\eqref{eq:asympto} to derive the leading orders from $\displaystyle\sum_{k=1}^{N}z_k G_{k}$ to obtain
\begin{equation}\label{eq:higherOrd}
\begin{array}{lll}
\displaystyle\sum_{k=1}^{N}z_k G_{k}&=&JN_{2,0}\\
&=&\sigma (\phi^{\lambda})^{2\sigma-1}\left(
\begin{array}{lll}
-2i (z\cdot\xi)(z\cdot \eta)\\
-(2\sigma+1) (z\cdot \xi)^2+ (z\cdot\eta)^2
\end{array}
\right)\\
&=& -2\sigma \delta_{\infty}^{2\sigma-1} (\phi_{lin})^{2\sigma-1} (z\cdot\xi^{lin})^2 \left(
\begin{array}{lll}
i\\
\sigma
\end{array}
\right)+\mathcal{O}(\delta_{\infty}^{4\sigma-1}|z|^2).
\end{array}
\end{equation}
\end{itemize}

Now we begin perturbation-expanding the function $z^* \Gamma(z,\bar{z}) z$ in the variable $\delta(\lambda).$ Use ~\eqref{mapping}-~\eqref{eq:higherOrd} to obtain $$
\begin{array}{lll}
& &z^* \Gamma(z,\bar{z})z\\
&=&-4\sigma^2 \delta^{4\sigma-2}(\lambda) Re\langle [(-\Delta+V+\lambda)J+2iE(\lambda)-0]^{-1}P_c^{lin}(\phi_{lin})^{2\sigma-1}(z\cdot\xi^{lin})^2 \left(
\begin{array}{lll}
i\\
\sigma
\end{array}
\right),\\
& &i(\phi_{lin})^{2\sigma-1}(z\cdot\xi^{lin})^2 J \left(
\begin{array}{lll}
i\\
\sigma
\end{array}
\right)\rangle+\mathcal{O}[\delta_{\infty}^{4\sigma-1}|z|^4]\\
&=&-4\sigma^2\delta^{4\sigma-2}(\lambda) Re\langle U^* [(-\Delta+V+\lambda)J+2iE(\lambda)-0]^{-1}U \ P_c^{lin}(\phi_{lin})^{2\sigma-1}(z\cdot\xi^{lin})^2 U^{*}\left(
\begin{array}{lll}
i\\
\sigma
\end{array}
\right),\\
& &(\phi_{lin})^{2\sigma-1}(z\cdot\xi^{lin})^2 U^* J i\left(
\begin{array}{lll}
i\\
\sigma
\end{array}
\right)\rangle+\mathcal{O}[\delta_{\infty}^{4\sigma-1}|z|^4].
\end{array}
$$ In the new expression certain terms can be computed explicitly: \begin{equation}\label{eq:diagonal}
\begin{array}{lll}
& &U^*[(-\Delta+V+\lambda)J+2iE(\lambda)-0]^{-1} U\\
&=&\left(
\begin{array}{lll}
-i[-\Delta+V+\lambda +2E(\lambda)]^{-1}&0\\
0& i[-\Delta+V+\lambda-2E(\lambda)-i0]^{-1}
\end{array}
\right),
\end{array}
\end{equation} and
\begin{equation}\label{eq:EqZ}
i U^*J \left(
\begin{array}{lll}
 i\\
 \sigma
 \end{array}
\right)=\frac{1}{\sqrt{2}}\left(
\begin{array}{lll}
i(\sigma-1)\\
\sigma+1
\end{array}
\right),\ \ U^* \left(
\begin{array}{lll}
 i\\
 \sigma
 \end{array}
\right)=\frac{1}{\sqrt{2}}\left(
\begin{array}{lll}
i(1-\sigma)\\
\sigma+1
\end{array}
\right)
\end{equation}

Put this back into the expression to find
\begin{equation}\label{eq:perExp}
z^* \Gamma(z,\bar{z}) z=z^* \Gamma_0(z,\bar{z}) z+H_{2,2}+\mathcal{O}(\delta_{\infty}^{4\sigma-1}|z|^4)
\end{equation} with $$
\begin{array}{lll}
H_{2,2}&:=&-2\delta^{4\sigma-2}(\lambda)\sigma^2(\sigma-1)^2\times\\
 & &\\
 & &Re\langle i(-\Delta+V+\lambda+2E(\lambda))^{-1}P_{c}^{lin}(\phi_{lin})^{2\sigma-1}(z\cdot\xi^{lin})^2, (\phi_{lin})^{2\sigma-1}(z\cdot\xi^{lin})^2\rangle.
\end{array}
$$
We observe that
\begin{equation}\label{eq:H22}
H_{2,2}=0
\end{equation}
by the fact that the operator $(-\Delta+V+\lambda+2E(\lambda))^{-1}P_{c}^{lin}$ is self-adjoint, hence that $\langle (-\Delta+V+\lambda+2E(\lambda))^{-1}P_{c}^{lin}\Omega, \Omega\rangle$ is real for any $\Omega.$
 Hence ~\eqref{eq:perExp} is the desired estimate and the proof is complete.
\end{proof}





 %


\section{Proof of Equation ~\eqref{eq:small}}\label{sec:periodic}
Recall the ideas present after Lemma ~\ref{LM:junk}, which basically is that the functions $\Pi_{m,n},\ m\not=n$ are almost periodic with period $\frac{2\pi}{E(m-n)}$.

Compute directly to obtain
\begin{equation}\label{eq:IntegrateParts}
\begin{array}{lll}
\int_{0}^{\infty}\Pi_{m,n}\ ds&=&\int_{0}^{\infty}\Pi_{m,n}+\frac{d}{ds}[\frac{1}{iE(\lambda)(m-n)}\Pi_{m,n}]\ ds-\int_{0}^{\infty}\frac{d}{ds}[\frac{1}{iE(\lambda)(m-n)}\Pi_{m,n}]\ ds\\
& &\\
&=&\int_{0}^{\infty}\dot{\lambda}\partial_{\lambda}(\frac{1}{iE(\lambda)(m-n)}\Pi_{m,n})
+\frac{1}{iE(\lambda)(m-n)}\partial_{z}\Pi_{m,n}\cdot[\dot{z}
+iE(\lambda)z]\\
& &\\
& &+\frac{1}{iE(\lambda)(m-n)}\partial_{\bar{z}}\Pi_{m,n}\cdot[\dot{\bar{z}}-iE(\lambda)\bar{z}]\ ds+\frac{1}{iE(\lambda)(m-n)}\Pi_{m,n}|_{t=0}.
\end{array}
\end{equation} It is easy to see that the last term satisfies the estimate
$$\frac{1}{iE(\lambda)(m-n)}\Pi_{m,n}|_{t=0}=o(|z_0|^2).$$

Collecting the estimates above we prove ~\eqref{eq:small}.

The proof is complete.
\begin{flushright}
$\square$
\end{flushright}

\section{Proof of Equation ~\eqref{eq:KeyTerm}}\label{sec:compare}
To facilitate our discussions we define
$$\rho:=z\cdot\xi,\ \omega:=z\cdot\eta.$$

Recall the definition of $\Pi_{2,2}$ in ~\eqref{eq:DetailLambda}. By direct computation we find
\begin{equation}\label{eq:Lambda22}
\Pi_{2,2}=-\langle
JN_{2,2},\left(
\begin{array}{lll}
\phi^{\lambda}\\
0
\end{array}\right)
\rangle+\Upsilon_{1,1}\langle
R_{1,1},\left(
\begin{array}{lll}
0\\
\phi^{\lambda}
\end{array}\right)
\rangle+\Upsilon_{1,1}\sum_{n=1}^{N}Q_{1,1}^{(n)}\langle \phi^{\lambda},\eta^{n}\rangle.
\end{equation}

Further expand the first
term to obtain
\begin{equation}\label{eq:DecN22}
\left\langle
JN_{2,2}\ ,\ \left(
\begin{array}{lll}
\phi^{\lambda}\\
0
\end{array}\right)
\right\rangle=\sum_{n=1}^{7}(\Phi_{n}+\overline{\Phi_{n}})+\Omega_{\sigma>1}
\end{equation}
with
$$\Phi_{1}:=\left\langle R_{2,0},\sigma (\phi^{\lambda})^{2\sigma-1}\left(
\begin{array}{lll}
-\frac{i}{2}(2\sigma-1)\rho\omega\\
-\frac{3}{4}\omega^{2}+\frac{1}{4} (2\sigma-1)\rho^{2}
\end{array}
\right)\right\rangle,$$
$$\Phi_{2}:=\left\langle \left(
\begin{array}{lll}
A_{2,0}^{(1)}\partial_{\lambda}\phi^{\lambda}+\displaystyle\sum_{n=1}^{N}P_{2,0}^{(n)}\xi_{n}\\
A_{2,0}^{(2)}\phi^{\lambda}+\displaystyle\sum_{n=1}^{N}Q_{2,0}^{(n)}\eta_{n}
\end{array}
\right),\sigma (\phi^{\lambda})^{2\sigma-1}\left(
\begin{array}{lll}
-\frac{i}{2}(2\sigma-1)\rho\omega\\
-\frac{3}{4}\omega^{2}+\frac{1}{4} (2\sigma-1)\rho^{2}
\end{array}
\right)
\right\rangle,$$
$$ \Phi_{3}:=\left\langle \left(
\begin{array}{lll}
\displaystyle\sum_{n=1}^{N}P_{1,1}^{(n)}\xi_{n}\\
\displaystyle\sum_{n=1}^{N}Q_{1,1}^{(n)}\eta_{n}
\end{array}
\right)+R_{1,1},\sigma(\phi^{\lambda})^{2\sigma-1}\left(
\begin{array}{lll}
\frac{i}{2}(2\sigma-1)(\rho\bar\omega-\bar\rho\omega)\\
\frac{3}{2}\omega\bar\omega+\frac{1}{2}(2\sigma-1)\rho\bar\rho
\end{array}
\right)\right\rangle,$$
$$\Phi_{4}:=\left\langle \left(
\begin{array}{lll}
\displaystyle\sum_{n=1}^{N}P_{2,1}^{(n)}\xi_{n}\\
\displaystyle\sum_{n=1}^{N}Q_{2,1}^{(n)}\eta_{n}
\end{array}
\right)+\left(
\begin{array}{lll}
A_{2,1}^{(1)}\partial_{\lambda}\phi^{\lambda}\\
A_{2,1}^{(2)}\phi^{\lambda}
\end{array}
\right),\left(
\begin{array}{lll}
-i\sigma (\phi^{\lambda})^{2\sigma}\omega\\
\sigma (\phi^{\lambda})^{2\sigma}\rho
\end{array}
\right)\right\rangle,$$
$$\Phi_{5}:=\left\langle R_{2,1},\left(
\begin{array}{lll}
-i\sigma (\phi^{\lambda})^{2\sigma}\omega\\
\sigma (\phi^{\lambda})^{2\sigma}\rho
\end{array}
\right)\right\rangle,$$
$$\Phi_{6}:=\left\langle 2\sigma(\phi^{\lambda})^{2\sigma}[A_{2,0}^{(1)}\partial_{\lambda}\phi^{\lambda}+
\sum_{n=1}^{N}P_{2,0}^{(n)}\xi_{n}+R_{2,0}^{(1)}],
[A_{2,0}^{(2)}\phi^{\lambda}+\sum_{n=1}^{N}Q_{2,0}^{(n)}\eta_{n}+R_{2,0}^{(2)}]\right\rangle,$$
$$\Phi_{7}:=\sigma\ \left\langle (\phi^{\lambda})^{2\sigma}[\sum_{n=1}^{N}P_{1,1}^{(n)}\xi_{n}+R^{(1)}_{1,1}],
\sum_{n=1}^{N}Q_{1,1}^{(n)}\eta_{n}+R_{1,1}^{(2)}\right \rangle,$$
$\Omega_{\sigma>1}$ only appears in $\sigma>1$
$$\Omega_{\sigma>1}:=\frac{1}{8i}\sigma(\sigma-1)(2\sigma-1)\langle (\phi^{\lambda})^{2\sigma-2}|\rho|^2, \rho\bar\omega-\bar\rho \omega\rangle-\frac{3}{8i}\sigma(\sigma-1)\langle (\phi^{\lambda})^{2\sigma-2}|\omega|^2 ,\bar\omega \rho-\bar\rho \omega\rangle.$$
The terms defined above satisfy the following estimate:
\begin{lemma}\label{LM:Junk}
\begin{equation}\label{eq:JunkTerms2}
\sum_{n=2}^{7}|\Phi_{n}+\overline{\Phi_{n}}|+|\Upsilon_{1,1}\langle
R^{(2)}_{1,1},\phi^{\lambda}\rangle|+|\Upsilon_{1,1}\sum_{n=1}^{d}Q_{1,1}^{(n)}\langle \phi^{\lambda},\eta^{n}\rangle|+|\Omega_{\sigma>1}| \lesssim \delta_{\infty}^{4\sigma-1}|z|^4,
\end{equation}
\end{lemma}
The lemma will be proved in the later part of this section.

Now we prove ~\eqref{eq:KeyTerm}.\\
{\bf{Proof of Equation ~\eqref{eq:KeyTerm}}}
Collecting the estimates in Equations ~\eqref{eq:Lambda22}, ~\eqref{eq:DecN22} and Lemma ~\ref{LM:Junk} we have
\begin{equation}\label{eq:LastStep}
\Pi_{2,2}=-2Re\langle R_{2,0}, K_{2,0}\rangle+\mathcal{O}(\delta_{\infty}^{4\sigma-1}|z|^4)
\end{equation} with $$
\begin{array}{lll}
K_{2,0}&:=&\sigma (\phi^{\lambda})^{2\sigma-1}\left(
\begin{array}{lll}
-\frac{i}{2}(2\sigma-1)\rho\omega\\
-\frac{3}{4}\omega^{2}+\frac{1}{4} (2\sigma-1)\rho^{2}
\end{array}
\right)\\
&=&\frac{\sigma}{2}(\phi^{\lambda})^{2\sigma-1}\rho^2\left(
\begin{array}{lll}
-i(2\sigma-1)\\
\sigma-2
\end{array}
\right)+\mathcal{O}[\delta_{\infty}^{4\sigma-1}|z|^2]\\
&=&\frac{\sigma}{2}\delta^{2\sigma-1}(\lambda)\phi^{lin}(z\cdot\xi^{lin})^2\left(
\begin{array}{lll}
-i(2\sigma-1)\\
\sigma-2
\end{array}
\right)+\mathcal{O}[\delta_{\infty}^{4\sigma-1}|z|^2]
\end{array}
$$ and recall $\delta(\lambda)$ defined in ~\eqref{eqn:perturb}.

Use ~\eqref{mapping} to expand $R_{2,0}$ in the space $\langle x\rangle^{4}L^{\infty}$ $$
\begin{array}{lll}
R_{2,0}&=&(L(\lambda)+2iE(\lambda)-0)^{-1}P_{c}JN_{2,0}\\
&=&[(-\Delta+V+\lambda)J+2iE(\lambda)-0]^{-1}P_{c}^{lin} JN_{2,0}+\mathcal{O}(\delta_{\infty}^{4\sigma-1}|z|^2).
\end{array}
$$ The estimate for $JN_{2,0}$ in ~\eqref{eq:higherOrd} implies $$
-2Re\langle R_{2,0}, K_{2,0}\rangle=M_{2,2}+\mathcal{O}(\delta_{\infty}^{4\sigma-1}|z|^4)$$ with $$\begin{array}{lll}
M_{2,2}&:=&2\sigma^2 \delta^{4\sigma-2}(\lambda)\langle [(-\Delta+V+\lambda)J+2iE(\lambda)-0]^{-1}P_{c}^{lin} (\phi_{lin})^{2\sigma-1}(z\cdot\xi^{lin})^{2}\left(
\begin{array}{lll}
i\\
\sigma
\end{array}
\right),\\
& & \phi^{lin}(z\cdot\xi^{lin})^2\left(
\begin{array}{lll}
-i(2\sigma-1)\\
\sigma-2
\end{array}
\right)\rangle
\end{array}
$$

To make the expression easier we diagonalize the matrix operator $[(-\Delta+V+\lambda)J+2iE(\lambda)-0]^{-1}$ which is essentially to diagonalize the matrix $J$. Recall the definition of the unitary matrix $U$ in ~\eqref{eq:MatrixA}.
Put $UU^*=Id$ into appropriate places to obtain
$$\begin{array}{lll}
M_{2,2}&=&2\sigma^2 \delta^{4\sigma-2}(\lambda)\langle U^*[(-\Delta+V+\lambda)J+2iE(\lambda)-0]^{-1} U \ P_{c}^{lin} (\phi_{lin})^{2\sigma-1}(z\cdot\xi^{lin})^{2}\ U^* \left(
\begin{array}{lll}
i\\
\sigma
\end{array}
\right),\\
& & \phi^{lin}(z\cdot\xi^{lin})^2\ U^*\left(
\begin{array}{lll}
-i(2\sigma-1)\\
\sigma-2
\end{array}
\right)\rangle.
\end{array}
$$ A few terms in the expression can be computed explicitly:
$U^*[(-\Delta+V+\lambda)J+2iE(\lambda)-0]^{-1} U$ was computed in ~\eqref{eq:diagonal}, $$U^{*}\left(
\begin{array}{lll}
i\\
\sigma
\end{array}
\right)=\frac{1}{\sqrt{2}}\left(
\begin{array}{lll}
i(1-\sigma)\\
1+\sigma
\end{array}
\right),\ \ \ U^{*}\left(
\begin{array}{lll}
-i(2\sigma-1)\\
\sigma-2
\sigma
\end{array}
\right)=-\frac{1}{\sqrt{2}}\left(
\begin{array}{lll}
3i(\sigma-1)\\
\sigma+1
\end{array}
\right).$$
Put this back into the expression of $M_{2,2}$ to obtain
\begin{equation}\label{eq:fiEsti}
M_{2,2}=\frac{1}{2}z^{*} \Gamma_{0}(z,\bar{z})z+\tilde{M}_{2,2}
\end{equation} with
$$
\begin{array}{lll}
\tilde{M}_{2,2}&:=&-3\delta^{4\sigma-2}(\lambda)\sigma^2(\sigma-1)^2\times\\
 & &\\
 & &Re\langle i(-\Delta+V+\lambda+2E(\lambda))^{-1}P_{c}^{lin}(\phi_{lin})^{2\sigma-1}(z\cdot\xi^{lin})^2, (\phi_{lin})^{2\sigma-1}(z\cdot\xi^{lin})^2\rangle,
\end{array}
$$ where, recall the definition of $z^{*} \Gamma_{0}(z,\bar{z})z$ in ~\eqref{eq:Gamma0}. Observe $\tilde{M}_{2,2}=0$ by the same argument as in proving ~\eqref{eq:H22}. Hence $$-2Re\langle R_{2,0}, K_{2,0}\rangle =\frac{1}{2}z^{*} \Gamma_{0}(z,\bar{z})z+\mathcal{O}(\delta_{\infty}^{4\sigma-1}|z|^4).$$

These together with ~\eqref{eq:LastStep} imply Equation ~\eqref{eq:KeyTerm}.

The proof is complete.
\begin{flushright}
$\square$
\end{flushright}
In the rest of this section we prove Lemma ~\ref{LM:Junk} by considering the cases $\sigma=1$ and $\sigma>1$ separately.
\begin{flushleft}
{\bf{Proof of Lemma ~\ref{LM:Junk} for $\sigma=1$}}
\end{flushleft}
In what follows we only study the term $\Phi_{5}$ and part of $\Phi_{4}.$ The estimates on the other terms are similar and easier.

We start with analyzing $JN_{2,1}$ since the terms are related to it. The definition in ~\eqref{eq:ThirdOrderJN} and the fact $\|\rho-\omega\|_{H^2}=\mathcal{O}(\delta_{\infty}^2|z|)$ imply that in the space $\langle x\rangle^{-4} H^2$
\begin{equation}\label{eq:JN21}
JN_{2,1}=-\frac{1}{2}\rho^2\bar\rho \left(
\begin{array}{lll}
i\\
1
\end{array}
\right)+\mathcal{O}(\delta_{\infty}|z|^3).
\end{equation}

Now we turn to $\Phi_{5}$. Recall the definition of $R_{2,1}$ from ~\eqref{eq:difRgeq3}. By the facts $P_{c}\left(
\begin{array}{lll}
\rho\\
0
\end{array}
\right)=P_{c}\left(
\begin{array}{lll}
0\\
\omega
\end{array}
\right)=0$ and $\rho-\omega=\mathcal{O}(\delta_{\infty}^2|z|)$ in ~\eqref{eq:asympto} we obtain $$\Upsilon_{1,1}P_{c}\left(
\begin{array}{lll}
i\omega\\
\rho
\end{array}
\right)=\Upsilon_{1,1}P_{c}\left(
\begin{array}{lll}
i(\omega-\rho)\\
\rho-\omega
\end{array}
\right)=\mathcal{O}(\delta_{\infty}^2|z|^3).$$
This together with ~\eqref{eq:JN21}, ~\eqref{mapping} implies that $$
\begin{array}{lll}
R_{2,1}&=&-\frac{1}{2}[(-\Delta+V+\lambda)J+iE(\lambda)]^{-1}P_{c}^{lin}\rho^2\bar\rho \left(
\begin{array}{lll}
i\\
1
\end{array}
\right)+\mathcal{O}(\delta_{\infty}|z|^3).
\end{array}
$$
Consequently
\begin{equation}\label{eq:RePhi5}
\Phi_{5}=\Theta_{2,2}+\mathcal{O}(\delta_{\infty}^3 |z|^4)
\end{equation} with $$
\begin{array}{lll}
\Theta_{2,2}&:=&-\frac{1}{2}\langle [(-\Delta+V+\lambda)J+iE(\lambda)]^{-1}P_{c}^{lin}\rho^2\bar\rho\left(
\begin{array}{lll}
i\\
1
\end{array}
\right), (\phi^{\lambda})^{2}\rho\left(
\begin{array}{lll}
-i\\
1
\end{array}
\right)\rangle.
\end{array}
$$

Now we claim $$\Theta_{2,2}=0$$ which trivially implies the desired estimate on $\Phi_{5}$.

To prove the claim we diagonalize the matrix $J$ to obtain a convenient form. Recall the definition of the unitary matrix $U$ in ~\eqref{eq:MatrixA}. Insert $UU^{*}=Id$ into appropriate places and use the fact $U^* JU=i\sigma_{3}$ to obtain
\begin{equation}\label{eq:DesiredEst}
\begin{array}{lll}
\Theta_{2,2}
&=&\frac{1}{2}\langle i\sigma_3[(-\Delta+V+\lambda)Id+\sigma_3 E(\lambda)]^{-1}\  P_{c}^{lin}\ \rho^2\bar\rho  U^*\left(
\begin{array}{lll}
i\\
1
\end{array}
\right), (\phi^{\lambda})^{2}\rho U^*\left(
\begin{array}{lll}
-i\\
1
\end{array}
\right)\rangle\\
&=&\langle i\sigma_3[(-\Delta+V+\lambda)Id+\sigma_3 E(\lambda)]^{-1}\  P_{c}^{lin}\ \rho^2\bar\rho  \left(
\begin{array}{lll}
0\\
1
\end{array}
\right), (\phi^{\lambda})^{2}\rho \left(
\begin{array}{lll}
-i\\
0
\end{array}
\right)\rangle\\
&=&0.
\end{array}
\end{equation} This last line follows from the observations that the column vector functions $\left(
\begin{array}{lll}
0\\
1
\end{array}
\right)$ and $\left(
\begin{array}{lll}
i\\
0
\end{array}
\right)$ are `disjoint', and the operator $\sigma_3[(-\Delta+V+\lambda)Id+\sigma_3 E(\lambda)]^{-1}$ is diagonal.

Now we choose a `difficult' term in $\Phi_4$ to study: $$D_{2,2}:=
\langle \left(
\begin{array}{lll}
A_{2,1}^{(1)}\partial_{\lambda}\phi^{\lambda}\\
A_{2,1}^{(2)}\phi^{\lambda}
\end{array}
\right),\left(
\begin{array}{lll}
-i (\phi^{\lambda})^{2}\omega\\
 (\phi^{\lambda})^{2}\rho
\end{array}
\right)\rangle.$$ Put the definitions of $A_{2,1}^{(1)}$ and $A_{2,1}^{(2)}$ in ~\eqref{eq:pkmn} into the expression to find $$
\begin{array}{lll}
D_{2,2}&=&\frac{1}{iE(\lambda)\langle \phi^{\lambda},\partial_{\lambda}\phi^{\lambda}\rangle }D_{2,2}^{(1)}-\frac{1}{2E(\lambda)\langle \phi^{\lambda},\partial_{\lambda}\phi^{\lambda}\rangle}D_{2,2}^{(2)}+\frac{A_{2,1}^{(1)}}{iE(\lambda)}\langle \phi^{\lambda},(\phi^{\lambda})^2\rho\rangle.
\end{array}
$$ with $$D_{2,2}^{(1)}:=\langle N_{2,1}^{Im},\phi^{\lambda}\rangle\langle \partial_{\lambda}\phi^{\lambda},-i(\phi^{\lambda})^2\omega\rangle-\langle N_{2,1}^{Re},\partial_{\lambda}\phi^{\lambda}\rangle \langle \phi^{\lambda},(\phi^{\lambda})^2\rho\rangle$$ and $$
\begin{array}{lll}
D_{2,2}^{(2)}&:=&\Upsilon_{1,1}[\langle \omega,\phi^{\lambda}\rangle\langle \partial_{\lambda}\phi^{\lambda},-i(\phi^{\lambda})^2 \omega\rangle+i\langle \rho,\partial_{\lambda}\phi^{\lambda}\rangle\langle \phi^{\lambda},(\phi^{\lambda})^2\rho\rangle]\\
&=&\Upsilon_{1,1}[\langle \omega-\rho,\phi^{\lambda}\rangle\langle \partial_{\lambda}\phi^{\lambda},-i(\phi^{\lambda})^2 \omega\rangle+i\langle \rho-\omega,\partial_{\lambda}\phi^{\lambda}\rangle\langle \phi^{\lambda},(\phi^{\lambda})^2\rho\rangle]
\end{array}
$$ where in the last step the facts $\phi^{\lambda}\perp \rho,\ \partial_{\lambda}\phi^{\lambda}\perp \omega$ in ~\eqref{eq:Orthogonality} are used.

Now we estimate all the three terms in the definition of $D_{2,2}$.

It is easy to see the third term is of the order $\mathcal{O}(\delta_{\infty}^3 |z|^{4})$ by the estimate of $A_{2,1}^{(1)}$ in Proposition ~\ref{Prop:Parameters}.

By the estimate $\rho-\omega=\mathcal{O}(\delta_{\infty}^2 |z|)$ it is not hard to obtain $$D_{2,2}^{(2)},\ -\frac{1}{2E(\lambda)\langle \phi^{\lambda},\partial_{\lambda}\phi^{\lambda}\rangle}D_{2,2}^{(2)}=\mathcal{O}(\delta_{\infty}^3|z|^4).$$

To estimate $D_{2,1}^{(1)}$ we use the definitions of $N_{2,1}^{Re}$ and $N_{2,1}^{Im}$ in ~\eqref{eq:ThirdOrderJN} and various estimates in Proposition ~\ref{Prop:Parameters} to find $$N_{2,1}^{Im}=-\frac{1}{4i}\rho^{2}\bar{\rho}+\mathcal{O}(\delta_{\infty}|z|^3),\ \ \ N_{2,1}^{Re}=-\frac{1}{4}\rho^{2}\bar{\rho}+\mathcal{O}(\delta_{\infty}|z|^3).$$ Put this into the expression of $D_{2,1}^{(1)}$ and use ~\eqref{eqn:perturb} and ~\eqref{eq:LambdaPhi2} on $\phi^{\lambda}$ and $\partial_{\lambda}\phi^{\lambda},$ after the cancelation of the terms of order $\delta_{\infty}^{2}|z|^4$ we obtain $$D_{2,1}^{(1)}, \ \frac{1}{iE(\lambda)\langle \phi^{\lambda},\partial_{\lambda}\phi^{\lambda}\rangle } D_{2,1}^{(1)} =\mathcal{O}(\delta_{\infty}^3|z|^4).$$

Collecting all the estimates we obtain $$D_{2,2}=\mathcal{O}(\delta_{\infty}^3 |z|^{4}).$$

The proof is complete.
\begin{flushright}
$\square$
\end{flushright}
{\bf{Proof of Lemma ~\ref{LM:Junk} for $\sigma>1$}}\\
In the case $\sigma> 1$ the strategy has to be different since some observations for $\sigma=1$, for example ~\eqref{eq:DesiredEst}, do not hold any more. Instead we use an important observation resulted from Theorem ~\ref{THM:MainTheorem2} whose second statement requires that $\xi_{n}^{lin}=\frac{x_{n}}{|x|}\xi(|x|)$, $n=1,2,\cdots, N=d,$ for some function $\xi^{lin}.$ By Lemma ~\ref{mainLem2} this implies that $$\left(
\begin{array}{lll}
\xi_{n}\\
i\eta_{n}
\end{array}
\right)=\frac{x_{n}}{|x|}\left(
\begin{array}{lll}
\xi(|x|)\\
i\eta(|x|)
\end{array}
\right)$$ for some real functions $\xi$ and $\eta$.

This makes $\Omega_{\sigma>1}=0$ by observing $\xi_{m}\eta_{n}-\xi_{n}\eta_{m}=0.$

In estimating the other terms on the right hand side of ~\eqref{eq:JunkTerms2} we only study $\Phi_{5}$, the estimation on the other terms are similar.

After some manipulation similar to that in ~\eqref{eq:RePhi5} we find
\begin{equation}\label{eq:phi5}
\Phi_{5}=iC_{3}(\sigma)D_1+iC_{4}(\sigma)D_2+\mathcal{O}(\delta_{\infty}^{4\sigma-1}|z|^4)\ \text{for some}\ C_{3},\ C_{4}\in \mathbb{R},
\end{equation} where $D_1$ and $D_2$ are defined as $$D_1:=\langle [-\Delta+V+\lambda+E(\lambda)]^{-1}P_{c}^{lin}(\phi^{\lambda})^{2\sigma-2}\rho^{2}\bar\rho,(\phi^{\lambda})^{2\sigma}\rho\rangle$$ $$D_2:=\langle [-\Delta+V+\lambda-E(\lambda)]^{-1}P_{c}^{lin}(\phi^{\lambda})^{2\sigma-2}\rho^{2}\bar\rho,(\phi^{\lambda})^{2\sigma}\rho\rangle.$$ We claim that $$D_1,\ D_2\in \mathbb{R}.$$ If the claim holds then it together with ~\eqref{eq:phi5} yields the desired estimate $$Re\Phi_5=\mathcal{O}(\delta_{\infty}^{4\sigma-1}|z|^4).$$

Now we prove the claim for $D_1$, the proof for $D_2$ is almost the same. The facts $\xi_{k}=\frac{x_{k}}{|x|}\xi(|x|)$ and the potential $V$ and $\phi^{\lambda}$ are spherically symmetric imply $$
\begin{array}{lll}
D_1&=&\displaystyle\sum_{k,l\leq N=d,\ k\not= l}z_{k}^{2}[\bar{z}_{l}]^2 D(k,l)+\sum_{k,l\leq N=d,\ k\not= l}|z_{k}|^{2}|z_{l}|^2 D(k,l)+\sum_{k\leq N}|z_{k}|^4 D(k,k)\\
&=&D(1,2)\displaystyle\sum_{k,l\leq N=d,\ k\not= l}z_{k}^{2}[\bar{z}_{l}]^2 +D(1,2)\sum_{k,l\leq N=d,\ k\not= l}|z_{k}|^{2}|z_{l}|^2 +D(1,1)\sum_{k\leq N}|z_{k}|^4
\end{array}
$$ where $D(k,l):=\langle [-\Delta+V+\lambda+E(\lambda)]^{-1}P_{c}^{lin}(\phi^{\lambda})^{2\sigma-2}\xi_k^2 \xi_{l},(\phi^{\lambda})^{2\sigma}\xi_l\rangle.$ The last line follows from the observations $D(k,l)=D(1,2)\in \mathbb{R}$ if $k\not=l$ and $D(k,k)=D(1,1)\in \mathbb{R}$ resulted from the permutation of coordinates.
By observing $\displaystyle\sum_{k,l\leq N=d,\ k\not= l}z_{k}^{2}[\bar{z}_{l}]^2 =|\sum_{k}z_{k}^2|^2-\sum_{l}|z_{l}|^4\in \mathbb{R}$ we have $D_1\in \mathbb{R}$.

The proof is complete.
\begin{flushright}
$\square$
\end{flushright}


\bibliographystyle{abbrv}
\addcontentsline{toc}{chapter}{Bibliography}
\bibliography{biblio}
\end{document}